\documentclass[a4paper,11pt]{amsart}

\usepackage{amstext,amsfonts,amsthm,indentfirst,graphicx,amssymb,amscd,epsfig}
\usepackage{amsmath}
\usepackage{graphics,wrapfig}
\usepackage[T2C,T1]{fontenc}
\usepackage[utf8]{inputenc}
\usepackage[russian,english]{babel}    
\usepackage{bm}
\usepackage{mathrsfs,dsfont}
\usepackage[colorlinks]{hyperref}
\hypersetup{colorlinks,linkcolor=red,linktoc=page}
\usepackage{enumitem,derivative}
\usepackage[bbgreekl]{mathbbol}

\DeclareSymbolFontAlphabet{\amsmathbb}{AMSb}

\newtheorem{proposition}{Proposition}[section]
\newtheorem{definition}{Definition}[section]
\newtheorem{theorem}{Theorem}[section]
\newtheorem{corollary}{Corollary}[section]
\newtheorem{remark}{Remark}[section]
\newtheorem{lemma}{Lemma}[section]
\newtheorem{mainthm}{Theorem}

\def\Co{\amsmathbb C}
\def\Na{\amsmathbb{N}}
\def\Re{\amsmathbb R}

\def\I{\mathds{I}}

\def\II{\mathds{II}}

\def\id{\mathbb{id}}
\def\A{\mathbb A}
\def\B{\mathbb B}

\def\D{\mathbb D}

\def\F{\mathbb F}
\def\G{\mathbb G}

\def\M{\mathbb{M}}

\def\U{\mathbb{U}}

\def\W{\mathbb{W}}

\def\US{\mathscr{U}}
\def\VS{\mathscr{V}}

\def\EC{\mathcal{E}}
\def\HC{\mathcal{H}}
\def\UC{\mathcal{U}}

\def\VC{\mathcal{V}}
\def\WC{\mathcal{W}}
\def\XC{\mathcal{X}}
\def\YC{\mathcal{Y}}
\def\ZC{\mathcal{Z}}

\def\XF{\mathfrak{X}}

\def\gt{\tilde{g}}
\def\dt{\tilde{\nabla}}

\newcommand{\Gat}[3]{\tilde{\Gamma}_{{#1}{#3}}^{#2}}
\newcommand{\Gah}[3]{\hat{\Gamma}_{{#1}{#3}}^{#2}}

\newcommand{\z}{\bar z}
\newcommand{\J}[1]{\mathds{J}_{#1}}
\newcommand{\dd}[2]{\frac{\partial{#1}}{\partial{#2}}}

\title[On the regularity of the structure of branched immersions]{On the regularity of the structure of branched immersions}

\author{T. A. Medina-Tejeda}

\date{\today}

\thanks{The author was supported by FAPESP Grant no. 2021/11253-3}

\address{Instituto de Ci\^encias Matem\'aticas e de Computa\c{c}\~ao - Universidade de S\~ao Paulo,
Av. Trabalhador s\~ao-carlense, 400 - Centro,
CEP: 13566-590 - S\~ao Carlos - SP, Brazil}

\email{talexanmedinat@gmail.com}

\begin{document}
\begin{abstract}
We show the existence of some special coordinate systems for expressing maps with branch points. These coordinates allow obtaining an explicit representation formula of branch immersions and understanding the regularity of some fundamental elements in the theory of R. D. Gulliver, R. Osserman and H. L. Royden. For conformal maps, we prove that these coordinates exist with certain regularity if and only if the mean curvature vector extends and is in a particular Hölder space. In general, we characterize these types of coordinate systems. Also, using these coordinates, we study the regularity and behavior of curvatures near branch points. 
\end{abstract}

\maketitle

\tableofcontents

\section{Introduction.}
This paper is devoted to determining the regularity of different elements associated with branched immersions. These elements include representation formulas, the tangent bundle, the mean curvature vector, the curvature tensor, the sectional curvature, and the Gauss-Kronecker curvature, among others. We show how the regularity and behavior of these elements near a branch point are determined by some special coordinate systems, the order, and the index of the branch point. 

Branched immersions were introduced by R. D. Gulliver, R. Osserman, and H. L. Royden in \cite{O1} to understand the nature of the solutions of some variational problems in geometry. One of the most popular variational problems is the Plateau problem. Although R. Osserman in \cite{O2} proved that the solution to this problem in $\Re^3$ is free of branch points, there are minimizing solutions of the problem in $\Re^n$, $n\geq 4$ having branch points (cf.\cite{WB}). This fact demonstrates the importance of studying maps with the presence of branch points. 

If $M^2$, $N^n$ are a bi-dimensional and an $n$-dimensional smooth manifolds respectively with $n\geq 2$, and $f:M^2\to N^n$ is a $C^1$-map, a point $p\in M^2$ is said to be a {\it branch point} of $f$ if there exist $s\in \amsmathbb{N}$ and coordinate systems $(U,x)$, $(V,y)$ at $p$ and $f(p)$ respectively, such that $x(p)=0$, $y(f(p))=0$ and $\hat{f}=y\circ f\circ x^{-1}=(\hat{f_1},\dots,\hat{f_n})$ satisfies:
\begin{align}\label{branchcond1}
	 \sigma(z):=\hat{f}_1(z)+i\hat{f}_2(z)-z^{s+1}=o(|z|^{s+1}),  
\end{align}
\begin{align}\label{branchcond2}
	\hat{f}_k(z)=o(|z|^{s+1}),\text{ } 3\leq k\leq n,
\end{align}
\begin{align}\label{branchcond3}
	\frac{\partial\sigma}{\partial x_j}(z),\frac{\partial\hat{f}_k}{\partial x_j}(z)=o(|z|^{s}),\text{ } 3\leq k\leq n,\text{ }j=1,2.
\end{align}
The point $p$ is also called a {\it branch point of order $s$}. A map $f:M^2\to N^n $ is a {\it branched immersion} if it is a $C^1$-map, regular, except for branch points. In the case that $f$ is a {\it branched conformal immersion} (i.e., a branched immersion which is also a conformal map) with extendable mean curvature vector $\HC$ of certain local Hölder regularity $C^{r,\beta}_{loc}$, the expression above of $f$ in the coordinate systems $(U,x)$ and $(V,y)$ hides an expression involving maps with certain regularity, which we will describe below in our first main result:
\begin{mainthm}[\ref{Theo1}]\label{TheoA}
	Let $f:M^2\to N^n$ be a $C^{r+2,\beta}_{loc}$-branched conformal immersion, $r\in \Na\cup \{0,\infty,\omega\}$, $\beta \in(0,1)$ and $\tilde{g}$ a $C^{r+1,\beta}_{loc}$-Riemannian metric on $N^n$. If the mean curvature vector $\HC$ of $f$ has a $C^{r,\beta}_{loc}$-extension to $M^2$, then for any branch point $p\in M^2$ of order $s$ and coordinate systems $(U,x)$, $(V,y)$ at $p$, $f(p)$ respectively, such that $x(p)=0$, $y(f(p))=0$ and $(\ref{branchcond3})$ is satisfied, we have
	\begin{align}
		&\dd{\hat{f}}{z}(z)=(s+1)z^sd(z),\label{cond1intro}\\
		&\dd{^2\hat{f}}{\z\partial z}(z)=|z|^{2s}l(z),\label{cond2intro}
	\end{align} 
	where $l:x(U)\to \Re^n$ is a $C^{r,\beta}_{loc}$-map and $d:x(U)\to \Co^n$ is a $C^{r+1,\beta}_{loc}$-map in which
	\begin{align}\label{cond3intro}
		d_1(z)=\frac{1}{2}(1+d'_1(z)),\text{ }d_2(z)=-\frac{i}{2}(1+d'_2(z)),
	\end{align}	 
	$d'_1(0)=0$, $d'_2(0)=0$ and $d_j(0)=0$ for $j=3,\dots,n$. 
\end{mainthm}
Here $\dd{}{z}$, $\dd{}{\z}$ are the differential operators $\frac{1}{2}(\dd{}{x_1}-i\dd{}{x_2})$ and $\frac{1}{2}(\dd{}{x_1}+i\dd{}{x_2})$ respectively. The most important aspect of theorem \ref{TheoA} is the regularity of $d$ and $l$, which influence the regularity of other elements associated with $f$. For example, the converse of theorem \ref{TheoA} is also true (see proposition \ref{theo1reciprocal}) in the following way. If for each branch point $p\in M^2$ there exist coordinate systems $(U,x)$, $(V,y)$ at $p$, $f(p)$ respectively, such that the conditions (\ref{cond1intro}), (\ref{cond2intro}) and (\ref{cond3intro}) are satisfied for some $s\in \Na$, then the mean curvature vector $\HC$ of $f$ has a $C^{r,\beta}_{loc}$-extension to $M^2$. So, the existence of such coordinate systems characterizes the extendibility and regularity of the mean curvature vector of branched conformal immersions. In section \ref{sectiontangentbundle}, we show that the regularity of $d$ in (\ref{cond1intro}) determines the regularity of a vector subbundle $T_f$ of the pullback bundle $f^*TN^n$ near branch points. The bundle $T_f$ is called the tangent bundle of $f$, and its fibers $T_{fp}$ contain the spaces $f_*T_pM^2$. These notable facts motivate us to give the following definition. 

If $f:M^2\to N^n$ is a $C^1$-map and $p\in M^2$, we will call {\it $C^{k,\beta}_{loc}$-branch coordinates at $p$} a pair of coordinate systems $(U,x)$, $(V,y)$ as in theorem \ref{TheoA} satisfying conditions (\ref{cond1intro}) and (\ref{cond3intro}) with $d$ being a $C^{k,\beta}_{loc}$-map. Strictly speaking, the definition requires two more technical conditions affordable if we shrink $U$ (see definition \ref{branchcoord} for details). Also, if additionally the condition (\ref{cond2intro}) is satisfied with $l$ being a $C^{k-1,\beta}_{loc}$-map, $(U,x)$, $(V,y)$ are called {\it $C^{k,\beta}_{loc}$-regular branch coordinates at $p$}. It is not hard to prove (see lemma \ref{lembranch}) that any pair of coordinates systems in which $\hat{f}$ satisfies (\ref{branchcond1}), (\ref{branchcond2}) and (\ref{branchcond3}) are $C^0$-branch coordinates. The converse is also valid (see lemma \ref{equivalencebranchcoor}), that is, every pair of $C^0$-branch coordinates satisfies (\ref{branchcond1}), (\ref{branchcond2}) and (\ref{branchcond3}). The benefit of these coordinates lies in the differentiability of $d$ and the Hölder continuity of $l$, as we will see later. So we will focus on $C^{k,\beta}_{loc}$-branch coordinates with $k\geq 1$. 

As we are interested in maps with branch points in general, it is natural to wonder about the existence of these types of coordinates for non-conformal maps and how can be estimated their associated degree of differentiability. If we look carefully, the existence of $C^{k}$-branch coordinates at a point $p$ for a $C^{k+1}$-map with $k\geq s$, implies that $\dd{^{s+1}\hat{f}}{z^{s+1}}(0)=(\frac{(s+1)!}{2},-i\frac{(s+1)!}{2},0,\dots,0)$ and for each $h'\in\{0,\dots,s-1\}$, $\dd{^{h+h'+1}\hat{f}}{\z^h\partial z^{h'+1}}(0)=0$ for all $h\in\{0,\dots,k-h'\}$. This type of condition is sufficient to determine if a pair of coordinates $(U,x)$, $(V,y)$ are branch coordinates, as indicated in our second main result:
\begin{mainthm}[\ref{Theo2}]\label{TheoB}
	Let $f:M^2\to N^n$ be a $C^{r+1,\beta}_{loc}$-map, $r\in\Na$ (resp. $r\in \{\infty,\omega\}$), $\beta \in [0,1]$, $s\in \Na$ such that $s\leq r$ and $p\in M^2$. If $(U,x)$, $(V,y)$ are coordinate systems at $p$, $f(p)$ respectively, such that $x(p)=0$, $y(f(p))=0$, $$\dd{^{s+1}\hat{f}}{z^{s+1}}(0)=(\frac{(s+1)!}{2},-i\frac{(s+1)!}{2},0,\dots,0)$$ and for each $h'\in\{0,\dots,s-1\}$, we have that $\dd{^{h+h'+1}\hat{f}}{\z^h\partial z^{h'+1}}(0)=0$ for all $h\in\{0,\dots,r-h'\}$ (resp. for all $h\in \{0\}\cup \Na$), then $p$ is a branch point of order $s$ and there exists an open neighborhood $U'\subset U$ of $p$ such that $(U',x)$, $(V,y)$ are $C^{r-s,\beta}_{loc}$-branch coordinates at $p$. If additionally, $r\geq (s+1)(s+2)/2$ then $U'$ can be shrunk in such way that $(U',x)$, $(V,y)$ are $C^{r-(s+1)(s+2)/2+1,\beta}_{loc}$-regular branch coordinates at $p$. 
\end{mainthm} 
The theorem \ref{TheoB} implies that every $C^{k,\beta}$-branch coordinates with $k$ sufficiently larger than $s$ are regular branch coordinates with a little less regularity. In the cases $k=\infty$ or $k=\omega$, there is no loss of regularity. A fundamental question that arises naturally is, what is the explicit expression of a map in these coordinates? We answer this question in the following result:
\begin{mainthm}[\ref{representbranchcoord}]\label{TheoC}
	Let $f:M^2\to N^n$ be a $C^{r+1,\beta}_{loc}$-map, $r\in \Na\cup \{\infty,\omega\}$, $\beta\in (0,1)$ and $p\in M^2$ a branch point of order $s$. If $(U,x)$, $(V,y)$ are $C^{r,\beta}_{loc}$-regular branch coordinates at $p$, then  
	\begin{align}
		\hat{f}_1(z)=&Re\{z^{s+1}\}+\sum_{j=0}^{s}\sum_{k=0}^{s}(-1)^{2s-j-k}\frac{(s!)^2}{j!k!}\frac{\partial^{j+k}\varphi_1}{\partial \z^{j}\partial z^{k}}(z)\z^{j}z^k\label{formulap1intro}\\
		&+2Re\{\int_{0}^{z}w^sF_1(w)dw\},\nonumber\\
		\hat{f}_2(z)=&Im\{z^{s+1}\}+\sum_{j=0}^{s}\sum_{k=0}^{s}(-1)^{2s-j-k}\frac{(s!)^2}{j!k!}\frac{\partial^{j+k}\varphi_2}{\partial \z^{j}\partial z^{k}}(z)\z^{j}z^k\label{formulap2intro}\\
		&+2Re\{\int_{0}^{z}w^sF_2(w)dw\},\nonumber\\
		\hat{f}_h(z)=&\sum_{j=0}^{s}\sum_{k=0}^{s}(-1)^{2s-j-k}\frac{(s!)^2}{j!k!}\frac{\partial^{j+k}\varphi_h}{\partial \z^{j}\partial z^{k}}(z)\z^{j}z^k+2Re\{\int_{0}^{z}w^sF_h(w)dw\},\label{formulap3intro}\\
		&h\in\{3,\dots,n\},\nonumber
	\end{align}
	for some real $C^{2s+r+1,\beta}_{loc}$-functions $\varphi_h$ and holomorphic functions $F_h$ on an open connected neighborhood $W\subset x(U)$ of $0$, satisfying 
	\begin{align}\label{conditionrepresentintro}
		(-1)^{s}s!\frac{\partial^{s+1}\varphi_h}{\partial z^{s+1}}(0)+F_h(0)=\varphi_h(0)=0
	\end{align} for each $h\in\{1,\dots,n\}$. Furthermore, for any $(m,\gamma)\in (\Na\cup \{0,\infty,\omega\})\times [0,1]$ and $s\in \Na$, if we substitute in the formulas (\ref{formulap1intro}), (\ref{formulap2intro}) and (\ref{formulap3intro}), real $C^{2s+m+1,\gamma}_{loc}$-functions $\varphi_h$ on an open connected neighborhood $\hat{U}\subset\Re^2$ of $0$ and holomorphic functions $F_h$ on $\hat{U}$ satisfying (\ref{conditionrepresentintro}), we obtain a $C^{m+1,\gamma}_{loc}$-map $\hat{f}:\hat{U}\to \Re^n$ with $0$ being a branch point of order $s$. Also, shrinking $\hat{U}$ if necessary, the standard coordinates of $\hat{U}$ and $\Re^n$ are $C^{m,\gamma}_{loc}$-regular branch coordinates at $0$ for $\hat{f}$.  
\end{mainthm}
Although the theorem \ref{TheoC} gives us explicitly the expression of branched immersions locally, this expression is not as simplified as one would like. However, there exists a procedure that can simplify the expression but with a loss of regularity. Given a map $\hat{f}$ satisfying (\ref{branchcond1}), (\ref{branchcond2}) and (\ref{branchcond3}), R. D. Gulliver, R. Osserman, and H. L. Royden proved in \cite{O1} the existence of a $C^1$-diffeomorphism $c$ on a neighborhood of $0$ such that $(\hat{f}_1\circ c)(w)+i(\hat{f}_2\circ c)(w)=w^{s+1}$, $(\hat{f}_k\circ c)(w)=o(w^{s+1})$ and $\dd{(\hat{f}\circ c)_k}{u_j}(w)=o(w^s)$ for all $3\leq k\leq n$, $j=1,2$. Initially, R. D. Gulliver showed this particular form before in \cite{Gu} for the case $N^n=\Re^3$; its proof is also valid for the general case. Despite $c$ is of class $C^1$, the particular form of $\hat{f}\circ c$ compensated for this inconvenience in the work developed by R. D. Gulliver, R. Osserman, and H. L. Royden. Several results in \cite{O1}, like the fundamental factorization theorem (cf. proposition 3.17 in \cite{O1}), rely heavily on this special form. Also, M. J. Micallef and B. White in \cite{M-W} proved the existence of a $C^2$-diffeormorphism $c$ with the same property as before for a generalized minimal immersion $f:B\to N^n$ of class $C^2$, where $B\subset \Co$ is the closed unit disk and $\hat{f}$ is the expression of $f$ respect to the standard coordinates of $B$ and a normal coordinate system $(V,y)$ at $f(0)$. They used the local form obtained through $c$ to deduce some consequences about the structure and topology of area-minimizing surfaces near branch points.

Due to the benefit of this local form in the study of branched immersions, we characterized the regularity of a diffeomorphism like $c$. Maps with the property of $c$ are related to a certain coefficient defined in terms of the maps $d_1'$, $d_2'$ in (\ref{cond3intro}). If $(U,x)$, $(V,y)$ are $C^{r,\beta}_{loc}$-branch coordinates, we define the {\it distinguished coefficient of $\hat{f}$} to be the map $\varpi(z)=\z^s\frac{1}{2}(\overline{d'_1(z)-d'_2(z)})(z^s(1+\frac{1}{2}(d'_1(z)+d'_2(z))))^{-1}$. This map extends at $0$; in fact, the strict definition of branch coordinates (see definition \ref{branchcoord}) implies that $\varpi$ is a $C^{0,\beta}_{loc}$-map on $x(U)$ (see lemma \ref{beltrami}). The main properties of $\varpi$ are shown in our fourth main result: 

\begin{mainthm}[\ref{Main3}]\label{TheoD}
	Let $f:M^2\to N^n$ be a $C^{r+1,\beta}_{loc}$-map, $(r,m)\in (\Na\cup \{\infty,\omega\})\times (\Na\cup \{\infty,\omega\})$ with $r\geq m$, $p\in M^2$ a branch point of order $s$ and $\beta\in [0,1]$. If $(U,x)$ and $(V,y)$ are $C^{r,\beta}_{loc}$-regular branch coordinates at $p$ with $x(U)$ an open convex set, the following assertions hold: 
	\begin{enumerate}[label=(\roman*)]
		\item The distinguished coefficient of $\hat{f}$ is a $C^{m,\beta}_{loc}$-map on $U$ if and only if there exist open neighborhoods $U'\subset \Re^2$, $U''\subset x(U)$ of $0$ and a $C^{m+1,\beta}_{loc}$-diffeomorphism $c:U'\to U''$ with $c(0)=0$, such that $(\hat{f}_1\circ c)(w)+i(\hat{f}_2\circ c)(w)=w^{s+1}$. 
		\item Every $C^{m+1,\beta}_{loc}$-diffeomorphism $c:U'\to U''$ as in the item $(i)$ is a quasiconformal map satisfiying $\dd{(c^{-1})}{\z}(z)=\varpi(z)\dd{(c^{-1})}{z}(z)$ and has the form $c(w)=wc_0(w)$, where $c_0$ is a $C^{m,\beta}_{loc}$-map with $c_0(0)\neq 0$.
		\item If $c:U'\to U''$ is a $C^{m+1,\beta}_{loc}$-diffeomorphism as in the item $(i)$ then $(U',id_{U'})$, $(\Re^n,id_{\Re^n})$ are $C^{min\{r,m+1\},\beta}_{loc}$-branch coordinates at $0$ for $\hat{f}\circ c$. More precisely, the expression of $\hat{f}\circ c$ is given by $(\hat{f}_1\circ c)(w)+i(\hat{f}_2\circ c)(w)=w^{s+1}$ and $(\hat{f}_h\circ c)(w)=(s+1)Re\{\int_{0}^{w}\overline{b}_h(z)z^{s}dz\}, h\in\{3,\dots,n\}$, for some complex $C^{min\{r,m+1\},\beta}_{loc}$-maps $b_h$ on $U'$  satisfying $b_h(0)=0$ for each $h\in\{3,\dots,n\}$. 
		\item If $m\in \Na$ (resp. $m\in\{\infty, \omega\}$), the distinguished coefficient $\varpi$ of $\hat{f}$ is a $C^{m,\beta}_{loc}$-map on $x(U)$ if and only if the maps $d'_1$ and $d'_2$ in (\ref{cond3}) satisfy $\dd{^j(\overline{d'_1-d'_2})}{\z^j}(0)=0$ for all $j\in\{0,\dots,m\}$ (resp. for all $j\in\Na$). Also, if $\varpi$ is a $C^{m,\beta}_{loc}$-map on $U$, then $\dd{^j\varpi}{\z^j}(0)=0$ for all $j\in\{0,\dots,m\}$ (resp. for all $j\in\Na\cup \{0\}$).
		\item If $f$ is a $C^{r+1,\beta}_{loc}$-branched conformal immersion and $(V,y)$ is a normal coordinate system at $f(p)$, the distinguished coefficient of $\hat{f}$ is a $C^{1,\beta}_{loc}$-map. 
	\end{enumerate}
\end{mainthm}
In theorem \ref{TheoD}, if $m=0$, $r\geq 0$ and $(U,x)$, $(V,y)$ are $C^{r,\beta}_{loc}$-branch coordinates at $p$, the items $(i)$, $(ii)$, $(iii)$ and $(iv)$ are true. If $f$ is a branched conformal map, note that items $(i)$, $(iii)$ and $(v)$ imply the existence of a $C^{2,\beta}_{loc}$-diffeomorphism $c$ like the one found by M. J. Micallef and B. White in \cite{M-W}. 
Theorem \ref{TheoD} tells us that the optimal regularity that one could expect from a $C^1$-diffeomorphism $c$ with the property mentioned in the item $(i)$ is determined by $d'_1-d'_2$. This optimal regularity could be influenced by geometrical properties of $f$ and $y$, as we will see later. 

Using $C^{\infty}$-branch coordinates $(U,x)$, $(V,y)$ at $p$, we define in section \ref{sectionrepresentation} the {\it distinguished degree of $p$ respect to $(U,x)$, $(V,y)$} to be $\varrho\in\Na\cup\{\infty\}$ if and only if the distinguished coefficient of $\hat{f}$ is of class $C^{\varrho-1}$ but not $C^{\varrho}$ on every neighborhood of $0$. To estimate $\varrho$, we extend the notion of {\it index} of a branch point studied in minimal surfaces (see, for example, \cite{Di} or \cite{Trom}) to be an element $\iota\in\Na\cup\{\infty\}$ associated to the pair of $C^{\infty}$-branch coordinates $(U,x)$, $(V,y)$ at $p$ (see definition \ref{defindex}). The estimation of $\varrho$ is possible because the inequality $\varrho\geq 2(\iota-s)$ holds when $f$ is a branched conformal immersion and $y$ is a conformal diffeomorphism (see proposition \ref{estimation}). Additionally, if $n=3$ we have $\varrho= 2(\iota-s)$. On the other hand, S. Hildebrandt and A. Tromba in \cite{Hild} gave estimates for the index using the behavior of the boundary contour of a minimal surface. We will give a different estimate based on the behavior of curvatures. In our latest main result below, we show that the index characterizes the behavior of curvatures near a branch point, so we can use this behavior to estimate the index.
\begin{mainthm}[\ref{Theofinal}]\label{TheoE}
	Let $f:M^2\to N^n$ be a $C^{\infty}$-map, $(N^n,\gt)$ a smooth Riemannian manifold, $p\in M^2$ a branch point of order $s$ and $(U,x)$, $(V,y)$ $C^{\infty}$-regular branch coordinates at $p$. If $y$ is a conformal diffeomorphism and $\iota$ the index of $p$, the following statements hold:
	\begin{enumerate}[label=(\roman*)]
		\item If $\iota\geq2s+1$, then the sectional curvature $Sec$ of the induced metric on $U-\{p\}$, the mean curvature vector $\HC\vert_{U-\{p\}}$ and the Gauss-Kronecker curvature $K^\xi\vert_{U-\{p\}}$ have $C^{0,1}_{loc}$-extensions on $U$, where $\xi$ is any local smooth section of the normal bundle of $f$ over $U$.  
		\item If $\iota<2s+1$, then $\lim_{q\to p}Sec(q)=\infty$ and $\lim_{q\to p}\sum_{j=1}^{n-2}K^{\xi'_j}=\infty$ for any smooth orthonormal frame field $(\xi'_1,\dots,\xi'_{n-2})$ for the normal bundle of $f$ over $U$.    
	\end{enumerate}
\end{mainthm}
Here, the normal bundle $\bot_f$ of $f$ is a vector subbundle of $f^*TN^n$, the orthogonal complement of the tangent bundle $T_{f}$. Now, observe that by theorem \ref{TheoA} and remark \ref{remGulliver}, if $(N^n,\gt)$ is a conformally flat manifold, $f$ a $C^{\infty}$-branched conformal immersion and $\HC$ has a $C^{\infty}$-extension on $M^2$, then the hypotheses of theorem \ref{TheoE} are affordable. If additionally the sectional curvature is bounded near to $p$, then by theorem \ref{TheoE} $\iota\geq2s+1$ and as a consequence the distinguished degree of $p$ satisfies $\varrho\geq 2(s+1)$. Theorem \ref{TheoE} also determines the behavior of some fundamental tensors. If $\XC$, $\YC$ and $\ZC$ are vector fields defined on a neighborhood of the branch point $p$, it is known that away from the branch points, the curvature tensor, the Ricci tensor and the scalar curvature of the induced metric $g$ on $M^2$ are given by $R(\XC,\YC)\ZC=(g(\YC,\ZC)\XC-g(\XC,\ZC)\YC)Sec$, $Ric(\XC,\YC)=g(\XC,\YC)Sec$ and $Scal=2Sec$ respectively. Therefore, under the hypotheses of theorem \ref{TheoE}, it holds that $R$, $Ric$ and $Scal$ have $C^{0,1}_{loc}$-extensions at $p$ if $\iota\geq2s+1$. 

We will now give a brief description of the proofs of the main results. To prove theorem \ref{TheoA}, we first show that $C^{0}$-branch coordinates exist at a branch point with $d$ being differentiable on a punctured neighborhood. Then, using the $\HC$-surface equation, we show that $d$ satisfies a non-linear inhomogeneous Cauchy-Riemann system of the type $\dd{d}{\z}(z)=E(z,d(z))$, where $E$ is of class $C^r$. A. Coffman, Y. Pan, and Y. Zhang established elliptic regularity for an equation of this type in \cite{Coff}. The corollary 2.7 in \cite{Coff} determines the regularity of continuous solutions. Using the corollary, we can extend it for a system of equations (see lemma \ref{reg}), so we imply that $d$ is a $C^{r,\beta}_{loc}$-map. By the explicit form of $E$, it is deduced that $E(z,d(z))$ is a $C^{r,\beta}_{loc}$-map and using Schauder estimates, we obtain that $d$ is a $C^{r+1,\beta}_{loc}$-map. The regularity of $l$ follows from algebraic manipulations in the equations.

To prove theorem \ref{TheoB}, first, we study the local form of complex maps $e$ in Hölder spaces satisfying $Im\{z^se(z)\}=0$ near the origin. We show this condition implies that $e(z)=\z^s\ell(z)$ for a real function $\ell$ with a specific regularity. Also, given a complex map $e'$ in a Hölder space, we establish the extendibility and Hölder regularity of complex maps with the forms $\frac{\z^s}{z^s}e'(z)$ and $\frac{e'(z)}{z^s}$ in terms of the higher order partial derivatives of $e'$ at $0$. We obtain the lemmas about this part using the integral representation $e'(z)=z\int_{0}^{1}\dd{e'}{z}(tz)dt+\z\int_{0}^{1}\dd{e'}{\z}(tz)dt$ and mathematical induction. Applying these lemmas in the context of theorem \ref{TheoB}, we show that $\dd{\hat{f}}{z}(z)/z^s$ extends with a specific Hölder regularity. Denoting this extension by $d$ and since $Im\{z^s\dd{d}{\z}\}=0$, then $\dd{^2\hat{f}}{\z\partial z}(z)=|z|^{2s}l(z)$ for some map $l$ with real components and certain regularity. Thus, we get the conditions required for branch coordinates with the regularity stated in theorem \ref{TheoB}.

As the distinguished coefficient $\varpi$ has the form $\frac{\z^s}{z^s}e(z)$, we use the lemmas obtained before to prove item $(iv)$ of theorem \ref{TheoD}. Item $(iv)$ implies $(v)$ by direct computations. To construct the diffeomorphism $c$ of item $(i)$ and determine its regularity, using the existence theorem for solutions of Beltrami equations and Schauder estimates for Beltrami operators is tentative. However, this way fails when $\beta=1$ or $\beta=0$. So, we construct the diffeomorphism $c^{-1}$ explicitly using $\varpi$, and we show directly that $c^{-1}$ satisfies a Beltrami equation with Beltrami coefficient equal to $\varpi$, which implies that $c$ is a quasiconformal map. For the rest, we use a result from the theory of quasiconformal mappings called the Stoilow factorization theorem (see \cite{As} for details) and some results obtained in section \ref{sectiontangentbundle}. 

To prove theorem \ref{TheoC}, we first show that solutions of equations like $\dd{b}{\z}(z)=\ell(z)\z^s$ have the local form $b(z)=\sum_{j=0}^{s}(-1)^{s-j}\frac{s!}{j!}\frac{\partial^{j+s+1}\varphi}{\partial \z^{j}\partial z^{s+1}}(z)\z^{j}+F(z)$ and conversely, for some real function $\varphi$ and $F(z)$ a holomorphic function. This is done observing that the local form satisfies $\dd{b}{\z}=\frac{1}{4^{s+1}}\Delta^{s+1}\varphi(z)\z^s$, where $\Delta^{s+1}$ is the Laplace operator applied $s+1$ times. Then, considering the equation $\Delta^{s+1}\varphi(z)=4^{s+1}\ell(z)$ we can obtain solutions $\varphi$ solving recursively $s+1$ Poisson's equations. Now, from conditions (\ref{cond1intro}) and (\ref{cond2intro}), we get the equation $\dd{d}{\z}(z)=\z^sl(z)$, which is an equation of the type discussed before. So, we can obtain the explicit form of $d$ and then verify that the derivative $\dd{}{z}$ of the expression given in theorem \ref{TheoC} is equal to the explicit expression of $(s+1)z^sd$. 

To prove theorem \ref{TheoE}, we construct a mathematical machinery using the tangent bundle and the normal bundle of $f$ that makes it possible to understand the behavior of fundamental tensors and curvatures commonly used in geometry. This machinery is accessible to a general reader with a modest submanifold theory and Riemannian geometry background. We use this in combination with the results of sections \ref{sectiontangentbundle} and \ref{sectionrepresentation} to prove theorem \ref{TheoE}. 

The paper is organized as follows. In section \ref{section-definition}, we establish the notation and conventions we use in most of the paper. In section \ref{sectionregularity}, we prove theorem \ref{TheoA} and characterize the extendibility and regularity of the mean curvature vector of branched conformal immersions. In section \ref{sectioncharacterizationbranch}, we prove theorem \ref{TheoB}. The lemmas in this section also will be essential for the remaining sections. In section \ref{sectiontangentbundle}, we first treat a type of map $f$ more general than branched immersions with a well-defined tangent bundle. Then, we treat maps with branch points in the context of the terminology introduced in this section. We show how to construct the tangent bundle $T_f$ of a branched immersion starting from branch coordinates, which also determine the regularity of this bundle. J. Eschenburg and R. Tribuzy in \cite{Eschenburg1988} also constructed this bundle via the Gauss map and proved the smoothness of $T_f$ for generalized minimal immersions and conformal maps of prescribed mean curvature in $3$-dimensional manifolds. In section \ref{sectionrepresentation}, we prove theorem \ref{TheoC} and theorem \ref{TheoD}. Also, we show the properties of the distinguished degree and the index of a branch point using $C^{\infty}$-branch coordinates. Finally, in section \ref{sectioncurvatures}, we prove theorem \ref{TheoE}. 
\section{Notation, conventions and basic notions.}\label{section-definition}
In this paper, we write $M^m$ to denote an $m$-dimensional smooth or real analytic manifold, $N^n$ an $n$-dimensional smooth or real analytic manifold, and $f: M^m\to N^n$ a $C^{r,\beta}_{loc}$-map. The term $C^{r,\beta}_{loc}$ will be used to mean some local Hölder regularity, i.e., for coordinate systems $(U,x)$, $(V,y)$ of $M^m$, $N^n$ respectively, such that $f(U)\subset V$, the map $y\circ f \circ x^{-1}$ belongs to the Hölder space $C^{r,\beta}_{loc}(x(U),\Re^n)$, where $r\in \Na\cup \{0,\infty,\omega\}$ and $\beta \in[0,1]$. Here $C^0=C^{0,\beta}_{loc}$ means continuous, $C^\infty=C^{\infty,\beta}_{loc}$ smooth, $C^\omega=C^{\omega,\beta}_{loc}$ real analytic and $C^k=C^{k,0}_{loc}$ $k$-differentiable for $k\in \Na$. If $f:M^m\to N^n$ is a real analytic map, we will always assume that $M^m$ and $N^n$ are real analytic manifolds. Otherwise, these will be smooth manifolds. In the case that $r=\omega$ and $s\in \Na\cup\{0\}$, the symbols $C^{r-s}_{loc}$, $C^{r+s}_{loc}$ mean also real analytic; similarly with $r=\infty$. We endow the set $\{0,\infty,\omega\}\cup \Na$ with the well-order $\leq:=\leq'\cup (\Na\cup\{0\}\times\{\infty,\omega\})\cup (\infty,\omega)$, where $\leq'$ is the usual well-order of $\Na\cup\{0\}$ and we write $x < y$ to mean that $x \leq y$ but $x\neq y$. Thus, we have $k<\infty<\omega$ for all $k\in\Na\cup\{0\}$. In this way, expressions involving inequalities with degrees of differentiability like the statement of theorem \ref{TheoD} for example, make sense.

For a point $p\in M^m$, whenever we write ``$(U,x)$, $(V,y)$ are coordinate systems at $p$, $f(p)$ respectively'', we will assume that $f(U)\subset V$ and usually write $\hat{f}$ instead of $y\circ f\circ x^{-1}$ if there is no risk of confusion. Suppose $E$ is a vector bundle over an arbitrary $s$-dimensional smooth or real analytic manifold $S^s$. We denote the sections of $E$ by $\Gamma(E)$ and the $k$th Cartesian power of $\Gamma(E)$ by $\Gamma(E)^k$. In particular, $\Gamma(TS^s)$ will be denoted by $\XF(S^s)$ and $\Gamma(TS^s\rvert_U)$ by $\XF(U)$ for any open set $U\subset S^s$. If $f:M^m\to S^s$ is a $C^1$-map, we denote the pullback bundle by $f^*E$ and the sections of $E$ along $f$ by $\Gamma_f(E)$. We shall identify $\Gamma(f^*E)$ with $\Gamma_f(E)$ via the natural bijection. If $(U,x)$ is a coordinate system of $M^m$, we write  $\dd{}{x}=(\dd{}{x_1},\dots,\dd{}{x_m})\in \XF(U)^m$ and $f_*\frac{\partial}{\partial x}=(f_*\frac{\partial}{\partial x_1},\dots,f_*\frac{\partial}{\partial x_m})\in \Gamma(f^*TS^s\vert_U)^m$. If $\gt$ is a Riemannian metric on $S^s$, this induces a Riemannian metric on the vector bundle $f^*TS^s$, which will also be denoted by $\gt$. 

Elements of $\Re^n$ are identified with vector columns in $\mathcal{M}_{n\times 1}(\Re)$. We denote the identity matrix in $\mathcal{M}_{n\times n}(\Re)$ by $\id_{n\times n}$ and the identity map on any set $U$ by $id_U:U\to U$. If $A\in \mathcal{M}_{n\times n}(\Re)$, we write $tr(A)$ the trace of $A$, $det(A)$ the determinant of $A$, $A^t$ the transpose of $A$ and $adj(A)$ the adjugate of $A$ (i.e., $Aadj(A)=adj(A)A=det(A)\id_{n\times n}$). Now, we give some definitions that will be used in later sections. 
\begin{definition}
	\normalfont Let $M^2$ be a Riemann surface and $(N^n,\tilde{g})$ a Riemannian manifold with $n\geq 2$. A $C^1$-map $f:M^2\to N^n$ is called a {\it conformal map} if for each $p\in M^2$ and coordinate systems $(U,x)$, $(V,y)$ at $p$ and $f(p)$ respectively, it is fulfilled that 
	\begin{align}
		\sum_{j=1}^{n}\sum_{k=1}^{n}\gt_{jk}\circ f\circ x^{-1}\frac{\partial\hat{f}_j}{\partial x_1}\frac{\partial\hat{f}_k}{\partial x_1}&=\sum_{j=1}^{n}\sum_{k=1}^{n}\gt_{jk}\circ f\circ x^{-1}\frac{\partial\hat{f}_j}{\partial x_2}\frac{\partial\hat{f}_k}{\partial x_2}\text{ and}\label{conformaleq1}\\
		\sum_{j=1}^{n}\sum_{k=1}^{n}\gt_{jk}\circ f\circ x^{-1}\frac{\partial\hat{f}_j}{\partial x_1}\frac{\partial\hat{f}_k}{\partial x_2}&=0,\label{conformaleq2}
	\end{align}
	where $\gt_{jk}$ are the components of the Riemannian metric on $N^n$ in the coordinates $(V,y)$.
\end{definition}
\begin{definition}\normalfont
	Let $f:M^m\to N^n$ be a $C^1$-map with $n\geq m$. The set $\Sigma(f):=\{p\in M^m: f_*(p) \text{ is not injective}\}$ is called the {\it singular set} of $f$ and $Reg(f):=M^m-\Sigma(f)$ the {\it regular set} of $f$. A point of $\Sigma(f)$ will be called a {\it singularity} or {\it singular point} of $f$, and a point of $Reg(f)$ will be called {\it regular point}.  
\end{definition}

From now on, when we write the expression $f:M^m\to N^n$, we will assume that $n\geq m$. Let $f:M^m\to N^n$ be a $C^{r+2,\beta}_{loc}$-map, $r\in \Na\cup \{0,\infty,\omega\}$, $\beta\in[0,1]$ and $\gt$ a $C^{r+1,\beta}_{loc}$-Riemannian metric on $N^n$. We will denote the metric induced by $f$ on $M^m$ by $g$, a degenerate metric on $\Sigma(f)$. The Levi-Civita connection $\tilde{\nabla}$ of $N^n$  induces a connection on $f^*TN^n$ denoted with the same symbol $\dt$. Since $Reg(f)$ is an open set of $M^m$, this is an $m$-dimensional $C^{\infty}$-submanifold (possibly non-connected) of $M^m$ if $Reg(f)\neq \emptyset$. If $p\in Reg(f)$, the orthogonal complement of $f_*T_pM^m$ in $T_{f(p)}N^n$ is denoted by $\bot_p Reg(f)$. The normal bundle $\bot Reg(f)$ of $f$ is the vector subbundle of $f^*TN^n\vert_{Reg(f)}$ whose fiber at $p$ is $\bot_p Reg(f)$. We have the decomposition $f^*TN^n\vert_{Reg(f)}=f_*(TM^m\vert_{Reg(f)})\bigoplus \bot Reg(f)$, so for all $\XC,\YC\in \XF(Reg(f))$, we can express $\dt_{\XC}f_*\YC=(\dt_{\XC}f_*\YC)^T+(\dt_{\XC}f_*\YC)^\bot,$ with respect to the last decomposition. The second fundamental form $\alpha:\XF(Reg(f))\times \XF(Reg(f))\to \Gamma(\bot Reg(f))$ of $f$ defined by $\alpha(\XC,\YC)=\dt_{\XC}f_*\YC^\perp$ is symmetric and $C^{\infty}(Reg(f),\Re)$-bilinear, therefore $\alpha$ can be seen as a symmetric $\bot Reg(f)$-valued covariant $2$-tensor field on $Reg(f)$. 

Despite the fact that the metric $g$ is degenerate on $\Sigma(f)$, this is a Riemannian metric on any open set $U\subset Reg(f)$ and we can define the Levi-Civita connection $\nabla$ of $Reg(f)$ with the Koszul formula $g(\nabla_{\XC}\YC,\ZC)=\frac{1}{2}(\XC(g(\YC,\ZC))+\YC(g(\ZC,\XC))-\ZC(g(\XC,\YC))-g(\YC,[\XC,\ZC])-g(\ZC,[\YC,\XC])+g(\XC,[\ZC,\YC])),$	for all $\XC,\YC,\ZC\in \XF(U)$, where $[\XC,\YC]$ is the Lie bracket of vector fields. Since $\nabla_{\XC}\YC$ is $C^{\infty}(U,\Re)$-linear in $\XC$, we can see $\nabla \YC$ as a $TM^m\rvert_{U}$-valued covariant 1-tensor field on $U$. The mean curvature vector $\mathcal{H}$ of $f$ at $p\in Reg(f)$ is defined by $\HC(p)=\frac{1}{m}\sum_{i=1}^{m}\alpha(\XC_{i}(p),\XC_{i}(p))$, where $\XC_{1}(p),\dots,\XC_{m}(p)$ is an orthonormal basis of $T_pReg(f)$. Note that $\HC:Reg(f)\to f^*TN^n$ is a $C^{r,\beta}_{loc}$-map.

\section{Regularity of branched coordinates for conformal maps.}\label{sectionregularity}
The next lemma is obtained from the results and ideas of A. Coffman, Y. Pan and Y. Zhang in \cite{Coff}. The lemma extends the corollary 2.7 in \cite{Coff} for a system of equations. We write $R\Subset U$ to mean that $R$ is a bounded, open rectangle of the form
$(t_1, t_2) \times (t_3, t_4)$, with closure $\bar{R}$ contained in the open set $U\subset \Co$.    

\begin{lemma}\label{reg}
	Let $U \subset \Co$, $V_k\subset \Co$ be open sets for each $k\in \{1,2,..,n\}$ and let $E:U\times V\to\Co^n$, $f:U\to V$ be
	continuous maps, where $V=V_1\times V_2\times ..\times V_n$. Suppose that the partial derivatives $\dd{f}{x_1}$, $\dd{f}{x_2}$ exist at every point in $U$ except for countably many and the equality $\frac{\partial f}{\partial\z}(z)=E(z,f(z))$ is satisfied almost everywhere.
	\begin{enumerate}[label=(\roman*)]
		\item If $E$ is continuous, then for any $R\Subset U$,
		$f|_R\in C^{0,\beta}(R,\Co^n)$ for all $\beta \in(0,1)$.
		\item If $\beta\in(0,1)$ and
		$E\in C^{0,\beta}_{loc}(U\times V,\Co^n)$, then for any
		$R\Subset U$, $f|_R\in C^{1,\beta}(R,\Co^n)$.
		\item For $r\in\Na\cup \{\infty,\omega\}$, if
		$E\in C^r(U\times V,\Co^n)$, then $f\in C^{r,\beta}_{loc}(U,\Co^n)$ for all $\beta \in (0,1)$.
	\end{enumerate}
\end{lemma}
\begin{proof} Setting $f=(f_1,f_2,..,f_n)$, $E=(E_1,E_2,..,E_n)$ and
	the complex functions $E_{kj}:U\times V_j \to \Co$ by  $E_{kj}(z,w)=E_k(z,f_1(z),..,f_{j-1}(z),w,f_{j+1}(z),..,f_n(z))$,
	for $k,j \in \{1,2,..,n\}$. For the item $(i)$, if $E$ is continuous then so is $E_{kk}$ for all $k$ and since $\frac{\partial f_k}{\partial\z}=E_{kk}(z,f_k(z))$ almost everywhere, using the first item of corollary 2.7 in \cite{Coff}, we get that $f_k|_R\in C^{0,\beta}(R,\Co)$ for all $0<\beta<1$ and $1 \leq k\leq n$. Thus, $f|_R\in C^{0,\beta}(R,\Co^n)$. For the item $(ii)$, if $E\in C^{0,\beta}_{loc}(U\times V,\Co^n)$ then by item $(i)$ for any $R\Subset U$,
	$f|_R\in C^{0,\beta}(R,\Co^n)$ and therefore $E_{kk}\in C^{0,\beta^2}_{loc}(R\times V_k)$ for all $k\in \{1,2,..,n\}$. Using the second item of corollary 2.7 in \cite{Coff}, we get that $f_k\in C^{1}(U)$ for all $k\in \{1,2,..,n\}$ and hence $f\in C^{1}(U,\Co^n)$. Thus, $E_{kk}\in C^{0,\beta}_{loc}(U\times V_k)$ for all $k\in \{1,2,..,n\}$ and using again item $(ii)$ of corollary 2.7 in \cite{Coff}, we have $f\vert_R\in C^{1,\beta}(R,\Co^n)$. For the last item, if $r=1$ then by item $(i)$, the map $E':U\times V\to \Co^n$ given by $E'(z,w)=E(z,f(z))$ belongs to $C^{0,\beta}_{loc}(U\times V,\Co^n)$ and by item $(ii)$ $f\in C^{1,\beta}_{loc}(U,\Co^n)$. Now, let us suppose the assertion is true for $r=m$. If $E\in C^{m+1}(U\times V,\Co^n)$ then by inductive hypothesis $f\in C^{m,\beta}_{loc}(U,\Co^n)$ and therefore $E(z,f(z))$ belongs to $C^{m,\beta}_{loc}(U,\Co^n)$. Since $\frac{\partial f_k}{\partial\z}=E_{kk}(z,f_k(z))$ for all $k\in \{1,2,..,n\}$, then by theorem 15.0.7 in \cite{As} $f_k\in C^{m+1,\beta}_{loc}(U,\Co^n)$. Thus, the assertion holds for $r\in\Na\cup \{\infty\}$. If $r=\omega$, we have that $f\in C^{\infty}(U,\Co^n)$. Also, using the chain rule, the equality $\dd{^2f}{z\partial\z}=\dd{E(z,f(z))}{z}$ leads to a second-order elliptic system, where the RHS is a real analytic expression in $z$, $f$ and the first derivatives of $f$. This implies that $f\in C^{\omega}(U,\Co^n)$ (cf.\cite{Mo}).  
\end{proof}
\begin{lemma}\label{lembranch}
	Let $f:U\to \Re^n$ be a $C^{r,\beta}_{loc}$-map, $r\in\Na\cup\{\infty,\omega\}$, $\beta \in[0,1]$ and $U\subset \Co$ an open neighborhood of $0$. If $f=(f_1,\dots,f_n)$ satisfies (\ref{branchcond3}), then there exists $d\in C(U,\Co^n)\cap C^{r-1,\beta}_{loc}(U-\{0\},\Co^n)$ such that $\dd{f}{z}(z)=(s+1)z^sd(z)$, $d_1(0)=\frac{1}{2}$, $d_2(0)=-\frac{i}{2}$ and $d_j(0)=0$, for $j=3,\dots,n$. 
\end{lemma}
\begin{proof}
	Conjugating the equality $f_1(z)+if_2(z)=z^{s+1}+\sigma(z)$, we obtain that $f_1(z)-if_2(z)=\z^{s+1}+\bar{\sigma}(z)$. Therefore, $\dd{f_1}{z}(z)+i\dd{f_2}{z}(z)=(s+1)z^s+\dd{\sigma}{z}(z)\text{ and }\dd{f_1}{z}(z)-i\dd{f_2}{z}(z)=\dd{\bar{\sigma}}{z}(z)$.
	This implies that $\dd{f_1}{z}(z)=(s+1)z^s\frac{1}{2}(1+z^{-s}(s+1)^{-1}(\dd{\sigma}{z}(z)+\dd{\bar{\sigma}}{z}(z)))$ and $\dd{f_2}{z}(z)=-(s+1)z^s\frac{i}{2}(1+z^{-s}(s+1)^{-1}(\dd{\sigma}{z}(z)-\dd{\bar{\sigma}}{z}(z)))$. As we have $\dd{\sigma}{z}(z), \dd{\bar{\sigma}}{z}(z),\dd{f_j}{z}(z)=o(|z|^s)$, for $3\leq j\leq n,$ then the maps $d_1(z)=\frac{1}{2}(1+z^{-s}(s+1)^{-1}(\dd{\sigma}{z}(z)+\dd{\bar{\sigma}}{z}(z)))$,
	$d_2(z)=-\frac{i}{2}(1+z^{-s}(s+1)^{-1}(\dd{\sigma}{z}(z)-\dd{\bar{\sigma}}{z}(z)))$,
	$d_j(z)=z^{-s}\dd{f_j}{z}(z)$, $3\leq j\leq n$
	have continuous extensions to $U$, satisfying $d_1(0)=\frac{1}{2}$, $d_2(0)=-\frac{i}{2}$, $d_j(0)=0$ for $j=3,\dots,n$ and $\dd{f}{z}(z)=(s+1)z^sd(z)$ on $U$. Also, observe that $d$ is a $C^{r-1,\beta}_{loc}$-map on $U-\{0\}$.
\end{proof}
\begin{theorem}\label{Theo1}
	Let $f:M^2\to N^n$ be a $C^{r+2,\beta}_{loc}$-branched conformal immersion, $r\in \Na\cup \{0,\infty,\omega\}$, $\beta \in(0,1)$ and $\tilde{g}$ a $C^{r+1,\beta}_{loc}$-Riemannian metric on $N^n$. If the mean curvature vector $\HC$ of $f$ has a $C^{r,\beta}_{loc}$-extension to $M^2$, then for any branch point $p\in M^2$ of order $s$ and coordinate systems $(U,x)$, $(V,y)$ at $p$, $f(p)$ respectively, such that $x(p)=0$, $y(f(p))=0$ and $(\ref{branchcond3})$ is satisfied, we have
	\begin{align}
		&\dd{\hat{f}}{z}(z)=(s+1)z^sd(z)\label{cond1},\\
		&\dd{^2\hat{f}}{\z\partial z}(z)=|z|^{2s}l(z)\label{cond2},
	\end{align} 
	where $l:x(U)\to \Re^n$ is a $C^{r,\beta}_{loc}$-map and $d:x(U)\to \Co^n$ is a $C^{r+1,\beta}_{loc}$-map in which
	\begin{align}\label{cond3}
		d_1(z)=\frac{1}{2}(1+d'_1(z)),\text{ }d_2(z)=-\frac{i}{2}(1+d'_2(z)),
	\end{align}	 
	$d'_1(0)=0$, $d'_2(0)=0$ and $d_j(0)=0$ for $j=3,\dots,n$.
\end{theorem}
\begin{proof}
	If $p \in M^2$ is a branch point of order $s$ and $(U,x)$, $(V,y)$ are coordinate systems at $p$, $f(p)$ respectively, such that $x(p)=0$, $y(f(p))=0$ and $(\ref{branchcond3})$ is satisfied, then by lemma \ref{lembranch}, there exists $d\in C(x(U),\Co^n)\cap C^{r+1,\beta}_{loc}(x(U)-\{0\},\Co^n)$ such that $\dd{\hat{f}}{z}(z)=(s+1)z^sd(z)$, $d_1(0)=\frac{1}{2}$, $d_2(0)=-\frac{i}{2}$, $d_j(0)=0$, $j=3,\dots,n$. Observe that, we can write $d_1$ and $d_2$ in the form of $(\ref{cond3})$. Shrinking $U$ if necessary, we can suppose that $U\cap Reg(f)=U-\{p\}$ and $Re\{d\}$, $Im\{d\}$ are linearly independent on $x(U)$. Now, let us prove that $d$ is a $C^{r+1,\beta}_{loc}$-map on $x(U)$. As $g$ is a conformal metric on $Reg(f)$, we have that $\nabla_\dd{}{x_1}\dd{}{x_1}+\nabla_\dd{}{x_2}\dd{}{x_2}=0$ on $U-\{p\}$, $\delta=g(\dd{}{x_1},\dd{}{x_1})=g(\dd{}{x_2},\dd{}{x_2})$ and $g(\dd{}{x_1},\dd{}{x_2})=0$. Therefore, on $U-\{p\}$ we have that 
	\begin{align*}
		&\dt_\dd{}{x_1}f_*\dd{}{x_1}+\dt_\dd{}{x_2}f_*\dd{}{x_2}\\
		&=(\dt_\dd{}{x_1}f_*\dd{}{x_1})^T+(\dt_\dd{}{x_2}f_*\dd{}{x_2})^T+\alpha(\dd{}{x_1},\dd{}{x_1})+\alpha(\dd{}{x_2},\dd{}{x_2})\\
		&=f_*(\nabla_\dd{}{x_1}\dd{}{x_1}+\nabla_\dd{}{x_2}\dd{}{x_2})+\alpha(\dd{}{x_1},\dd{}{x_1})+\alpha(\dd{}{x_2},\dd{}{x_2})\\
		&=\delta(\alpha(\delta^{-\frac{1}{2}}\dd{}{x_1},\delta^{-\frac{1}{2}}\dd{}{x_1})+\alpha(\delta^{-\frac{1}{2}}\dd{}{x_2},\delta^{-\frac{1}{2}}\dd{}{x_2}))=2\delta\HC.
	\end{align*}
	Now, using that $f_*\dd{}{x_i}=\sum_{j=1}^{n}\dd{(y_j\circ f)}{x_i}\dd{}{y_j}\circ f$, $\delta=\sum_{j=1}^{n}\sum_{k=1}^{n}\gt_{jk}\circ f\frac{\partial (y_j\circ f)}{\partial x_1}\frac{\partial (y_k\circ f)}{\partial x_1}=\sum_{j=1}^{n}\sum_{k=1}^{n}\gt_{jk}\circ f\frac{\partial (y_j\circ f)}{\partial x_2}\frac{\partial (y_k\circ f)}{\partial x_2}$ and $\sum_{j=1}^{n}\sum_{k=1}^{n}\gt_{jk}\circ f\frac{\partial (y_j\circ f)}{\partial x_1}\frac{\partial (y_k\circ f)}{\partial x_2}=0$, we can express the equality $\dt_\dd{}{x_1}f_*\dd{}{x_1}+\dt_\dd{}{x_2}f_*\dd{}{x_2}=2\delta\HC$ in the local coordinates $(V,y)$ and get the system of equations $\dd{^2(y_j\circ f)}{x_1^2}+\dd{^2(y_j\circ f)}{x_2^2}=\sum_{i=1}^{n}\sum_{k=1}^{n}((\gt_{ik}\circ f)\HC_j-\Gat{i}{j}{k}\circ f)\sum_{i'=1}^{2}\dd{(y_i\circ f)}{x_{i'}}\dd{(y_k\circ f)}{x_{i'}}$ on $U-\{p\}$, for $j=1,\dots, n$, obtained by H. Kaul in \cite{kaul1}, where $\gt_{ik}$ are the components of the Riemannian metric $\gt$ in $(V,y)$, $\Gat{i}{j}{k}$ are the Christoffel symbols associated to the Levi-Civita connection of $N^n$ and $\HC_j$, for $j=1,\dots,n$ are the unique $C^{r,\beta}_{loc}$-functions on $U-\{p\}$ such that $\HC=\sum_{j=1}^{n}\HC_j\dd{}{y_j}\circ f$. Composing this with $x^{-1}$, we obtain the system of equations
	\begin{align}\label{eqH}
		\dd{^2\hat{f}_j}{x_1^2}+\dd{^2\hat{f}_j}{x_2^2}=\sum_{i=1}^{n}\sum_{k=1}^{n}(\hat{g}_{ik}(\HC_j\circ x^{-1})-\Gah{i}{j}{k})\sum_{i'=1}^{2}\dd{\hat{f}_i}{x_{i'}}\dd{\hat{f}_k}{x_{i'}}
	\end{align} on $x(U)-\{0\}$, for $j=1,\dots, n$, where $\hat{g}_{ik}=\gt_{ik}\circ f\circ x^{-1}$ and $\Gah{i}{j}{k}=\Gat{i}{j}{k}\circ f\circ x^{-1}$. Since $\dd{^2\hat{f}_j}{x_1^2}(z)+\dd{^2\hat{f}_j}{x_2^2}(z)=4\dd{^2\hat{f}_j}{\z\partial z}(z)=4(s+1)z^s\dd{d_j}{\z}(z)$ on $x(U)-\{0\}$ and \begin{align}\label{eqauxH}
		\dd{\hat{f}_i}{x_{1}}\dd{\hat{f}_k}{x_{1}}(z)+\dd{\hat{f}_i}{x_{2}}\dd{\hat{f}_k}{x_{2}}(z)&=4Re\{\dd{\hat{f}_i}{z}(z)\}Re\{\dd{\hat{f}_k}{z}(z)\}+4Im\{\dd{\hat{f}_i}{z}(z)\}Im\{\dd{\hat{f}_k}{z}(z)\}\\
		&=4Re\{\dd{\hat{f}_i}{z}(z)\overline{\dd{\hat{f}_k}{z}}(z)\}=4(s+1)^2z^s\z^sRe\{d_i\overline{d_k}\}\nonumber,
	\end{align} 
	substituting these expressions in (\ref{eqH}) and eliminating the factor $z^s$ in both sides this leads to the system of equations
	\begin{align}\label{eqHconformal}
		\dd{d_j}{\z}(z)=(s+1)\sum_{i=1}^{n}\sum_{k=1}^{n}(\hat{g}_{ik}(z)(\HC_j\circ x^{-1})(z)-\Gah{i}{j}{k}(z))\z^sRe\{d_i(z)\overline{d_k}(z)\},
	\end{align} on $x(U)-\{0\}$, for $j=1,\dots,n$. Since $\HC$ has a $C^{r,\beta}_{loc}$-extension to $U$, the functions $\HC_j$ for $j=1,\dots,n$ have $C^{r,\beta}_{loc}$-extensions as well. Thus, this system of equations has the form $\frac{\partial d}{\partial\z}(z)=E(z,d(z))$, where $E:x(U)\times \Co^n \to \Co^n$ is a $C^r$-map. Therefore, by lemma \ref{reg} $d$ is a $C^{r,\beta}_{loc}$-map on $x(U)$ and thus the RHS of the system of equations (\ref{eqHconformal}) is $C^{r,\beta}_{loc}$ on $x(U)$. By theorem 15.0.7 in \cite{As}, $d$ is a $C^{r+1,\beta}_{loc}$-map. Also, substituting  $\dd{\hat{f}_i}{x_{1}}\dd{\hat{f}_k}{x_{1}}(z)+\dd{\hat{f}_i}{x_{2}}\dd{\hat{f}_k}{x_{2}}(z)=4(s+1)^2z^s\z^sRe\{d_i\overline{d_k}\}$ in the system (\ref{eqH}) extended to $x(U)$, we obtain (\ref{cond2}).
\end{proof}

\begin{remark}\label{remGulliver}\normalfont
	Let $f:M^2 \to N^n$ be a $C^{2}$-conformal map and $p\in \Sigma(f)$. Suppose the mean curvature vector $\HC$ of $f$ has a $C^1$-extension to $M^2$. It is clear that $f$ satisfies the system (\ref{eqH}) for any coordinate systems $(U,x)$, $(V,y)$ at $p$, $f(p)$ respectively such that $x(p)=0$ and $y(f(p))=0$. By corollary 7.1 in \cite{Gu} and lemma 2.2 in \cite{O1}, there exists $\A\in GL(n)$ such that $\A (y\circ f\circ x^{-1})$ satisfies (\ref{branchcond1}), (\ref{branchcond2}) and (\ref{branchcond3}). In the case that $N^n$ is conformally flat, we can choose a conformal diffeomorphism $y$ (i.e., a diffeomorphism $y$ whose pullback metric is a multiple of the metric on the domain by a positive scalar function $\rho$ that we call the conformal factor), so $y\circ f:U\to \Re^n$ is a conformal map with $\Re^n$ endowed with the usual metric. Thus, by remark 2.3 in \cite{O1}, $\A$ can be chosen in such a way that $\ell\A \in O(n)$ for some $\ell \in \Re^+$. Therefore, setting $y'=\ell\A y$ and $x'(z)=\ell^{1/s}x(z)$, we have that $y'\circ f\circ x'^{-1}$ satisfies (\ref{branchcond1}), (\ref{branchcond2}) and (\ref{branchcond3}) with $y'$ being a conformal diffeomorphism. Thus, we can apply the theorem \ref{Theo1} to the coordinate systems $(U,x')$ and $(V,y')$. 
\end{remark}

\begin{proposition}\label{theo1reciprocal}
	Let $f:M^2\to N^n$ be a $C^{r+2,\beta}_{loc}$-branched conformal immersion, $r\in \Na\cup \{0,\infty,\omega\}$, $\beta \in[0,1]$ and $\tilde{g}$ a $C^{r+1,\beta}_{loc}$-Riemannian metric on $N^n$. If for each branch point $p\in M^2$ there exist coordinate systems $(U,x)$, $(V,y)$ at $p$, $f(p)$ respectively with $x(p)=0$, such that the conditions (\ref{cond1}), (\ref{cond2}) and (\ref{cond3}) are satisfied for some $s\in \Na$, then the mean curvature vector $\HC$ of $f$ has a $C^{r,\beta}_{loc}$-extension to $M^2$.
\end{proposition}
\begin{proof}
	If $(U,x)$, $(V,y)$ are coordinate systems at $p$, $f(p)$ respectively, satisfying conditions (\ref{cond1}), (\ref{cond2}) and (\ref{cond3}), then equations (\ref{eqH}) and (\ref{eqauxH}) hold on $x(U)-\{0\}$. Substituting (\ref{cond2}) in the LHS of (\ref{eqH}) and (\ref{eqauxH}) in the RHS of (\ref{eqH}), we get that $|z|^{2s}l_j(z)=\sum_{i=1}^{n}\sum_{k=1}^{n}(\hat{g}_{ik}(\HC_j\circ x^{-1})-\Gah{i}{j}{k})(s+1)^2|z|^{2s}Re\{d_i\overline{d_k}\}$. Eliminating the term $|z|^{2s}$, we can write the last equality in the form $(\HC_j\circ x^{-1})\sum_{i=1}^{n}\sum_{k=1}^{n}\hat{g}_{ik}Re\{d_i\overline{d_k}\}=F_j$, where $F_j\in C^{r,\beta}_{loc}(x(U),\Re)$. By condition (\ref{cond3}), we have $(\sum_{i=1}^{n}\sum_{k=1}^{n}\hat{g}_{ik}Re\{d_i\overline{d_k}\})(0)=\hat{g}_{11}|d_1(0)|^2+\hat{g}_{22}|d_2(0)|^2+2\hat{g}_{12}Re\{d_1(0)\overline{d_2(0)}\}=(\hat{g}_{11}(0)+\hat{g}_{22}(0))/4>0$. Then, shrinking $U$ if necessary, we have $(\HC_j\circ x^{-1})=F_j(\sum_{i=1}^{n}\sum_{k=1}^{n}\hat{g}_{ik}Re\{d_i\overline{d_k}\})^{-1}$. Thus, $\HC_j$ extends at $p$ and this extension  belongs to $C^{r,\beta}_{loc}(x(U),\Re)$. It follows the result.   
\end{proof}

\section{Characterization of branch coordinates.}\label{sectioncharacterizationbranch}
The lemmas in this section also will be essential in the sections \ref{sectionrepresentation} and \ref{sectioncurvatures}; here we will use these lemmas to obtain a criterion to determine if a map $f$ (possibly non-conformal) expressed in coordinate systems, satisfies the conditions  (\ref{cond1}), (\ref{cond2}) and (\ref{cond3}) of theorem \ref{Theo1}.
Observe that the existence of a pair of coordinate systems $(U,x)$, $(V,y)$ at points $p$, $f(p)$ respectively, satisfying the condition (\ref{cond1}) with $d$ being a $C^{r+1,\beta}_{loc}$-map, implies that $f$ is of class $C^{r+2,\beta}_{loc}$ on $U$. If in addition (\ref{cond3}) is satisfied, the open set $U$ can be shrunk in such a way that
\begin{align}
	&det(\begin{pmatrix}
		2Re\{d_1\}&-2Im\{d_1\}\\
		2Re\{d_2\}&-2Im\{d_2\}
	\end{pmatrix})\neq 0 \text{ on } x(U),\label{frontalcond}\\ 
	&\text{ }1+\frac{1}{2}(d'_1(z)+d'_2(z))\neq 0\text{ on } x(U)\text{ and }\sup_{z\in x(U)}\frac{|\frac{1}{2}(d'_1-d'_2)|}{|1+\frac{1}{2}(d'_1+d'_2)|}<1.\label{quasiregularcond}
\end{align}
The condition (\ref{frontalcond}) guarantees $\Sigma(f)\cap U=\{p\}$ and that $y\circ f\circ x^{-1}$ is a normalized frontal, a notion we will introduce later. The condition (\ref{quasiregularcond}) is related to the quasiconformality and quasiregularity of some maps that we consider in section \ref{sectionrepresentation}. All these conditions motivate the following definition. 
\begin{definition}\label{branchcoord}\normalfont
	Let $f:M^2\to N^n$ be a $C^{1}$-map, $\beta\in [0,1]$, $k\in \Na\cup \{0,\infty,\omega\}$ and $p\in M^2$. We will call {\it $C^{k,\beta}_{loc}$-branch coordinates at $p$ for $f$} or simply {\it $C^{k,\beta}_{loc}$-branch coordinates at $p$}, a pair of coordinate systems $(U,x)$, $(V,y)$ at $p$, $f(p)$ respectively, with $x(p)=0$ and $y(f(p))=0$, such that the conditions (\ref{cond1}), (\ref{cond3}), (\ref{frontalcond}) and (\ref{quasiregularcond}) are satisfied with $d$ being a $C^{k,\beta}_{loc}$-map. If $(U,x)$, $(V,y)$ are $C^{k,\beta}$-branch coordinates at $p$ with $k\geq 1$, we will call these pair of coordinate systems {\it $C^{k,\beta}_{loc}$-regular branch coordinates at $p$ for $f$} or simply {\it $C^{k,\beta}_{loc}$-regular branch coordinates at $p$} if additionally the condition (\ref{cond2}) is satisfied with $l$ being a $C^{k-1,\beta}_{loc}$-map.  
\end{definition}
From lemma \ref{lembranch}, we can see that for every branch point $p$ of a $C^1$-map, there exist $C^0$-branch coordinates at $p$. We will see that every $C^{k,\beta}_{loc}$-branch coordinates with $k\in \Na\cup \{0,\infty,\omega\}$ satisfies (\ref{branchcond1}), (\ref{branchcond2}) and (\ref{branchcond3}). On the other hand, it is natural to wonder which of the conditions (\ref{cond1}), (\ref{cond2}) and (\ref{cond3}) imply the rest so that we will show in lemma \ref{relations} a relation between (\ref{cond1}) and (\ref{cond2}). We shall need the following lemmas. 
\begin{lemma}\label{lemdivision}
	Let $U=(t_1,t_2)\times(t_3,t_4)\subset \Re^2$ be an open neighborhood of $0$, $k\in\{0\}\cup\Na$, $V\subset \Re^k$ an open set and $F:U\times V \to \Co$ a map of class $C^{r,\beta}_{loc}$, where $r \in \amsmathbb{N}\cup\{\omega,\infty\}$, $\beta\in [0,1]$. The following assertions hold:
	\begin{enumerate}[label=(\roman*)]
		\item There exists a real function $l:U\times V \to \Re$ of class $C^{r-1,\beta}_{loc}$ such that $F(z,y)=l(z,y)z$ on $U\times V$ if and only if $Im\{F(z,y)\bar{z}\}=0$ on $U\times V$.
		\item If $n$ is a natural number such that $r\geq n(n+1)/2$ and $F(z,y)\bar{e}(z,y)\z^n=E(z,y)e(z,y)z^n$ for some $C^{r,\beta}_{loc}$-map $E:U\times V \to \Co$ and a non-vanishing $C^{r,\beta}_{loc}$-map $e:U\times V \to \Co$, then there exist a complex $C^{r-n(n+1)/2,\beta}_{loc}$-map $G$ on $U\times V$, such that $F(z,y)=G(z,y)e(z,y)z^n$ and $E(z,y)=G(z,y)\bar{e}(z,y)\z^n$.
	\end{enumerate}
\end{lemma}
\begin{proof}
	One direction of the item $(i)$ is immediate. We shall prove the other direction. Setting $F(x_1,x_2,y)=(a(x_1,x_2,y),b(x_1,x_2,y))$, we have $Im\{F(z,y)\bar{z}\}=-x_2a(x_1,x_2,y)+x_1b(x_1,x_2,y)=0$ which implies that $a(0,x_2,y)=0$, $b(x_1,0,y)=0$ for all $x_1\in (t_1,t_2)$, $x_2\in (t_3,t_4)$ and $y\in V$. Then, $a(x_1,x_2,y)=\int_{0}^{1}\dd{a(tx_1,x_2,y)}{t}dt=x_1\int_{0}^{1}\dd{a}{x_1}(tx_1,x_2,y)dt$ and similarly $b(x_1,x_2,y)=x_2\int_{0}^{1}\dd{b}{x_2}(x_1,tx_2,y)dt$ on $U$. Therefore, setting $c_1(x_1,x_2,y)=\int_{0}^{1}\dd{a}{x_1}(tx_1,x_2)dt \text{ and } c_2(x_1,x_2,y)=\int_{0}^{1}\dd{b}{x_2}(x_1,tx_2,y)dt$, we have $0=-x_2a(x_1,x_2,y)+x_1b(x_1,x_2,y)=-x_2x_1c_1(x_1,x_2,y)+x_1x_2c_2(x_1,x_2,y)$ which implies $c_1(x_1,x_2,y)=c_2(x_1,x_2,y)$ on $U\times V$. Also, by the triangle inequality and the mean value theorem for integrals, $c_1$ and $c_2$ are $C^{r-1,\beta}_{loc}$-maps. Thus, $F(z)=c_1(x_1,x_2,y)z$ with $c_1$ being a real function of class $C^{r-1,\beta}_{loc}$ on $U\times V$. For the item $(ii)$, it suffices to prove the case in which $e=1$ on $U\times V$. Let us prove this item by induction on $n$. If $F(z,y)\z=E(z,y)z$ for some $C^{r,\beta}_{loc}$-map $E:U\times V \to \Co$ on $U\times V$, set the maps $H'_1(z,y):=E(z,y)-\overline{F}(z,y)$ and $H'_2(z,y):=E(z,y)+\overline{F}(z,y)$. Then, $z H'_1(z,y)$ is a purely imaginary number and $z H'_2(z,y)$ is a real number, which is equivalent to $Im\{z iH'_1(z,y)\}=0$ and $Im\{z H'_2(z,y)\}=0$ on $U\times V$. By item $(i)$, there exist real functions $l'_1$, $l'_2$ of class $C^{r-1,\beta}_{loc}$ on $U\times V$ such that $iH'_1(z,y)=l'_1(z,y)\z$ and $H'_2(z,y)=l'_2(z,y)\z$. Thus, $E(z,y)=(H'_1(z,y)+H'_2(z,y))/2=\frac{1}{2}(l'_2(z,y)-il'_1(z,y))\z$ and $\overline{F}(z,y)=(H'_2(z,y)-H'_1(z,y))/2=\frac{1}{2}(l'_2(z,y)+il'_1(z,y))\z$. So, the assertion is true for $n=1$. Let us suppose the assertion holds for all the natural numbers less than or equal to $n$. If $r\geq (n+1)(n+2)/2$ and 
	\begin{align}\label{specialeq}
		F(z,y)\z^{n+1}=E(z,y)z^{n+1}
	\end{align} for some $C^{r,\beta}_{loc}$-map $E:U\times V \to \Co$ on $U\times V$, then we have 
	\begin{align}\label{specialeq2}
		E(z,y)z^n=E_1(z,y)\overline{z} \text{ and } F(z,y)\z^n=F_1(z,y)z
	\end{align} for some complex $C^{r-1,\beta}_{loc}$-maps $E_1$, $F_1$ on $U\times V$. Then, there exist complex $C^{r-2,\beta}_{loc}$-maps $E_2$, $F_2$ on $U\times V$ such that $E_1(z,y)=E_2(z,y)z$ and $F_1(z,y)=F_2(z,y)\z$. Substituting these in (\ref{specialeq2}) we obtain that $E(z,y)z^{n-1}=E_2(z,y)\overline{z}$ and $F(z,y)\z^{n-1}=F_2(z,y)z$. Continuing this procedure, we obtain $C^{r-n-1,\beta}_{loc}$-maps $E_{n+1}$, $F_{n+1}$ such that $E(z,y)=E_{n+1}(z,y)\overline{z}$ and $F(z,y)=F_{n+1}(z,y)z$. Now, substituting these expressions in (\ref{specialeq}), we get $F_{n+1}(z,y)\z^n=E_{n+1}(z,y)z^n$. Since $r-n-1\geq n(n+1)/2$, by inductive hypothesis there exists a complex $C^{r-(n+1)(n+2)/2,\beta}_{loc}$-map $G$ on $U\times V$, such that $E_{n+1}(z,y)=\z^nG(z,y)$ and $F_{n+1}(z,y)=z^nG(z,y)$. We have $E(z,y)=E_{n+1}(z,y)\overline{z}=G(z,y)\z^{n+1}$ and $F(z,y)=F_{n+1}(z,y)z=G(z,y)z^{n+1}$.    
\end{proof}
\begin{remark}\normalfont
	If $k=0$ in lemma \ref{lemdivision}, it means that $V=\{0\}$ or $F$ is simply a map defined on $U$.
\end{remark}
\begin{lemma}\label{lemspecial}
	Let $U=(t_1,t_2)\times(t_3,t_4)\subset \Re^2$ be an open neighborhood of $0$, $k\in\{0\}\cup\Na$, $V\subset \Re^k$ an open set and $F:U\times V \to \Co$ a complex map of class $C^{r,\beta}_{loc}$, where $r \in \amsmathbb{N}\cup\{\omega,\infty\}$, $\beta\in [0,1]$. If $n$ is a natural number such that $r\geq n(n+1)/2$, then the following assertions hold: 
	\begin{enumerate}[label=(\roman*)]
		\item If $b:U\times V\to \Co$ is a non-vanishing complex map of class $C^{r,\beta}_{loc}$, then there exists a real function $l:U\times V \to \Re$ of class $C^{r-n(n+1)/2,\beta}_{loc}$ such that $F(z,y)=l(z,y)b(z,y)z^n$ on $U\times V$ if and only if $Im\{F(z,y)\bar{b}(z)\bar{z}^n\}=0$ on $U\times V$.
		\item If $F(z,y)(a(z,y)\z^n+b(z,y)z^n)=E(z,y)(a(z,y)\z^n-b(z,y)z^n)$ for some $C^{r,\beta}_{loc}$-map $E:U\times V \to \Co$ and non-vanishing complex $C^{r,\beta}_{loc}$-maps $a$, $b$ on $U\times V$, then there exists a complex $C^{r-n(n+1)/2,\beta}_{loc}$-map $G$ on $U\times V$, such that $F(z,y)=G(z,y)(a(z,y)\z^n-b(z,y)z^n)$ and $E(z,y)=G(z,y)(a(z,y)\z^n+b(z,y)z^n)$. 
	\end{enumerate}
\end{lemma}
\begin{proof}
	For the item $(i)$, if $Im\{F(z,y)\bar{b}(z)\bar{z}^n\}=0$ on $U\times V$, then $F(z,y)\bar{b}(z,y)\bar{z}^n=\bar{F}(z,y)b(z,y)z^n$ and by lemma \ref{lemdivision}, there exists a complex $C^{r-n(n+1)/2,\beta}_{loc}$-map $G$ on $U\times V$ such that $F(z,y)=G(z,y)b(z,y)z^n$ and $\bar{F}(z,y)=G(z,y)\bar{b}(z,y)\bar{z}^n$. It implies that $\bar{G}(z,y)=G(z,y)$ and therefore $G(z,y)$ is real. The converse is immediate. On the other hand, since the first equality given in item $(ii)$ is equivalent to $a(z,y)z^n(F(z,y)-E(z,y))=-b(z,y)\z^n(E(z,y)+F(z,y))$, applying the item $(ii)$ of lemma \ref{lemdivision}, the result follows. 	   
\end{proof}
\begin{lemma}\label{relations}
	Let $f:M^2\to N^n$ be a $C^{1}$-map, $\beta \in[0,1]$ and $p\in M^2$. If there exist coordinate systems $(U,x)$, $(V,y)$ at $p$, $f(p)$ respectively, such that for some $s\in\Na$, the condition (\ref{cond1}) is satisfied with $d$ being a $C^{k,\beta}_{loc}$-map with $k\in (\Na-\{1\})\cup \{\infty,\omega\}$, then if $k\geq s(s+1)/2+1$, the condition (\ref{cond2}) is satisfied on an open neighborhood $\hat{U}\subset x(U)$ of $0$ with $l$ being a map of class $C^{k-s(s+1)/2-1,\beta}_{loc}$.
\end{lemma} 
\begin{proof}
	If we have (\ref{cond1}) then $\dd{^2\hat{f}}{\z\partial z}(z)=(s+1)z^sd_{\z}(z)$ and therefore $Im\{z^sd_{\z}\}=0$ on $x(U)$. Now, applying the item $(i)$ of lemma \ref{lemspecial}, the assertion follows.
\end{proof}
By lemma \ref{relations}, if $p$ is a branch point, every $C^{k,\beta}_{loc}$-branch coordinates at $p$ in the cases $k=\infty, \omega$ are $C^{k,\beta}_{loc}$-regular branch coordinates at $p$. The following lemma and lemma \ref{lembranch} show that a point $p$ is a branch point of order $s'$ if and only if there exist $C^{0}$-branch coordinates at $p$ with the number $s$ in $(\ref{cond1})$ being equal to $s'$. 
\begin{lemma}\label{equivalencebranchcoor}
	Let $f:M^2\to N^n$ be a $C^{1}$-map and $p\in M^2$. Then, for all pair of coordinate systems $(U,x)$, $(V,y)$ at $p$ and $f(p)$ respectively with $x(p)=0$ and $y(f(p))=0$, such that conditions (\ref{cond1}) and (\ref{cond3}) are satisfied with $d$ being a continuous map, we also have that (\ref{branchcond1}), (\ref{branchcond2}) and (\ref{branchcond3}) are satisfied. 
\end{lemma} 
\begin{proof}
	In such coordinate systems, for $j\in\{1,\dots,n\}$, since $\dd{\hat{f}_j}{z}(z)=(s+1)z^sd_j(z)$, then $\dd{\hat{f}_j}{\z}(z)=(s+1)\z^s\bar{d}_j(z)$. Therefore, as $d_j(0)=0$ for $j=3,\dots,n$, we have that $\dd{\hat{f}_j}{x_1}(z)=\dd{\hat{f}_j}{z}(z)+\dd{\hat{f}_j}{\z}(z)=o(z^s)$ and $\dd{\hat{f}_j}{x_2}(z)=i(\dd{\hat{f}_j}{z}(z)-\dd{\hat{f}_j}{\z}(z))=o(z^s)$. Using (\ref{cond3}), we have $(s+1)z^s+\dd{\sigma}{z}=\dd{\hat{f}_1+i\hat{f}_2}{z}(z)=(s+1)z^s(1+\frac{d'_1(z)+d'_2(z)}{2})$ and then $\dd{\sigma}{z}=(s+1)z^s(\frac{d'_1(z)+d'_2(z)}{2})$. Also,  $\dd{\sigma}{\z}=\dd{\hat{f}_1+i\hat{f}_2}{\z}(z)=(s+1)\z^s(\frac{\overline{d'_1(z)-d'_2(z)}}{2})$ and since $d'_1(0)=d'_2(0)=0$, we have that $\dd{\sigma}{x_1}(z)=\dd{\sigma}{z}(z)+\dd{\sigma}{\z}(z)=o(z^s)$,  $\dd{\sigma}{x_2}(z)=i(\dd{\sigma}{z}(z)-\dd{\sigma}{\z}(z))=o(z^s)$, so condition (\ref{branchcond3}) is satisfied. On the other hand, $\sigma(z)=\int_{0}^{1}\dd{\sigma(tz)}{t}dt=\int_{0}^{1}(\dd{\sigma}{z}(tz)z+\dd{\sigma}{\z}(tz)\z) dt=z^{s+1}(s+1)\int_{0}^{1}t^s(\frac{d'_1(tz)+d'_2(tz)}{2})dt+\z^{s+1}(s+1)\int_{0}^{1}t^s(\frac{\overline{d'_1(tz)-d'_2(tz)}}{2})dt=o(z^{s+1})$ and similarly, for $j=3,\dots,n$, we obtain that $\hat{f}_j(z)=z^{s+1}(s+1)\int_{0}^{1}t^sd_j(tz)dt+\z^{s+1}(s+1)\int_{0}^{1}t^s\bar{d}_j(tz)dt=o(z^{s+1})$, so conditions (\ref{branchcond1}) and (\ref{branchcond2}) are satisfied.
\end{proof}
The following lemmas will be fundamental for most results of the rest of the paper.
\begin{lemma}\label{holderl1}
	Let $e:U \to \Co$ be a complex $C^{0,\beta}$-function, $U\subset \Co$ an open neighborhood of $0$, $s\in \Na$ and $\beta \in [0,1]$. Let us set the map 
	\begin{align}\label{function}
		\varpi(z)=\frac{\z^s}{z^s}e(z)
	\end{align} defined on $U-\{0\}$.
	\begin{enumerate}[label=(\roman*)]
		\item If $e(0)=0$, then $\varpi$ is a complex $C^{0,\beta}$-function on $U-\{0\}$. Also, if we suppose that $e$ is a complex $C^{0,\beta}_{loc}$-function on $U-\{0\}$ then $\varpi$ is a complex $C^{0,\beta}_{loc}$-function on $U-\{0\}$.  
		\item The map $\varpi$ has a unique $C^{0,\beta}_{loc}$-extension on $U$ if and only if $e(0)=0$. Also, a continuous extension $\varpi$ satisfies that $\varpi(0)=0$.  
	\end{enumerate}  
\end{lemma}
\begin{proof}
	To prove item $(i)$, observe that $|e(z_2)-e(z_1)|\leq k_1|z_2-z_1|^\beta$ for some $k_1\in[0,\infty)$. Since $e(0)=0$, then $|e(z_1)|\leq k_1|z_1|^\beta$ and thus
	\begin{align*}
		\frac{|\frac{\z_2^s}{z_2^s}e(z_2)-\frac{\z_1^s}{z_1^s}e(z_1)|}{|z_2-z_1|^\beta}&
		=\frac{|\frac{\z_2^s}{z_2^s}(e(z_2)-e(z_1))+(\frac{\z_2^s}{z_2^s}-\frac{\z_1^s}{z_1^s})e(z_1)|}{|z_2-z_1|^\beta}\\
		&\leq k_1+\frac{k_1|\frac{\z_2^s}{z_2^s}-\frac{\z_1^s}{z_1^s}||z_1|^\beta}{|z_2-z_1|^\beta}\\
		&= k_1+\frac{k_1|\frac{\z_2}{z_2}-\frac{\z_1}{z_1}||\sum_{j=0}^{s-1}(\frac{\z_2}{z_2})^{s-j-1}(\frac{\z_1}{z_1})^{j}||z_1|^\beta}{|z_2-z_1|^\beta}\\
		&\leq k_1+\frac{k_1s|\frac{\z_2}{z_2}-\frac{\z_1}{z_1}||z_1|^\beta}{|z_2-z_1|^\beta}.
	\end{align*}
	As $|z|\leq |z|^\beta$ for $0\leq |z|\leq 1$ and $|\frac{1}{2}(\frac{\z_2}{z_2}-\frac{\z_1}{z_1})|\leq 1$, we have
	\begin{align*}
		\frac{|\frac{\z_2}{z_2}e(z_2)-\frac{\z_1}{z_1}e(z_1)|}{|z_2-z_1|^\beta}&\leq k_1+\frac{2k_1s|\frac{1}{2}(\frac{\z_2}{z_2}-\frac{\z_1}{z_1})|^\beta|z_1|^\beta}{|z_2-z_1|^\beta}\\
		&=  k_1+\frac{2^{1-\beta}k_1s|z_1\z_2-z_2\z_1|^{\beta}|z_1|^\beta}{|z_1|^\beta|z_2|^\beta|z_2-z_1|^\beta}\\
		&=k_1+2^{1-\beta}k_1s(\frac{|(z_1-z_2)\z_2+z_2(\z_2-\z_1)|}{|z_2||z_2-z_1|})^\beta\\
		&\leq k_1+2^{1-\beta}k_1s(\frac{|z_1-z_2||\z_2|+|z_2||\z_2-\z_1|}{|z_2||z_2-z_1|})^\beta\\
		&= k_1(1+2s).
	\end{align*}
	The last part follows from applying the previous reasoning on compact subsets $K\subset U-\{0\}$. On the other hand, to prove $(ii)$, if we have $e(0)=0$, then $\lim_{z\to 0} |\varpi(z)|=\lim_{z\to 0} |e(z)|=0$ and therefore $\varpi$ has a unique continuous extension on $U$. By item $(i)$, $\varpi$ is a $C^{0,\beta}_{loc}$-map on $U-\{0\}$ and therefore also the extension. Conversely, if $\lim_{z\to 0} \varpi(z)=c<\infty$, then for $t\in \Re$, $c=\lim_{t\to 0} \varpi(t+i0)=e(0)$ and $c=\lim_{t\to 0} \varpi(it)=-e(0)$. This implies that $0=c=e(0)$.
\end{proof}

\begin{lemma}\label{lemisolated}
	Let $f:U\to \Re$ be a  $C^{r,\beta}_{loc}$-function, $r\in\Na\cup\{0\}$, $\beta \in[0,1]$, $U\subset \Re^m$ an open set and $p\in U$. If $f$ is a $C^{r+1,\beta}_{loc}$-function on $U-\{p\}$ and all its first partial derivatives have $C^{r,\beta}_{loc}$-extensions on $U$, then $f$ is a $C^{r+1,\beta}_{loc}$-function on $U$.  
\end{lemma}
\begin{proof}
	It suffices to prove the case $r=0$, the general case is obtained applying this last one to the partial derivatives of order $r$. Denoting $p=(p_1,\dots, p_m)$, let us take $j\in\{1,\dots,m\}$, $\epsilon>0$ such that $B_\epsilon(p)\subset U$ and a sequence of real numbers $t_n\in (0,\epsilon)$ such that $t_n\to 0$. Then, by the mean value theorem, for each $n\in\Na$ there exists $t'_n\in (p_j,p_j+t_n)$ such that $f(p_1,\dots,p_j+t_n,\dots,p_m)-f(p)=\dd{f}{x_j}(t'_n)t_n$. Therefore, as $p_j<t'_n<p_j+t_n$ we have $\lim_{n\to \infty}t'_n=p_j$ and $\lim_{n\to \infty}(f(p_1,\dots,p_j+t_n,\dots,p_m)-f(p))t_{n}^{-1}=\lim_{n\to \infty}\dd{f}{x_j}(t'_n)<\infty$.
	Similarly, we can show the same for a sequence $t_n\in (-\epsilon,0)$ such that $t_n\to 0$. This fact shows that all the partial derivatives exist at $p$ and coincide with the extensions evaluated at $p$. Since, all the partial derivatives exist and are $C^{0,\beta}_{loc}$-functions, $f$ is differentiable at $p$ and therefore $f$ is a $C^{1,\beta}_{loc}$-function on $U$.
\end{proof}

\begin{lemma}\label{lemext1}
	If $e:U \to \Co$ is a $C^{r,\beta}_{loc}$-map, $r\in \Na\cup\{0\}$, $U\subset \Co$ an open convex neighborhood of $0$, $\beta \in [0,1]$ and $k$ is an integer such that $0\leq k\leq r$, then the map \begin{align}\label{belt1}
		\varpi'(z)=\frac{\z}{z}e(z)
	\end{align} defined on $U-\{0\}$ has a $C^{k,\beta}_{loc}$-extension on $U$ if and only if $\dd{^je}{\z^j}(0)=0$ for all $0\leq j \leq k$. Furthermore, if $k\geq 1$ and $\varpi'$ has a $C^{k,\beta}_{loc}$-extension on $U$, the following is true:
	\begin{enumerate}[label=(\roman*)]
		\item If $h, h'\in \Na$ with $2\leq h+h'\leq k$ and $\dd{^{h+h'}e}{\z^h \partial z^{h'}}(0)=0$, then $(h'+1)\dd{^{h+h'}\varpi'}{\z^h\partial z^{h'}}(0)=h\dd{^{h+h'}e}{\z^{h-1}\partial z^{h'+1}}(0)$ and $\dd{^{h+h'}\varpi'}{\z^{h+1}\partial z^{h'-1}}(0)=0$.
		\item If $j\in \Na$ with $1\leq j\leq k$, then $\dd{^{j}\varpi'}{ z^{j}}(0)=0$. If we also have that $\dd{^{j}e}{z^{j}}(0)=0$, then $\dd{^{j}\varpi'}{\z\partial z^{j-1}}(0)=0$.
		\item If $j\in \Na$ with $1\leq j\leq k$, then $\dd{^{j}\varpi'}{\z^j}(0)=j\dd{^{j}e}{\z^{j-1}\partial z}(0)$. 
	\end{enumerate}
\end{lemma}
\begin{proof}
	Let us prove the equivalence in the first part by induction on $k$. By item $(ii)$ of lemma \ref{holderl1}, the assertion is true for $k=0$. We can assume that $r\in\Na$, otherwise $r=0$ and the assertion is true. Suppose the assertion holds for all the natural numbers less than or equal to $n$, where $0\leq n<r$. We have $\dd{(\frac{\z}{z}e)}{\z}(z)=\frac{e(z)}{z}+\frac{\z}{z}\dd{e}{\z}(z)=\frac{1}{z}(z\int_{0}^{1}\dd{e}{z}(tz)dt+\z\int_{0}^{1}\dd{e}{\z}(tz)dt)+\frac{\z}{z}\dd{e}{\z}(z)$ and hence 
	\begin{align}\label{eqaux}
		\dd{(\frac{\z}{z}e)}{\z}(z)-\int_{0}^{1}\dd{e}{z}(tz)dt=\frac{\z}{z}(\int_{0}^{1}\dd{e}{\z}(tz)dt+\dd{e}{\z}(z))
	\end{align}
	on $U-\{0\}$. If we suppose that $\dd{^je}{\z^j}(0)=0$ for all $0\leq j \leq n+1$, then by the inductive hypothesis the RHS of (\ref{eqaux}) has a $C^{n,\beta}_{loc}$-extension on $U$. On the other hand, 
	\begin{align}\label{equationaux}
		\dd{(\frac{\z}{z}e)}{z}(z)=\frac{\z}{z}(\dd{e}{z}(z)-\frac{e(z)}{z})
	\end{align} and $\dd{e}{z}(z)-\frac{1}{z}e(z)=\dd{e}{z}(z)-\int_{0}^{1}\dd{e}{z}(tz)dt+\frac{\z}{z}\int_{0}^{1}\dd{e}{\z}(tz)dt$. By inductive hypothesis the RHS of this last equality has a $C^{n,\beta}_{loc}$-extension $e'$ on $U$ that vanishes at $z=0$. Also, $\dd{^j e'}{\z^j}(z)=\dd{^j(\dd{e}{z}(z)-\frac{1}{z}e(z))}{\z^j}=\dd{^{j+1}e}{z\partial \z^j}(z)-\frac{1}{z}\dd{^je}{\z^j}(z)=\dd{^{j+1}e}{z\partial \z^j}(z)-\int_{0}^{1}\dd{^{j+1}e}{z\partial\z^j}(tz)dt-\frac{\z}{z}\int_{0}^{1}\dd{^{j+1}e}{\z^{j+1}}(tz)dt$ and since $\frac{\z}{z}\int_{0}^{1}\dd{^{j+1}e}{\z^{j+1}}(tz)dt$ has a continuous extension on $U$ for all $j\in\{1,\dots,n\}$, then by lemma \ref{holderl1}, $\dd{^j e'}{\z^j}(0)=0$ for all $j\in\{0,\dots,n\}$. By inductive hypothesis, the map $\frac{\z}{z}e'$ which is the RHS of (\ref{equationaux}) has a $C^{n,\beta}_{loc}$-extension. Thus, as $\dd{\varpi'}{z}$ and $\dd{\varpi'}{\z}$ have $C^{n,\beta}_{loc}$-extensions, by lemma \ref{lemisolated} the map $\varpi'$ is of class $C^{n+1,\beta}_{loc}$ on $U$. Conversely, if $\varpi'$ has a $C^{n+1,\beta}_{loc}$-extension on $U$, then $e(0)=0$ and the equality (\ref{eqaux}) holds on $U-\{0\}$. Since the LHS of this equality has a $C^{n,\beta}_{loc}$-extension on $U$, by the inductive hypothesis $\dd{^j(\int_{0}^{1}\dd{e}{\z}(tz)dt+\dd{e}{\z}(z))}{\z^j}\vert_{z=0}=0$ for all $0\leq j \leq n$. This last is equivalent to having $\dd{^je}{\z^j}(0)=0$ for all $1\leq j \leq n+1$. 
	
	For the item $(i)$ in the last part, if $h, h'\in \Na$ with $2\leq h+h'\leq k$, as $z\varpi'(z)=\z e(z)$ then $z\dd{^h\varpi'}{\z^h}(z)=h\dd{^{h-1}e}{\z^{h-1}}(z)+\z\dd{^he}{\z^h}(z)$ and therefore 
	\begin{align}\label{initialeqaux}
		h'\dd{^{h+h'-1}\varpi'}{\z^h\partial z^{h'-1}}(z)+z\dd{^{h+h'}\varpi'}{\z^h \partial z^{h'}}(z)=h\dd{^{h+h'-1}e}{\z^{h-1}\partial z^{h'}}(z)+\z\dd{^{h+h'}e}{\z^h \partial z^{h'}}(z).
	\end{align}
	Hence, $c(z):=h'\dd{^{h+h'-1}\varpi'}{\z^h\partial z^{h'-1}}(z)-h\dd{^{h+h'-1}e}{\z^{h-1}\partial z^{h'}}(z)$ vanishes at $z=0$, so using that $c(z)=z\int_{0}^{1}\dd{c}{z}(tz)dt+\z\int_{0}^{1}\dd{c}{\z}(tz)dt$ and substituting $c$ by this expression in (\ref{initialeqaux}), we can obtain  
	\begin{align}\label{eqaux2}
		&h'\int_{0}^{1}\dd{^{h+h'}\varpi'}{\z^h\partial z^{h'}}(tz)dt+h'\frac{\z}{z}\int_{0}^{1}\dd{^{h+h'}\varpi'}{\z^{h+1}\partial z^{h'-1}}(tz)dt+\dd{^{h+h'}\varpi'}{\z^h \partial z^{h'}}(z)\\
		&=h\int_{0}^{1}\dd{^{h+h'}e}{\z^{h-1}\partial z^{h'+1}}(tz)dt+h\frac{\z}{z}\int_{0}^{1}\dd{^{h+h'}e}{\z^{h}\partial z^{h'}}(tz)dt+\frac{\z}{z}\dd{^{h+h'}e}{\z^h \partial z^{h'}}(z).\nonumber
	\end{align}
	By item $(ii)$ of lemma \ref{holderl1}, the RHS of this equality has a $C^{0,\beta}_{loc}$-extension, therefore the term $h'\frac{\z}{z}\int_{0}^{1}\dd{^{h+h'}\varpi'}{\z^{h+1}\partial z^{h'-1}}(tz)dt$ has a $C^{0,\beta}_{loc}$-extension that vanishes at $z=0$. Evaluating at $z=0$ the sides of (\ref{eqaux2}), we obtain $(i)$. 
	
	For the item $(ii)$, since $\frac{\z}{z}\overline{\varpi'(z)}=\overline{e(z)}$, the first part of this lemma implies that $\dd{^j\varpi}{z^j}(0)=\overline{\dd{^j\overline{\varpi}}{\z^j}(0)}=0$ for all $j\in \{0,\dots,k\}$. On the other hand, as $z\varpi'(z)=\z e(z)$ then $\z\dd{^je}{z^j}(z)=j\dd{^{j-1}\varpi}{z^{j-1}}(z)+z\dd{^j\varpi}{z^j}(z)$ for $j\in \{1,\dots,k\}$ and therefore $j\int_{0}^{1}\dd{^{j}\varpi'}{z^{j}}(tz)dt+j\frac{\z}{z}\int_{0}^{1}\dd{^{j}\varpi'}{\z\partial z^{j-1}}(tz)dt+\dd{^{j}\varpi'}{z^{j}}(z)=\frac{\z}{z}\dd{^{j}e}{z^{j}}(z)$. If $\dd{^je}{z^j}(0)=0$, then $\frac{\z}{z}\dd{^j}{z^j}e(z)$ has a $C^{0,\beta}_{loc}$-extension on $U$ and therefore so does $\frac{\z}{z}\int_{0}^{1}\dd{^{j}\varpi'}{\z\partial z^{j-1}}(tz)dt$. This implies that $\dd{^{j}\varpi'}{\z\partial z^{j-1}}(0)=0$.  
	
	The item $(iii)$ follows similarly from the equality $z\dd{^h\varpi'}{\z^h}(z)=h\dd{^{h-1}e}{\z^{h-1}}(z)+\z\dd{^he}{\z^h}(z)$ for all $h\in\{1,\dots,k\}$ and the fact that $\dd{^je}{\z^j}(0)=0$ for all $j\in\{0,\dots,k\}$.   
\end{proof}

\begin{lemma}\label{lemext2}
	Let $e:U \to \Co$ be a complex $C^{k,\beta}_{loc}$-function, $k\in \Na$, $U\subset \Co$ an open convex neighborhood of $0$, $\beta \in [0,1]$ and $s\in\Na$. If for each $h'\in\{0,\dots,min\{s,k\}\}$, $\dd{^{h+h'}e}{\z^h\partial z^{h'}}(0)=0$ for all $h\in\{0,\dots,k-h'\}$ and $m\in \{1,\dots,s\}$, then the map $\varpi_m$ defined by $\varpi_m=\frac{\z^m}{z^m}e(z)$ on $U-\{0\}$ has a $C^{k,\beta}_{loc}$-extension on $U$, such that for each $h'\in\{0,\dots,min\{s-m,k\}\}$, we have $\dd{^{h+h'}\varpi_m}{\z^h\partial z^{h'}}(0)=0$ for all $h\in\{0,\dots,k-h'\}$.
\end{lemma}
\begin{proof}
	Let us prove it by induction on $m$. If $m=1$, by items $(ii)$ and $(iii)$ of lemma \ref{lemext1} we have $\dd{^j\varpi_{1}}{z^j}(0)=0$ and $\dd{^j\varpi_{1}}{\z^j}(0)=0$ for all $j\in\{0,\dots,k\}$. Thus, it remains to be proven that $\dd{^{h+h'}\varpi_1}{\z^h\partial z^{h'}}(0)=0$ for $2\leq h+h'\leq k$ with $h'\in\{0,\dots,min\{s-1,k\}\}$, which follows from the item $(i)$ of lemma \ref{lemext1}. Now, let us suppose the assertion is true for $m=n$ with $1\leq n<s$. By the inductive hypothesis, the map $\varpi_n$ has a $C^{k,\beta}_{loc}$-extension on $U$, such that for each $h'\in\{0,\dots,min\{s-n,k\}\}$, we have $\dd{^{h+h'}\varpi_n}{\z^h\partial z^{h'}}(0)=0$ for all $h\in\{0,\dots,k-h'\}$. Since $\varpi_{n+1}(z)=\frac{\z}{z}\varpi_{n}(z)$, by lemma \ref{lemext1}, we have that $\varpi_{n+1}$ has a $C^{k,\beta}_{loc}$-extension such that 
	\begin{align}\label{inductionaux}
		\dd{^j\varpi_{n+1}}{z^j}(0)=0\text{ and }\dd{^j\varpi_{n+1}}{\z^j}(0)=0
	\end{align}
	for all $j\in\{0,\dots,k\}$. So, we can assume $k\geq 2$; otherwise, by (\ref{inductionaux}), the assertion for $m=n+1$ is true. Now, if $h'\in \{0,\dots,min\{s-(n+1),k\}\}$, $h\in\{0,\dots,k-h'\}$, $1\leq h'\leq k-1$ and $1\leq h$, then $0\leq h'+1\leq min\{s-n,k\}$. Therefore, by item $(i)$ of lemma \ref{lemext1}, we have $(h'+1)\dd{^{h+h'}\varpi_{n+1}}{\z^h\partial z^{h'}}(0)=h\dd{^{h+h'}\varpi_n}{\z^{h-1}\partial z^{h'+1}}(0)=0$. Thus, putting this together with $(\ref{inductionaux})$, it holds that for each $h'\in\{0,\dots,min\{s-(n+1),k\}\}$, we have $\dd{^{h+h'}\varpi_m}{\z^h\partial z^{h'}}(0)=0$ for all $h\in\{0,\dots,k-h'\}$ and therefore the assertion holds for $m=n+1$.     
\end{proof}
\begin{lemma}\label{lemext3}
	Let $e:U \to \Co$ be a complex $C^{r,\beta}_{loc}$-function, $r\in \Na$, $U\subset \Co$ an open convex neighborhood of $0$, $\beta \in [0,1]$, $s\in\Na$ and $k$ an integer such that $1\leq k\leq r$. If for each $h'\in\{1,\dots,min\{s,k\}\}$, $\dd{^{h+h'}e}{\z^h\partial z^{h'}}(0)=0$ for all $h\in\{0,\dots,k-h'\}$, then the map $\varpi$ defined by  (\ref{function}) on $U-\{0\}$ has a $C^{k,\beta}_{loc}$-extension on $U$ if and only if $\dd{^je}{\z^j}(0)=0$ for all $0\leq j \leq k$. Furthermore, if $\varpi$ has a $C^{k,\beta}_{loc}$-extension on $U$, then it holds that $\dd{^j\varpi}{\z^j}(0)=0$ for all $0\leq j\leq k$. 
\end{lemma}
\begin{proof}
	If $\dd{^je}{\z^j}(0)=0$ for all $0\leq j \leq k$, then for each $h'\in\{0,\dots,min\{s,k\}\}$, $\dd{^{h+h'}e}{\z^h\partial z^{h'}}(0)=0$ for all $h\in\{0,\dots,k-h'\}$ and by lemma \ref{lemext2}, the map $\varpi$ has a $C^{k,\beta}_{loc}$-extension on $U$. Now, let us prove the converse by induction on $k$. If $k=1$, then $\dd{e}{z}(0)=0$ and by lemma \ref{holderl1} $e(0)=0$. As $\dd{\varpi}{\z}(z)=\frac{s\z^{s-1}}{z^s}e(z)+\frac{\z^s}{z^s}\dd{e}{\z}(z)$ on $U-\{0\}$, then 
	\begin{align}\label{equaux3}
		\dd{\varpi}{\z}(z)-s\frac{\z^{s-1}}{z^{s-1}}\int_{0}^{1}\dd{e}{z}(tz)dt=\frac{\z^s}{z^s}(s\int_{0}^{1}\dd{e}{\z}(tz)dt+\dd{e}{\z}(z)).
	\end{align}
	Since $\dd{e}{z}(0)=0$, by lemma \ref{holderl1} the LHS of (\ref{equaux3}) has a $C^{0,\beta}_{loc}$-extension on $U$ and therefore $\dd{e}{\z}(0)=0$. Thus, the converse is true for $k=1$. Let us suppose this is true for $k=n$ with $n<r$. If for each $h'\in\{1,\dots,min\{s,n+1\}\}$, $\dd{^{h+h'}e}{\z^h\partial z^{h'}}(0)=0$ for all $h\in\{0,\dots,n+1-h'\}$ and the map $\varpi$ has a $C^{n+1,\beta}_{loc}$-extension on $U$, then setting the $C^{r-1,\beta}_{loc}$-map $e'(z)=\dd{e}{z}(z)$, it holds that for each $h'\in\{0,\dots,min\{s-1,n\}\}$, $\dd{^{h+h'}e'}{\z^h\partial z^{h'}}(0)=0$ for all $h\in\{0,\dots,n-h'\}$. Therefore, by lemma \ref{lemext2} the LHS of (\ref{equaux3}) has a $C^{n,\beta}_{loc}$-extension on $U$. On the other hand, setting the $C^{r-1,\beta}_{loc}$-map $e''(z)=(s\int_{0}^{1}\dd{e}{\z}(tz)dt+\dd{e}{\z}(z))$, we have that for each $h'\in\{1,\dots,min\{s,n\}\}$, $\dd{^{h+h'}e''}{\z^h\partial z^{h'}}(0)=0$ for all $h\in\{0,\dots,n-h'\}$. Then, by the inductive hypothesis $\dd{^je''}{\z^j}(0)=0$ for all $j\in\{0,\dots,n\}$, which is equivalent to having $\dd{^je}{\z^j}(0)=0$ for all $j\in\{1,\dots,n+1\}$. Also, as $\varpi$ has a continuous extension on $U$, $e(0)=0$ by lemma \ref{holderl1} and the assertion is true for $k=n+1$. The last part follows from lemma \ref{lemext2}.   
\end{proof}
\begin{lemma}\label{lemext4}
	Let $e:U \to \Co$ be a complex $C^{r,\beta}_{loc}$-function, $r\in \Na\cup \{\infty,\omega\}$, $\beta \in[0,1]$, $U\subset \Co$ an open convex neighborhood of $0$ and $s\in\Na\cup \{0\}$. The following assertions hold:
	\begin{enumerate}[label=(\roman*)]
		\item If $r\in\Na$, $r\geq s+1$ and for each $h'\in\{0,\dots,s\}$, $\dd{^{h+h'}e}{\z^h\partial z^{h'}}(0)=0$ for all $h\in\{0,\dots,r-h'\}$, then $e(z)=z^{s+1}c(z)$ for some $C^{r-(s+1),\beta}_{loc}$-map $c:U\to \Co$ with $(s+1)!c(0)=\dd{^{s+1}e}{z^{s+1}}(0)$. 
		\item If $r\in \{\infty, \omega\}$ and for each $h'\in\{0,\dots,s\}$, $\dd{^{h+h'}e}{\z^h\partial z^{h'}}(0)=0$ for all $h\in\{0\}\cup \Na$, then $e(z)=z^{s+1}c(z)$ for some $C^{r,\beta}_{loc}$-map $c:U\to \Co$. 
	\end{enumerate}
	
\end{lemma}	
\begin{proof}
	The case $r=\infty$ follows from item $(i)$. Let us first prove the existence of $c$ in the other cases by induction on $s$. If $s=0$ and $r\in\Na$, by lemma \ref{lemext1} the map $c(z)=\int_{0}^{1}\dd{e}{z}(tz)dt+\frac{\z}{z}\int_{0}^{1}\dd{e}{\z}(tz)dt$ has a $C^{r-1,\beta}_{loc}$-extension and hence $e(z)=zc(z)$ as we wished. Let us suppose the assertion is true for all the non-negative integer $s$ less than or equal to $n\in\Na\cup \{0\}$, where $n< r$. If $(n+1)+1\leq r$ and for each $h'\in\{0,\dots,n+1\}\}$, $\dd{^{h+h'}e}{\z^h\partial z^{h'}}(0)=0$ for all $h\in\{0,\dots,r-h'\}$, by lemma \ref{lemext2} the map $c'(z)=\int_{0}^{1}\dd{e}{z}(tz)dt+\frac{\z}{z}\int_{0}^{1}\dd{e}{\z}(tz)dt$ has a $C^{r-1,\beta}_{loc}$-extension on $U$, such that for each $h'\in\{0,\dots,n\}$, $\dd{^{h+h'}c'}{\z^h\partial z^{h'}}(0)=0$ for all $h\in\{0,\dots,r-1-h'\}$. As $r-1\geq n+1$, by the inductive hypothesis $c'(z)=z^{n+1}c(z)$ for some $C^{r-1-(n+1),\beta}_{loc}$-map $c(z)$ on $U$. We conclude that $e(z)=z^{(n+1)+1}c(z)$ with $c(z)$ being a $C^{r-(n+1+1),\beta}_{loc}$ and the assertion holds for $s=n+1$. 
	
	Now, let us prove the case $r=\omega$ of item (ii). If $s=0$, consider the complexification of $e(x_1,x_2)$ denoted by $D(z_1,z_2)$, where $z_1=x_1+ix_1'$ and $z_2=x_2+ix_2'$. This map is holomorphic on some polydisc containing the origin of $\Co\times \Co$ and satisfies $D(x_1,x_2)=e(x_1,x_2)$. By the Weierstrass preparation theorem (theorem  6.1.1 in \cite{Ho1}), there exist holomorphic maps $Q(z_1,z_2)$, $R(z_2)$ defined on open neighborhoods of the origin, such that $D(z_1,z_2)=(z_1+iz_2)Q(z_1,z_2)+R(z_2)$. Substituting $x_1'=x_2'=0$ in this equality, we get that $e(z)=zQ(x_1,x_2)+R(x_2)$ with $Q(x_1,x_2)$, $R(x_2)$ being complex $C^{\omega}$-maps on neighborhoods of the origin. Since $\dd{^je}{\z^j}(0)=0$ for all $j\in\{0\}\cup \Na$, then all the derivatives of $R(x_2)$ at the origin are $0$ and therefore $e(z)=zQ(x_1,x_2)$ on an open neighborhood of $0$. It follows that $\frac{e(z)}{z}$ has a $C^\omega$-extension on $U$ and the assertion holds. Suppose the assertion is true for all the non-negative integer $s$ less than or equal to $n\in\Na\cup \{0\}$. If for each $h'\in\{0,\dots,n+1\}\}$, $\dd{^{h+h'}e}{\z^h\partial z^{h'}}(0)=0$ for all $h\in\{0\}\cup \Na$, then $e(z)=zc_0(z)$ for some $C^\omega$-map $c_0$ on $U$. As $\dd{^je}{z^j}=j\dd{^{j-1}c_0}{z^{j-1}}+z\dd{^{j}c_0}{z^{j}}$, then for each $h'\in\{0,\dots,n\}$, $\dd{^{h+h'}c_0}{\z^h\partial z^{h'}}(0)=0$ for all $h\in\{0\}\cup \Na$ and by inductive hypothesis $c_0(z)=z^{n+1}c(z)$ for some $C^\omega$-map $c$ on $U$. Therefore, $e(z)=z^{(n+1)+1}c(z)$ with $c(z)$ being a $C^{\omega}$-map. For the point condition in item $(i)$, on a neighborhood of $0$ just write $e$ in its Taylor expansion using $\dd{}{z}$, $\dd{}{\z}$ as follows $$e(z)=\sum\limits_{0 \leq j + k \leq s+1} \frac{1}{j!k!}\dd{^{j+k}e}{\z^k\partial z^{j}}(0) z^j\bar{z}^k + o(|z|^{s+1})=\frac{1}{(s+1)!}\dd{^{s+1}e}{z^{s+1}}(0) z^{s+1}+ o(|z|^{s+1}).$$
	Hence, $(s+1)!c(0)=\lim_{z\to 0}(s+1)!\frac{e(z)}{z^{s+1}}=\dd{^{s+1}e}{z^{s+1}}(0)$.  
\end{proof}
\begin{theorem}\label{Theo2}
	Let $f:M^2\to N^n$ be a $C^{r+1,\beta}_{loc}$-map, $r\in\Na$ (resp. $r\in \{\infty,\omega\}$), $\beta \in [0,1]$, $s\in \Na$ such that $s\leq r$ and $p\in M^2$. If $(U,x)$, $(V,y)$ are coordinate systems at $p$, $f(p)$ respectively, such that $x(p)=0$, $y(f(p))=0$, 
	\begin{align}\label{conditionmain2}
		\dd{^{s+1}\hat{f}}{z^{s+1}}(0)=(\frac{(s+1)!}{2},-i\frac{(s+1)!}{2},0,\dots,0)
	\end{align} and for each $h'\in\{0,\dots,s-1\}$, we have that $\dd{^{h+h'+1}\hat{f}}{\z^h\partial z^{h'+1}}(0)=0$ for all $h\in\{0,\dots,r-h'\}$ (resp. for all $h\in \{0\}\cup \Na$), then $p$ is a branch point of order $s$ and there exists an open neighborhood $U'\subset U$ of $p$ such that $(U',x)$, $(V,y)$ are $C^{r-s,\beta}_{loc}$-branch coordinates at $p$. If additionally, $r\geq (s+1)(s+2)/2$ then $U'$ can be shrunk in such way that $(U',x)$, $(V,y)$ are $C^{r-(s+1)(s+2)/2+1,\beta}_{loc}$-regular branch coordinates at $p$. 
\end{theorem}
\begin{proof}
	Since for each $h'\in\{0,\dots,s-1\}$, we have $\dd{^{h+h'}(\dd{\hat{f}}{z})}{\z^h\partial z^{h'}}(0)=0$ for all $h\in\{0,\dots,r-h'\}$, by item $(i)$ of lemma \ref{lemext4}, we have that $\dd{\hat{f}}{z}(z)=z^sc(z)$ for some $C^{r-s,\beta}_{loc}$-map $c:x(U)\to \Co^n$ with $s!c(0)$ equal to the RHS of (\ref{conditionmain2}). Then, setting $d=\frac{c}{s+1}$, $d'_1=2d_1-1$ and $d'_2=i2d_2-1$ the condition (\ref{cond1}) and (\ref{cond3}) are satisfied. As a consequence, on some open neighborhood $U'\subset U$ of $p$, the conditions (\ref{frontalcond}), (\ref{quasiregularcond}) are satisfied and therefore $(U',x)$, $(V,y)$ are $C^{r-s,\beta}_{loc}$-branch coordinates at $p$. By lemma \ref{equivalencebranchcoor}, $p$ is a branch point of order $s$. The last part follows from lemma \ref{relations}. 
\end{proof}
\section{The tangent bundle of a branch immersion}\label{sectiontangentbundle}
This section introduces some terminologies and properties about a class of maps with a well-defined tangent bundle that includes branch immersions. Then, we will treat maps with branch points in this context. Using the branch coordinates, we will show how a tangent bundle for a branch immersion can be constructed with certain regularity. Several results from this section will be required later. We will often use the following operations and properties.
\begin{definition}
	\normalfont Let $E$ be a vector bundle of rank $n$ over $M^m$, $V\subset M^m$ an open set,  $\mathcal{U}=(\UC_1,\dots,\UC_k)\in \Gamma(E\rvert_V)^k$ and $\mathcal{E}$ a local frame field for $E$ over $V$, we will call the matrix-valued map $\mathcal{E}^*\mathcal{U}\in C(V,\mathcal{M}_{n\times k}(\Re))$ defined by $(\mathcal{E}^*\mathcal{U})_{ij}:=(\theta_i(\UC_j)),$ the {\it matrix field of $\mathcal{U}$ in $\mathcal{E}$}, where $\theta=(\theta_1,\dots,\theta_n)$ is the coframe field dual to $\mathcal{E}$.  
\end{definition}

\begin{definition}
	\normalfont Let $E$ be a vector bundle of rank $n$ over $M^m$, $V\subset M^m$ an open set and $(s,k)\in \amsmathbb{N}\times \Na$. If $\mathbb{B}\in C(V,\mathcal{M}_{k\times s}(\Re))$ and $\mathcal{U}=(\UC_1,\dots,\UC_k)\in \Gamma(E\rvert_V)^k$, we define the right product $\cdot:\Gamma(E\rvert_V)^k \times C(V,\mathcal{M}_{k\times s}(\Re)) \to \Gamma(E\rvert_V)^{s}$ by $(\mathcal{U}\cdot\mathbb{B})_j:=\sum_{i=1}^{k}\mathbb{B}_{ij}\UC_i.$	
\end{definition}
The notation $\mathcal{E}^*\mathcal{U}$ can be seen as an operation $^*$ between local frame fields and $k$-tuples of local sections of $E$. Since the operations $^*$, $\cdot$ act from different sides and objects, we can omit its symbols and write $\mathcal{E}\mathcal{U}$, $\mathcal{U}\mathbb{B}$ instead of $\mathcal{E}^*\mathcal{U}$, $\mathcal{U}\cdot\mathbb{B}$ respectively. Next, we will show some of the basic properties of these operations.

\begin{lemma}\label{basic}
	Let $E$ be a vector bundle of rank $n$ over $M^m$, $V\subset M^m$ an open set, $\mathcal{E}$ a local frame field for $E$ over $V$ and $(s,k)\in\Na\times \Na$. If $\mathcal{U},\mathcal{V}\in \Gamma(E\rvert_V)^k$ and $\mathbb{F}\in C(V,\mathcal{M}_{k\times s}(\Re))$ then the following properties hold:
	\begin{enumerate}[label=(\roman*)]
		\item $\mathcal{E}(\ell\mathcal{U}+\mathcal{V})=\ell(\mathcal{E}\mathcal{U})+\mathcal{E}\mathcal{V}$ for any $\ell\in C(V,\Re)$,
		\item the operation $\cdot$ is bilinear, where the spaces $\Gamma(E\rvert_V)^k$ and $C(V,\mathcal{M}_{k\times s}(\Re))$ are seen as $C(V,\Re)$-modules,
		\item $\mathcal{E}\mathcal{E}=\id_n$,
		\item $\mathcal{E}(\mathcal{E}\mathcal{U})=\mathcal{U}$,
		\item the components of $\mathcal{U}$ at $q\in V$ are linearly independent if and only if $rank(\mathcal{E}\mathcal{U})=k$ at $q$,
		\item $\mathcal{U}\id_k=\mathcal{U}$,
		\item If $h\in \amsmathbb{N}$ and $\mathbb{A}\in C(V,\mathcal{M}_{s\times h}(\Re))$, then $(\mathcal{U}\mathbb{F})\mathbb{A}=\mathcal{U}(\mathbb{F}\mathbb{A})$,
		\item $(\mathcal{E}\mathcal{U})\mathbb{F}=\mathcal{E}(\mathcal{U}\mathbb{F})$,	
		\item If the components of $\mathcal{U}$ at $q\in V$ are linearly independent and $\mathbb{A}\in C(V,\mathcal{M}_{k\times k}(\Re))$, then the components of $\mathcal{U}\mathbb{A}$ at $q$ are linearly independent if and only if $rank(\mathbb{A})=k$ at $q$,
		\item $(\mathcal{E}\mathbb{A}^{-1})\mathcal{U}=\mathbb{A}\mathcal{E}\mathcal{U}$ for any $\mathbb{A}\in C(V,GL(n))$.
	\end{enumerate}
	
\end{lemma} 
\begin{proof}
	Items $(i)$ to $(viii)$ follow immediately from the definitions. To prove $(ix)$, observe that by item $(v)$ and $(viii)$, the components of $\mathcal{U}\mathbb{A}$ at $q$ are linearly independent if and only if $k=rank(\mathcal{E}(\mathcal{U}\mathbb{A}))=rank((\mathcal{E}\mathcal{U})\mathbb{A})$ at $q$. By $(v)$ $rank(\mathcal{E}\mathcal{U})=k$ at $q$, then $k=rank(\mathcal{E}(\mathcal{U}\mathbb{A}))$ at $q$ if and only if $rank(\mathbb{A})=k$ at $q$ and the result follows. To prove item $(x)$, by $(ix)$ we have that $\mathcal{E}\mathbb{A}^{-1}$ is a local frame field for $E$, then by $(iii)$ $(\mathcal{E}\mathbb{A}^{-1})(\mathcal{E}\mathbb{A}^{-1})=\id_n$. Multiplying by the right side with $\mathbb{A}$ and using items $(vii)$,$(viii)$, we get $(\mathcal{E}\mathbb{A}^{-1})\mathcal{E}=\mathbb{A}$ and thus by item $(iv)$ we have $\mathbb{A}\mathcal{E}\mathcal{U}=((\mathcal{E}\mathbb{A}^{-1})\mathcal{E})\mathcal{E}\mathcal{U}=\mathcal{E}\mathbb{A}^{-1}(\mathcal{E}(\mathcal{E}\mathcal{U}))=(\mathcal{E}\mathbb{A}^{-1})\mathcal{U}$.
\end{proof}
\begin{definition}\label{frontal}
	\normalfont Let $f:M^m \to N^n$ be a $C^{r,\beta}_{loc}$-map, $r\in \Na\cup\{\infty,\omega\}$ and $\beta \in [0,1]$. We call $f$ a {\it $C^{r,\beta}_{loc}$-frontal} if for each $p\in M^m$, there exist an open neighborhood $V_p$ of $p$ and linearly independent $C^{r-1,\beta}_{loc}$-vector fields $\WC^p_1, \WC^p_2,\dots, \WC^p_m$ along $f$, over $V_p$, such that $f_*T_qM^m \subset span(\WC^p_1(q),\WC^p_2(q),\dots,\WC^p_m(q))$ for all $q\in V_p$. 
\end{definition}

There are several ways to define frontals; the interested reader can easily verify the equivalence between this definition and the given in the survey article (cf.\cite{ishifrontal2}; see lemma 7.5) for the cases $r\in\{\infty,\omega\}$. The definition presented here and later characterizations will be more convenient for our analytical and geometrical purposes. So, the reader will not require prior knowledge about frontals. To prove that a $C^{r,\beta}_{loc}$-map is a $C^{r,\beta}_{loc}$-frontal, it is enough to show that the conditions in definition \ref{frontal} hold at singular points. This is because for any coordinate system $(U,x)$ at a regular point, the vector fields $f_*\frac{\partial}{\partial x_1},\dots,f_*\frac{\partial}{\partial x_m}$ provide the condition. 

In this paper, whenever we mention that $f$ is a $C^{r,\beta}_{loc}$-frontal, we will assume that $\Sigma(f)$ has an empty interior. This condition is equivalent to having the closure of $Reg(f)$ equal to $M^m$. In the literature, frontals with this condition are called {\it proper frontals} (cf.\cite{ishifrontal2}), but we shall omit this term for the sake of brevity. 

Suppose $\WC^p_1, \WC^p_2,\dots, \WC^p_m$ are the vector fields along $f$ over $V_p$ in the last definition. In that case, we can consider the open cover $C=\{V_p\}_{p\in M^m}$ of $M^m$ and define $$D(C)=\bigcup\limits_{p \in M} span(\WC^p_1(p),\WC^p_2(p),\dots,\WC^p_m(p)),$$ which is a $C^{r-1,\beta}_{loc}$-vector subbundle of $f^*TN^n$ for $r\geq 2$ and a topological vector subbundle for $r=1$. This subbundle is independent of the open cover $C$. In fact if we take another open cover $\hat{C}=\{\hat{V}_p\}_{p\in M^m}$ of $M^m$ with $\hat{\WC}^p_1, \hat{\WC}^p_2,\dots, \hat{\WC}^p_m$ being the vector fields along $f$ over $\hat{V}_p$, linearly independent for each $p\in M^m$, we have that $span(\hat{\WC}^p_1(p), \hat{\WC}^p_2(p),\dots, \hat{\WC}^p_m(p))=f_*T_pM^m=span(\WC^p_1(p),\WC^p_2(p),\dots,\WC^p_m(p))$
for all $p\in Reg(f)$. To show the equality in the case that $p\in \Sigma(f)$, consider a Riemannian metric $\tilde{g}$ on $N^n$ and let us complete the vector fields $\WC^p_1, \WC^p_2,\dots, \WC^p_m$ with linearly independent $C^{r-1,\beta}_{loc}$-vector fields $\VC^p_{m+1},\VC^p_{m+2},\dots,\VC^p_{n}$ along $f$, defined on an open neighborhood $W\subset V_p\cap \hat{V}_p$ of $p$, in such way that $\VC^p_j(q)\in span(\WC^p_1(q),\WC^p_2(q),\dots,\WC^p_m(q)^\perp$ for all $q\in W$ and $m+1\leq j \leq n$. Since $span(\hat{\WC}^p_1(q), \hat{\WC}^p_2(q),\dots, \hat{\WC}^p_m(q))=f_*T_qM^m=span(\WC^p_1(q),\WC^p_2(q),\dots,\WC^p_m(q))$ holds for all $q\in Reg(f)\cap W$, then for each $j\in\{1,2,\dots,m\}$ we have that $g(\hat{\WC}_j^p(q),\VC_i^p(q))=0$ for all $m+1\leq i \leq n$ and $q\in Reg(f)\cap W$. By density of $Reg(f)$ in $M^m$ and the continuity of the Riemannian metric $\tilde{g}$, the equality $\tilde{g}(\hat{\WC}_j^p(p),\VC_i^p(p))=0$ holds. Therefore, we have $span(\hat{\WC}^p_1(p), \hat{\WC}^p_2(p),\dots, \hat{\WC}^p_m(p))=span(\VC^p_{m+1}(p),\VC^p_{m+2}(p),\dots,\VC^p_{n}(p))^\perp=span(\WC^p_1(p),\WC^p_2(p),\dots,\WC^p_m(p))$, and thus $D(C)=D(\hat{C})$. The vector bundle $T_f:=D(C)$ is called the {\it tangent bundle of $f$} and the fiber at $p\in M^m$ is denoted by $T_{fp}$.

Let $E$ be a vector bundle over $M^m$. If $(V,x)$ is a coordinate system of $M^m$ and $\WC$ is a frame field for $E\vert_V$, we will call the $3$-tuple $\mathscr{V}=(V,x,\WC)$ a {\it framed coordinate system of $E$}. On the other hand, if $f:M^m\to N^n$ is a $C^{r,\beta}_{loc}$-frontal with $r\in \Na\cup\{\infty,\omega\}$ and $\beta \in[0,1]$, we will call {\it frontal system of $f$} a $2$-tuple $(\VS,\EC)$, where $\VS=(V,x,\WC)$ is a framed coordinate system of $T_f$, the components $\WC_j$ are $C^{r-1,\beta}_{loc}$-sections of $f^*TN^n\vert_V$ and $\EC$ is a frame field for $f^*TN^n\vert_V$. If $\EC$ is an orthonormal frame field, we will call $(\VS,\EC)$ a {\it canonical system of $f$}. 

\begin{remark}\normalfont
	If $f:M^m\to N^n$ is a $C^{r,\beta}_{loc}$-frontal, $r\in\Na\cup\{\infty,\omega\}$, $\beta\in[0,1]$ and $((V,x,\WC),\mathcal{E})$ is a frontal system of $f$, then it is clear that for every $k$-tuple $\UC=(\UC_1,\dots,\UC_k)$ of $C^{r-1,\beta}_{loc}$-vector fields along $f$ over $V$, we have $\WC\UC\in C^{r-1,\beta}_{loc}(V,\mathcal{M}_{m\times k}(\Re))$ for the case $r\geq 2$. In the case $r=1$, this is also true despite $T_f$ and the dual vector bundle are just topological vector bundles. In fact, by lemma \ref{basic} we have $\EC\UC=(\EC\WC)(\WC\UC)$, also as $rank(\EC\WC)=m$, we have that $(\EC\WC)^t\EC\WC$ is invertible and thus $\WC\UC=((\EC\WC)^t\EC\WC)^{-1}(\EC\WC)^t\EC\UC$. Since $\EC\UC\in C^{r-1,\beta}_{loc}(V,\mathcal{M}_{n\times k}(\Re))$ and $\EC\WC\in C^{r-1,\beta}_{loc}(V,\mathcal{M}_{n\times m}(\Re))$, then $\WC\UC\in C^{r-1,\beta}_{loc}(V,\mathcal{M}_{m\times k}(\Re))$.
\end{remark}

\begin{definition}
	\normalfont Let $f:M^m\to N^n$ be a $C^1$-map and $(V,x,\mathcal{E})$ a framed coordinate system of $f^*TN^n$. We define the matrix-valued map $\J{f}:V\to \mathcal{M}_{n\times m}$ by $\J{f}:=\mathcal{E}f_*\frac{\partial}{\partial x}.$ We call $\J{f}$ the {\it Jacobian matrix field of $f$ in $(V,x,\mathcal{E})$}.
\end{definition}

\begin{lemma}\label{Lambda}
	Let $f:M^m \to N^n$ be a $C^{r,\beta}_{loc}$-map, $r\in \Na\cup\{\infty,\omega\}$ and $\beta \in [0,1]$. Then $f$ is a $C^{r,\beta}_{loc}$-frontal if and only if for each $p\in M^m$ there exist a framed coordinate system $(V_p,x,\mathcal{E})$ of $f^*TN^n$ at $p$ and $C^{r-1,\beta}_{loc}$-maps $\mathbb{W}:V_p\to \mathcal{M}_{n\times m}(\Re)$, $\mathds{J}':V_p\to \mathcal{M}_{m\times m}(\Re)$ such that:
	\begin{enumerate}[label=(\roman*)]
		\item $rank(\mathbb{W}(p))=m$ for all $p \in V_p$ and
		\item $\J{f}=\mathbb{W}\mathds{J}'$ on $V_p$.
	\end{enumerate}  
\end{lemma}
\begin{proof}
	If $f$ is a $C^{r,\beta}_{loc}$-frontal, let $((V_p,x,\WC),\mathcal{E})$ be a frontal system of $f$ at $p$ and since $f_*\frac{\partial}{\partial x}\in \Gamma(T_f\vert_ {V_p})^m$, we can consider the $C^{r-1,\beta}_{loc}$-maps $\mathds{J}'=\WC (f_*\frac{\partial}{\partial x})$ and $\mathbb{W}=\mathcal{E}\WC$. By lemma \ref{basic} $f_*\frac{\partial}{\partial x}=\WC\mathds{J}'$ and operating by the left side with $\mathcal{E}$ we get $\J{f}=\mathbb{W}\mathds{J}'$. Observe that $rank(\mathbb{W})=m$ by item $(v)$ of lemma \ref{basic}. For the converse, if for each $p\in M$, there exist a framed coordinate system $(V_p,x,\mathcal{E})$ of $f^*TN^n$ at $p$ and matrix-valued maps $\mathbb{W}$, $\mathds{J}'$ satisfying the hypothesis, let us set $\WC=\mathcal{E}\mathbb{W}$ whose components $\WC_j$ are linearly independent by item $(v)$ of lemma \ref{basic}. Thus, $f_*\frac{\partial}{\partial x}=\mathcal{E}\J{f}=\mathcal{E}\mathbb{W}\mathds{J}'=\WC\mathds{J}'$. We conclude that $f_*T_qM^m=span(f_*\frac{\partial}{\partial x_1}\rvert_q,\dots,f_*\frac{\partial}{\partial x_m}\rvert_q) \subset span(\WC_1(q),\dots,\WC_m(q))$ for all $q\in V_p$. 	
\end{proof}

\begin{corollary}\label{corbasic}
	Let $f:M^m \to N^n$ be a $C^{r,\beta}_{loc}$-map, $r\in \Na\cup\{\infty,\omega\}$ and $\beta \in [0,1]$. Then $f$ is a $C^{r,\beta}_{loc}$-frontal if and only if for each $p\in M^m$ and any pair of coordinate systems $(U,x)$, $(V,y)$ at $p$, $f(p)$ respectively,  there exist $C^{r-1,\beta}_{loc}$-maps $\mathbb{W}:U'\to \mathcal{M}_{n\times m}(\Re)$, $\mathds{J}':U'\to \mathcal{M}_{m\times m}(\Re)$ on an open neighborhood $U'\subset x(U)$ of $x(p)$, such that:
	\begin{enumerate}[label=(\roman*)]
		\item $rank(\mathbb{W}(q))=m$ for all $q \in U'$ and
		\item $\J{\hat{f}}=\mathbb{W}\mathds{J}'$ on $U'$,
	\end{enumerate} 
	where $\J{\hat{f}}$ is the Jacobian matrix of $\hat{f}=y\circ f\circ x^{-1}$ in the classical sense. 
\end{corollary}
\begin{proof}
	Just observe that $\J{f}(p)=\J{\hat{f}}(x(p))$ for all $p\in U$, where $\J{f}$ is the Jacobian matrix field of $f$ in $(U,x,\mathcal{E})$ with $\mathcal{E}=\frac{\partial}{\partial y}\circ f$. Applying lemma \ref{Lambda}, it follows the equivalence. Also, if items $(i)$ and $(ii)$ are satisfied for a particular pair of coordinate systems, these are satisfied for any other.
\end{proof}
\begin{lemma}\label{diffeo}
	Let $f:M^m \to N^n$ be a $C^{r,\beta}_{loc}$-map, $r\in \Na\cup\{\infty,\omega\}$, $\beta \in [0,1]$, $h:\bar{M}^{m}\to M^m$ a $C^{r,\beta}_{loc}$-diffeomorphism and $k:N^n\to \bar{N}^{\bar{n}}$ a $C^{r,\beta}_{loc}$-immersion. It holds that:
	\begin{enumerate}[label=(\roman*)]
		\item $f$ is a $C^{r,\beta}_{loc}$-frontal if and only if $f\circ h$ is a $C^{r,\beta}_{loc}$-frontal,
		
		\item $f$ is a $C^{r,\beta}_{loc}$-frontal if and only if $k\circ f$ is a $C^{r,\beta}_{loc}$-frontal.
	\end{enumerate}
\end{lemma}
\begin{proof}
	For any pair of coordinate systems $(U,x)$, $(V,y)$ at $p\in \bar{M}^m$, $(f\circ h)(p)$ respectively, with $(f\circ h)(U)\subset V$, shrinking $U$ if necessary, we can find a coordinate system $(U',x')$ at $h(p)$, such that $y\circ f\circ h\circ x^{-1}=y\circ f\circ x'^{-1}\circ x'\circ h\circ x^{-1}$, therefore $\J{y\circ f\circ h\circ x^{-1}}=\J{y\circ f\circ x'^{-1}}(c)\J{c}$, where $c=x'\circ h\circ x^{-1}$. Since $\J{c}\in C^{r-1,\beta}_{loc}(x(U),GL(m))$, applying corollary \ref{corbasic}, the equivalence of $(i)$ follows. To prove $(ii)$, since $k$ is an immersion, for each $p\in M^m$ there exist coordinate systems $(V, y)$, $(U', x')$ at $(k\circ f)(p)$, $f(p)$ respectively, such that $y\circ k \circ x'^{-1}$ is the inclusion map. Choosing a coordinate system $(U,x)$ at $p$ such that $f(U)\subset U'$, we have $\begin{pmatrix}
		\id_n\\
		0
	\end{pmatrix}\J{\tilde{f}}=\J{y\circ k \circ x'^{-1}}(\tilde{f})\J{\tilde{f}}=\J{y\circ k\circ f\circ x^{-1}}$, where $\tilde{f}=x'\circ f\circ x^{-1}$. Thus, by corollary \ref{corbasic}, if $f$ is a $C^{r,\beta}_{loc}$-frontal then $k\circ f$ is a $C^{r,\beta}_{loc}$-frontal too. If $k\circ f$ is a $C^{r,\beta}_{loc}$-frontal, there exist maps $\mathbb{W}$, $\mathds{J}'$ as in corollary \ref{corbasic} such that $\J{y\circ k\circ f\circ x^{-1}}=\mathbb{W}\mathds{J}'$ on a open neighborhood $\hat{U}\subset x(U)$ of $x(p)$. Therefore, $\begin{pmatrix}
		\id_n\\
		0
	\end{pmatrix}\J{x'\circ f\circ x^{-1}}=\mathbb{W}\mathds{J}'=\begin{pmatrix}
		\A\\
		\B
	\end{pmatrix}\mathds{J}'$,
	where the sizes of $\A$, $\B$ are $n\times m$, $(\bar{n}-n)\times m$ respectively. From this equality we get that $0=\B\mathds{J}'$ and $\J{y\circ k\circ f\circ x^{-1}}=\A\mathds{J}'$. Observe that $\mathds{J}'(q)$ is invertible if and only if $rank(\J{y\circ k\circ f\circ x^{-1}}(q))=m$, that is, if $q\in Reg(y\circ k\circ f\circ x^{-1})\cap \hat{U}$. Since $Reg(y\circ k\circ f\circ x^{-1})$ is dense in $x(U)$, this implies that $\B=0$ on $\hat{U}$ and as $rank(\mathbb{W})=m$, we have that $rank(\A)=m$. The result follows from corollary \ref{corbasic}.  
\end{proof}
If $f:M^m\to N^n$ is a $C^{1}$-frontal and $(\mathscr{V},\EC)=((V,x,\WC),\mathcal{E})$ a frontal system of $f$, we shall denote the matrix field of $\WC$ in $\mathcal{E}$ by $\mathbb{W}$. We will generally use blackboard bold font in a frame field's letter to denote the corresponding matrix field.
If $N^n=\Re^n$, we usually work with the canonical frame field $\mathcal{E}$ for $f^*\Re^n$ unless otherwise specified. So, in this case, for every local frame field $\WC$ for $T_f$, the symbol $\mathbb{W}$ will denote the matrix field of $\WC$ in the standard frame field of $f^*\Re^n$. Whenever $M^m=U \subset \Re^m$ is an open set, we will work with the standard coordinates. 

\begin{definition}
	\normalfont Let $f:M^m\to N^n$ be a $C^{1}$-frontal and $\VS=(V,x,\WC)$ a framed coordinate system of $T_f$. We define the map $\J{\VS}:=\WC f_*\frac{\partial}{\partial x}$.
\end{definition}
\begin{remark}\label{remarkframeinduced}\normalfont
	\normalfont If $f:M^m\to N^n$ is a $C^1$-frontal, a decomposition of the Jacobian matrix field in a framed coordinate system $(V',x',\mathcal{E}')$ of $f^*TN^n$, like item $(ii)$ of lemma \ref{Lambda}, induces a framed coordinate system $\VS'=(V',x',\mathcal{E}'\mathbb{W})$ of $T_{f}$ such that $\J{\VS'}=\mathds{J}'$. 
\end{remark}
\begin{definition}
	\normalfont Let $f:M^m\to N^n$ be a $C^{r,\beta}_{loc}$-frontal, $r\in\Na\cup\{\infty,\omega\}$, $\beta\in[0,1]$ and $(\mathscr{V},\EC)=((V,x,\WC),\mathcal{E})$ a frontal system of $f$ at $p\in V$. We will call $(\mathscr{V},\EC)$ a {\it normalized system of $f$ at $p$} if the matrix field $\mathbb{W}$ of $\WC$ in $\mathcal{E}$ has the form
	$\W=\begin{pmatrix}
		\mathbb{id}_m\\
		\mathbb{B}
	\end{pmatrix}$, where $\mathbb{B}:V \to \mathcal{M}_{(n-m)\times m}(\Re)$ is a $C^{r-1,\beta}_{loc}$-map such that $\mathbb{B}(p)=0$. 
\end{definition}
\begin{lemma}\label{normalized}
	Let $f:M^m\to N^n$ be a $C^{r,\beta}_{loc}$-frontal, $r\in\Na\cup\{\infty,\omega\}$, $\beta\in[0,1]$ and $((V,x,\WC),\mathcal{E})$ a canonical system of $f$ at $p\in M^m$. There exist an open neighborhood $U'\subset V$ of $p$, $\mathbb{A} \in C^{r-1,\beta}_{loc}(U',GL(m))$ and $\mathbb{Q}\in SO(n)$ such that $((U',x,\WC \mathbb{A}),\mathcal{E}\mathbb{Q})$ is a normalized system of $f$ at $p$. 
\end{lemma}
\begin{proof}
	Let $\WC'$ be a local orthonormal frame field over an open neighborhood $U$ of $p$. We have $\mathbb{A}':=\WC\WC'\in C^{r-1,\beta}_{loc}(U\cap V,GL(m))$ and by lemma \ref{basic} $\WC'=\WC\mathbb{A}'$ on $U'=U\cap V$. Since $\WC'$ is an orthonormal frame field, the vector columns of $\mathbb{W}'=\EC \WC'$ are orthonormal. Let us complete the ordered $m$ vector columns of $\mathbb{W}'(p)$ into a positive orthonormal basis of $\Re^n$ and consider the matrix $\mathbb{Q}$ with these vector columns. As $\mathbb{Q}\in SO(n)$ and this has the form $\begin{pmatrix}
		\W'(p)& \mathbb{O}
	\end{pmatrix}$ with $\mathbb{O}\in \mathcal{M}_{n\times (n-m)}(\Re)$, then $\mathbb{Q}^t\W'(p)=\begin{pmatrix}
		\mathbb{id}_m\\
		0
	\end{pmatrix}.$ Shrinking $U'$ if necessary, we have that $\mathbb{Q}^t\mathbb{W}'=\begin{pmatrix}
		\F\\
		\G
	\end{pmatrix},$ where $\F:U'\to GL(m)$ and $\G:U'\to \mathcal{M}_{n-m\times m}(\Re)$ are $C^{r-1,\beta}_{loc}$-maps, with $\G(p)=0$. Setting $\mathbb{A}=\mathbb{A}'\F^{-1}$, by item $(x)$ of lemma \ref{basic}, we have $\mathcal{E}\mathbb{Q}(\WC \mathbb{A})=\mathbb{Q}^t\mathcal{E}(\WC \mathbb{A}'\F^{-1})=\mathbb{Q}^t\W'\F^{-1}=\begin{pmatrix}
		\F\\
		\G
	\end{pmatrix}\F^{-1}=\begin{pmatrix}
		\mathbb{id}_m\\
		\G\F^{-1}
	\end{pmatrix},$ where $(\G\F^{-1})(p)=0$. Therefore, $((U',x,\WC \mathbb{A}),\mathcal{E}\mathbb{Q})$ is a normalized system of $f$ at $p$. 
\end{proof}

\begin{proposition}\label{prepre}
	Let $f:M^m\to N^n$ be a $C^{r,\beta}_{loc}$-frontal, $r\in \Na\cup\{\infty,\omega\}$, $\beta\in[0,1]$, $p\in M^m$ and  $(U,x)$, $(V,y)$ coordinate systems at $p$ and $f(p)$ respectively. There exist a rigid motion $T:\Re^n \to \Re^n$, an open connected neighborhood $\hat{U}$ of $0=x(p)$ and a  $C^{r-1,\beta}_{loc}$-map $B:\hat{U} \to \mathcal{M}_{n-m\times m}(\Re)$ such that $T\circ\hat{f}=(a_1,\ldots,a_n)$ is equal to the expression 
	\begin{equation}\label{fr}
		(A(x),\int_{0}^{x}B_1\J{A}\cdot d\gamma,\ldots,\int_{0}^{x}B_{n-m}\J{A}\cdot d\gamma),
	\end{equation} where $B_j$ is the $j$-row of $B$,
	\begin{enumerate}[label=(\roman*)]
		\item $A(x)=(a_1,\ldots,a_m)$, $A(0)=0$, $B(0)=0$ and
		\item $\J{A}^t\J{B_j}=\J{B_j}^t\J{A}$ on $\hat{U}$, for all $1\leq j\leq n-m$ in the case $r\geq 2$. 
	\end{enumerate}
	Here $\int_{0}^{x}()\cdot d\gamma$ is the line integral over an arbitrary piecewise smooth curve $\gamma:[t_0,t_1]\to \hat{U}$ with $\gamma(t_0)=0$ and $\gamma(t_1)=x$. Furthermore, if we substitute a $C^{r,\beta}_{loc}$-map $A$ with $r\geq 2$ and a $C^{r-1,\beta}_{loc}$-map $B$ in the formula (\ref{fr}) satisfying $(ii)$, this produces a $C^{r,\beta}_{loc}$-frontal with Jacobian matrix equal to
	\begin{align}\label{Jacob}
		\begin{pmatrix}
			\id_m\\
			B
		\end{pmatrix}\J{A}.
	\end{align} 
\end{proposition}
\begin{proof}
	By lemma \ref{diffeo}, $y\circ f:M^m \to \Re^n$ is a frontal. Let us take a canonical system $(\US,\EC)=((U',x,\WC'),\EC)$ of $y\circ f$ at $p$, with $U'$ an open connected subset of $U$, $\WC'$ a frame field for $T_f$ over on $U'$ and $\EC$ the standard frame field for $f^*\Re^n$. Shrinking $U'$ if necessary, by lemma \ref{normalized}, there exist $\mathbb{A} \in C^{r-1,\beta}_{loc}(U',GL(m))$ and $\mathbb{Q}\in SO(n)$ such that $((U',x,\WC' \mathbb{A}),\mathcal{E}\mathbb{Q})$ is a normalized system of $y\circ f$ at $p$. Setting $\WC=\WC' \mathbb{A}$, $\VS=(U',x,\WC)$ and since $(y\circ f)_*\frac{\partial}{\partial x}=((y\circ f)_*\frac{\partial}{\partial x_1},\dots,(y\circ f)_*\frac{\partial}{\partial x_m})=\WC\J{\VS}$, we have that  $(\mathcal{E}\mathbb{Q})((y\circ f)_*\dd{}{x})=(\mathcal{E}\mathbb{Q})\WC \J{\VS}=\begin{pmatrix}
		\id_m\\
		B'
	\end{pmatrix}\J{\VS}$, for some $C^{r-1,\beta}_{loc}$-map $B'$ on $U'$ with $B'(p)=0$. Then, by item $(x)$ of lemma \ref{basic}, $\mathbb{Q}^{-1}\J{\hat{f}}=\begin{pmatrix}
		\id_m\\
		B'\circ x^{-1}
	\end{pmatrix}(\J{\VS}\circ x^{-1})$. 
	Setting $T(q)=\mathbb{Q}^{-1}q-\mathbb{Q}^{-1}(y\circ f)(p)$ for all $q\in \Re^n$, $B=B'\circ x^{-1}$, $T\circ \hat{f}=(a_1,\dots,a_n)$ and $A=(a_1,\dots,a_m)$, we can conclude that $\J{T\circ \hat{f}}=\begin{pmatrix}
		\id_m\\
		B
	\end{pmatrix}\J{\VS}\circ x^{-1}$ and $\J{\VS}\circ x^{-1}=\J{A}$.
	Integrating, we get the formula (\ref{fr}). Observe that the item $(ii)$ is the integrability condition, and it follows from the fact that $\dd{(B\J{A}^j)}{x_i}=\dd{(B\J{A}^i)}{x_j}$ for all $1\leq i\leq m$ and $1\leq j\leq m$, where $\J{A}^j$ is the $j$-column of $\J{A}$. The last part follows from corollary \ref{corbasic}. 
\end{proof}
\begin{corollary}\label{prepre2}
	Let $f:M^2\to N^n$ be a $C^{r,\beta}_{loc}$-frontal, $r\in \Na\cup\{\infty,\omega\}$, $\beta\in[0,1]$, $p\in M^2$ and  $(U,x)$, $(V,y)$ coordinate systems at $p$ and $f(p)$ respectively. There exist a rigid motion $T:\Re^n \to \Re^n$, an open connected neighborhood $\hat{U}$ of $0=x(p)$ and $C^{r-1,\beta}_{loc}$-maps $b_j:\hat{U} \to \Co$, $1\leq j\leq n-2$, such that $T\circ\hat{f}=(a_1,\ldots,a_n)$ is equal to the expression 
	\begin{equation}\label{fr2}
		(a(z),Re\{\int_{0}^{z}\overline{b}_1\dd{a}{z}+b_1\dd{\overline{a}}{z}dw\},\ldots,Re\{\int_{0}^{z}\overline{b}_{n-2}\dd{a}{z}+b_{n-2}\dd{\overline{a}}{z}dw\}),
	\end{equation} where 
	\begin{enumerate}[label=(\roman*)]
		\item $a(z)=a_1(z)+ia_2(z)$, $a(0)=0$,  $b_j(0)=0$ for all $j\in\{1,\dots,n-2\}$ and
		\item $Im\{\dd{a}{z}\dd{\overline{b}_k}{\z}+\dd{\overline{a}}{z}\dd{b_k}{\z}\}=0$ on $\hat{U}$, for all $k\in\{1,\dots,n-2\}$ in the case $r\geq 2$. 
	\end{enumerate}
	Here $\int_{0}^{z}()dw$ is the complex line integral over an arbitrary piece-wise smooth curve $\gamma:[t_0,t_1]\to \hat{U}$ with $\gamma(t_0)=0$ and $\gamma(t_1)=z$. Furthermore, if we substitute a $C^{r,\beta}_{loc}$-map $a:\hat{U}\to \Co$ with $r\geq 2$ and $C^{r-1,\beta}_{loc}$-maps $b_j:\hat{U} \to \Co$ in the formula (\ref{fr}) satisfying $(ii)$, this produces a $C^{r,\beta}_{loc}$-frontal.
\end{corollary} 
\begin{proof}
	By proposition \ref{prepre}, there exists a rigid motion $T:\Re^n\to \Re^n$ such that $T\circ\hat{f}$ is equal to the expression (\ref{fr}). Consider the maps $A$, $B$ from proposition \ref{prepre} and set the complex maps  $b_j(z):=B_{j1}(z)+iB_{j2}(z)$ for all $j\in \{1,\dots,n-2\}$ and $a(z)=a_1(z)+ia_2(z)$. We can obtain the formula (\ref{fr2}) by rewriting the expressions in the formula (\ref{fr}) in terms of the previous complex maps and complex line integrals. The condition $(ii)$ follows from applying $\dd{}{\z\partial z}$ to the expression (\ref{fr2}) and equaling to $0$ the imaginary part.
\end{proof}  

Let $f:U\to \Re^n$ be a $C^{r,\beta}_{loc}$-frontal, $r\in \Na\cup\{\infty,\omega\}$, $\beta\in[0,1]$ and $U\subset \Re^m$ an open neighborhood of $0$. If $f(0)=0$ and there exist maps $A:U \to \mathcal{M}_{1\times m}(\Re)$ and $B:U \to \mathcal{M}_{n-m\times m}(\Re)$ such that $A=(f_1,\dots,f_m)$, $B$ is a $C^{r-1,\beta}_{loc}$-map with $B(0)=0$ and $\J{f}$ has the form (\ref{Jacob}), we say that $f$ is a {\it normalized frontal} or simply $f$ is {\it normalized}, otherwise we say $f$ is {\it unnormalized}. Observe that, for a normalized frontal $f$, the maps $A$ and $B$ are unique, in fact if there exist $A'$ and $B'$ satisfying these conditions, it is immediate that $A=A'$ and $\begin{pmatrix}
	\id_m\\
	B
\end{pmatrix}\J{A}=\begin{pmatrix}
	\id_m\\
	B'
\end{pmatrix}\J{A'}$.
Since the points in which the matrix $\J{A'}=\J{A}$ is invertible are dense in $M^m$, then $B=B'$. We call the map $A$ the {\it principal part of $f$} and $B$ the {\it co-principal part of $f$}. For $m=2$, the complex map $a$ in corollary \ref{prepre2} will be called the {\it complex principal part of f} and the map $b:\hat{U}\to \Co^{n-2}$ whose components are the maps $b_j$ in the corollary \ref{prepre2} will be called the {\it complex co-principal part of f}. 

If $f$ is unnormalized by proposition \ref{prepre}, shrinking $U$ if necessary, we can turn this into a normalized frontal $T\circ f$ through a rigid motion $T$. We will call $T\circ f$ a {\it normalization of $f$}. We warn that there could be different normalizations. However, these satisfy the next relation. 

\begin{proposition}\label{isometries}
	Let $f:U\to \Re^n$ be a $C^{r,\beta}_{loc}$-frontal, where $U\subset \Re^m$ is an open neighborhood of $0$, $r\in \Na\cup\{\infty,\omega\}$ and $\beta\in[0,1]$. If $T$ and $T'$ are isometries of $\Re^n$, such that $T\circ f$ and $T'\circ f$ are normalizations of $f$ defined on $U$, then $T\circ (T')^{-1}(q)=\M q$ for some matrix \begin{align}\label{ortogonalmatrix}
		\M=\begin{pmatrix}
			\D&0\\
			0&\G
		\end{pmatrix},
	\end{align} where $\D\in O(m)$ and $\G\in O(n-m)$. Also, if $A$, $B$ are the principal part and the co-principal part of $T\circ f$ respectively, and $A'$, $B'$ are the principal part and the co-principal part of $T'\circ f$ respectively, then we have that $A=A'\D^t$ and $B=\G B'\D^t$.
\end{proposition} 
\begin{proof}
	Denoting $F=T\circ f$ and $F'=T'\circ f$, we have that $F=T\circ (T')^{-1}\circ F'$. It implies that $T\circ (T')^{-1}(0)=0$ and since $T$, $T'$ are isometries, then $T\circ (T')^{-1}$ is a linear map, whose representation in the canonical basis of $\Re^n$ is a matrix $\M\in O(n)$. Writing $\M=\begin{pmatrix}
		\D&\mathbb{E}\\
		\F&\G
	\end{pmatrix}$, where $\D\in \mathcal{M}_{m\times m}(\Re)$, $\G\in \mathcal{M}_{(n-m)\times (n-m)}(\Re)$, $\mathbb{E}\in \mathcal{M}_{m\times (n-m)}(\Re)$ and $\mathbb{F}\in \mathcal{M}_{(n-m)\times m}(\Re)$, we have that $\begin{pmatrix}
		\J{A}\\
		B\J{A}
	\end{pmatrix}=\begin{pmatrix}
		\D&\mathbb{E}\\
		\F&\G
	\end{pmatrix}\begin{pmatrix}
		\J{A'}\\
		B'\J{A'}
	\end{pmatrix}$ and $\begin{pmatrix}
		\D^t&\F^t\\
		\mathbb{E}^t&\G^t
	\end{pmatrix}\begin{pmatrix}
		\J{A}\\
		B\J{A}
	\end{pmatrix}=\begin{pmatrix}
		\J{A'}\\
		B'\J{A'}
	\end{pmatrix}$.
	Therefore $\J{A}=\D\J{A'}+\mathbb{E}B'\J{A'}$,  $\J{A'}=\D^t\J{A}+\mathbb{F}^tB\J{A}$ and combining these equations, we get that $\J{A}=(\D+\mathbb{E}B')(\D^t+\F^tB)\J{A}$. Since $\J{A}$ is invertible on $Reg(f)$, which is dense in $U$, we have that $\id_m=(\D+\mathbb{E}B')(\D^t+\F^tB)$ and using that $B(0)=B'(0)=0$, we can conclude that $\D\in O(m)$. As the rows and columns of $\M$ are orthonormal, then $\mathbb{E}=\F=0$ and thus $\G\in O(n-m)$. On the other hand, we have $\J{A}=\D\J{A'}$, which implies $A=\D A'$. Also, $\begin{pmatrix}
		\J{A}\\
		B\J{A}
	\end{pmatrix}=\begin{pmatrix}
		\D&0\\
		0&\G
	\end{pmatrix}\begin{pmatrix}
		\J{A'}\\
		B'\J{A'}
	\end{pmatrix}=\begin{pmatrix}
		\D&0\\
		0&\G
	\end{pmatrix}\begin{pmatrix}
		\D^t\J{A}\\
		B'\D^t\J{A}
	\end{pmatrix}$ on $U$ and eliminating $\J{A}$ from this equality, we can conclude that $B=\G B'\D^t$.
\end{proof}

\begin{remark}\label{complexprincipalpart}\normalfont
	In the last proposition, for $m=2$, if we write the principal parts $A$, $A'$ in the complex forms $a$, $a'$ as in corollary \ref{prepre2} respectively, the equality $A=\D A'$ is equivalent to the equality $a=a'q$, where $q=\D_{11}-i\D_{12}\in S^1\subset \Co$.
\end{remark}
\begin{lemma}\label{normalizedchange}
	Let $f:U\to \Re^n$ be a normalized frontal of class $C^{r,\beta}_{loc}$, where $U\subset \Re^m$ is an open neighborhood of $0$, $r\in \Na\cup\{\infty,\omega\}$ and $\beta\in[0,1]$. If $A$, $B$ are the principal part and the co-principal part of $f$ respectively, and $c:U'\to U$ is a $C^{r,\beta}_{loc}$-diffeomorphism on an open neighborhood $U'\subset \Re^m$ of $0$ with $c(0)=0$, then $f\circ c$ is a normalized frontal whose principal part is $A\circ c$ and co-principal part is $B\circ c$. 
\end{lemma}
\begin{proof}
	By the chain rule, $\J{f\circ c}=\J{f}(c)\J{c}=\begin{pmatrix}
		\id_m\\
		B(c)
	\end{pmatrix}\J{A}(c)\J{c}=\begin{pmatrix}
		\id_m\\
		B(c)
	\end{pmatrix}\J{A\circ c}$. Since $A\circ c=((f\circ c)_1,\dots, (f\circ c)_m)$, $(f\circ c)(0)=0$ and $(B\circ c)(0)=0$, it follows the result.
\end{proof}

\begin{lemma}\label{normalizedfrontal}
	Let $f:U\to \Re^n$ be a frontal of class $C^{r,\beta}_{loc}$, where $U\subset \Re^m$ is an open neighborhood of $0$, $f(0)=0$, $r\in \Na\cup\{\infty,\omega\}$ and $\beta\in[0,1]$. If $f$ is normalized, every decomposition of the Jacobian matrix in the form 
	\begin{align}\label{normalizedcond}
		\J{f}=\begin{pmatrix}
			D'\\
			D
		\end{pmatrix}\mathds{J}'
	\end{align} for some $C^{r-1,\beta}_{loc}$-maps $D',\mathds{J}':U\to \mathcal{M}_{m\times m}(\Re)$ and $D:U\to \mathcal{M}_{(n-m)\times m}(\Re)$ with $rank(\begin{pmatrix}
		D'\\
		D
	\end{pmatrix})=m$, satisfies $D(0)=0$ and $det(D')\neq 0$ on $U$. Conversely, if $\J{f}$ has a decomposition like (\ref{normalizedcond}) satisfying the previous conditions, then $f$ is normalized with co-principal part equal to $D(D')^{-1}$.
\end{lemma}
\begin{proof}
	If $f$ is normalized, then $\J{f}=\begin{pmatrix}
		\id_m\\
		B
	\end{pmatrix}C$ for some $C^{r-1,\beta}_{loc}$-maps $C:U\to \mathcal{M}_{m\times m}(\Re)$ and $B:U\to \mathcal{M}_{(n-m)\times m}(\Re)$ with $B(0)=0$. Let $\EC$ be the standard frame field of $f^*T\Re^n$, $\WC=\EC\begin{pmatrix}
		D'\\
		D
	\end{pmatrix}$ and $\VC=\EC\begin{pmatrix}
		\id_m\\
		B
	\end{pmatrix}$. By remark \ref{remarkframeinduced} $\WC$, $\VC$ are frame fields for $T_f$ over $U$, then by lemma \ref{basic} we have $\WC=\VC(\VC\WC)$ and therefore $\begin{pmatrix}
		D'\\
		D
	\end{pmatrix}=\EC\WC=\EC\VC(\VC\WC)=\begin{pmatrix}
		\id_m\\
		B
	\end{pmatrix}(\VC\WC)$. Since $det(\VC\WC)\neq 0$ by item $(v)$ of lemma \ref{basic}, it follows the result. For the converse, as $\J{f}=\begin{pmatrix}
		D'\\
		D
	\end{pmatrix}\mathds{J}'=\begin{pmatrix}
		\id_m\\
		D(D')^{-1}
	\end{pmatrix}D'\mathds{J}'$, this implies that $D'\mathds{J}'=\J{A}$, where $A=(f_1,\dots,f_m)$.  
\end{proof}
\begin{lemma}\label{complexfrontal}
	Let $f:U\to \Re^n$ be a $C^{r,\beta}_{loc}$-map, $U\subset \Co$ an open set, $r\in\Na\cup\{\infty,\omega\}$ and $\beta\in[0,1]$. If $\dd{f}{z}=cd$, 
	for some $C^{r-1,\beta}_{loc}$-maps $c:U\to \Co$, $d:U \to \Co^n$ with $Re\{d\}$, $Im\{d\}$ being linearly independent on $U$, then $f$ is a $C^{r,\beta}_{loc}$-frontal and there exists a framed coordinate system $\VS=(U,id_U,\WC)$ of $T_f$ such that
	\begin{align}\label{eqcomplexfrontal}
		\W=\begin{pmatrix}
			2Re\{d\}&-2Im\{d\}
		\end{pmatrix},\text{ }\J{\VS}=\begin{pmatrix}
			Re\{c\}&-Im\{c\}\\
			Im\{c\}&Re\{c\}
		\end{pmatrix}.
	\end{align}
\end{lemma}
\begin{proof}
	Setting $Re\{c\}=c_1$, $Im\{c\}=c_2$, $Re\{d\}=d_1$ and $Im\{d\}=d_2$, we have that $\frac{1}{2}\dd{f}{x_1}-\frac{i}{2}\dd{f}{x_2}=c_1d_1-c_2d_2+i(c_1d_2+c_2d_1),$ therefore $\dd{f}{x_1}=c_1(2d_1)+c_2(-2d_2)$ and $\dd{f}{x_2}=-c_2(2d_1)+c_1(-2d_2)$, that is, $$\J{f}=\begin{pmatrix}
		2Re\{d\}&-2Im\{d\}
	\end{pmatrix}\begin{pmatrix}
		Re\{c\}&-Im\{c\}\\
		Im\{c\}&Re\{c\}
	\end{pmatrix}.$$
	Since $rank(\begin{pmatrix}
		2Re\{d\}&-2Im\{d\}
	\end{pmatrix})=2$ on $U$, by corollary \ref{corbasic} $f$ is a $C^{r,\beta}_{loc}$-frontal and this matrix-valued map induces a global frame field $\WC$ for $T_f$ satisfying what we wished.
\end{proof}
\begin{proposition}\label{branchfrontal}
	Let $f:M^2\to N^n$ be a $C^{1}$-map, $r\in\Na\cup\{0,\infty,\omega\}$, $\beta\in[0,1]$ and $p\in M^m$ a branch point. If $(U,x)$, $(V,y)$ are $C^{r,\beta}_{loc}$-branch coordinates at $p$, then $\hat{f}:x(U)\to \Re^n$ is a normalized $C^{r+1,\beta}_{loc}$-frontal with complex co-principal part $b:x(U)\to \Co^{n-2}$ given by $b_j=(d_{j+2}(1+\frac{1}{2}(d_1'+d_2'))-\bar{d}_{j+2}\overline{(d_1'-d_2')})(Re\{(1+d_1')(1+\bar{d}_2')\})^{-1}$ for each $j\in\{1,\dots,n-2\}$, where $d_{j+2}$, $d_1'$, $d_2'$ are de the maps in (\ref{cond1}) and (\ref{cond3}). In particular, every branched immersion is a $C^1$-frontal. 
\end{proposition}
\begin{proof}
	As $\hat{f}$ satisfies (\ref{cond3}) and (\ref{frontalcond}) on $x(U)$, then applying lemma \ref{complexfrontal} and lemma \ref{normalizedfrontal}, we have that $\hat{f}$ is a normalized $C^{r+1,\beta}_{loc}$-frontal with co-principal part $B$ given by $\begin{pmatrix}
		\id_m\\
		B
	\end{pmatrix}=\W\begin{pmatrix}
		2Re\{d_1\}&-2Im\{d_1\}\\
		2Re\{d_2\}&-2Im\{d_2\}
	\end{pmatrix}^{-1}$, where $\W$ is the map in (\ref{eqcomplexfrontal}) with $d$ being the $C^{r,\beta}_{loc}$-map in (\ref{cond1}).  Thus, by direct computations the complex co-principal part is  $b_j=(d_{j+2}(1+\frac{1}{2}(d_1'+d_2'))-\bar{d}_{j+2}\overline{(d_1'-d_2')})(Re\{(1+d_1')(1+\bar{d}_2')\})^{-1}$ for each $j\in\{1,\dots,n-2\}$. The last part follows from lemma \ref{lembranch} and lemma \ref{complexfrontal}.
\end{proof}
\begin{corollary}
	Let $f:M^2\to N^n$ be a $C^{r+2,\beta}_{loc}$-branched conformal immersion, $r\in \Na\cup \{0,\infty,\omega\}$, $\beta \in(0,1)$ and $\tilde{g}$ a $C^{r+1,\beta}_{loc}$-Riemannian metric on $N^n$. If the mean curvature vector $\HC$ of $f$ has a $C^{r,\beta}_{loc}$-extension to $M^2$, then $f$ is a $C^{r+2,\beta}_{loc}$-frontal. In particular, the tangent bundle $T_f$ is a $C^{r+1,\beta}_{loc}$-vector bundle.
\end{corollary}
\begin{proof}
	It follows from theorem \ref{Theo1} and proposition \ref{branchfrontal}.
\end{proof}
\section{Representation formulas for branched immersions.}\label{sectionrepresentation}
Here, we will obtain a general representation formula for maps with branch points expressed in regular branch coordinates. Then, we show the properties of the distinguished coefficient, the distinguished degree, and the index of a branch point.  
\begin{lemma}\label{lemform}
	Let $b:U \to \Co$ be a map of class $C^{r,\beta}_{loc}$ on the open neighborhood $U\subset \Re^2$ of $0$, with $r\in \amsmathbb{N}\cup\{\omega,\infty\}$, $\beta \in (0,1)$ and $s\in \amsmathbb{N}\cup \{0\}$. We have that 
	\begin{align}\label{eql}
		\dd{b}{\z}(z)=l(z)\z^s
	\end{align} for some real $C^{r-1,\beta}_{loc}$-function $l$ on an open neighborhood $V_1\subset U$ of $0$ if and only if  
	\begin{align}\label{eqt}
		b(z)=\sum_{j=0}^{s}(-1)^{s-j}\frac{s!}{j!}\frac{\partial^{j+s+1}\varphi}{\partial \z^{j}\partial z^{s+1}}(z)\z^{j}+F(z)
	\end{align} for some real $C^{2s+r+1,\beta}_{loc}$-function $\varphi$ on an open neighborhood $V_2\subset U$ of $0$ and $F(z)$ a holomorphic function on $V_2$.   
	
\end{lemma}
\begin{proof}
	If we have (\ref{eqt}), then 
	\begin{align*}
		\dd{b}{\z}(z)=&\sum_{j=1}^{s}(-(-1)^{s-(j+1)}\frac{s!}{j!}\frac{\partial^{j+1+s+1}\varphi}{\partial \z^{j+1}\partial z^{s+1}}(z)\z^{j}+(-1)^{s-j}\frac{s!}{(j-1)!}\frac{\partial^{j+s+1}\varphi}{\partial \z^{j}\partial z^{s+1}}(z)\z^{j-1})\\
		&+(-1)^s s!\frac{\partial^{s+2}\varphi}{\partial \z\partial z^{s+1}}(z)
	\end{align*} where the finite series is a telescoping sum. Therefore, 
	\begin{align*}
		\dd{b}{\z}&=\frac{\partial^{2s+2}\varphi}{\partial \z^{s+1}\partial z^{s+1}}(z)\z^s+(-1)^{s-1} s!\frac{\partial^{s+2}\varphi}{\partial \z\partial z^{s+1}}(z)+(-1)^s s!\frac{\partial^{s+2}\varphi}{\partial \z\partial z^{s+1}}(z)\\
		&=\frac{1}{4^{s+1}}\Delta^{s+1}\varphi(z)\z^s
	\end{align*} on $V_2$, where $\frac{1}{4^{s+1}}\Delta^{s+1}\varphi(z)$ is a real $C^{r-1,\beta}_{loc}$-function. On the other hand, if we have (\ref{eql}), it is well known that the Poisson's equation $\Delta f_1=4^{s+1}l$ has a real solution $f_1$ of class $C^{r+1,\beta}_{loc}$ on an open neighborhood $V_1'\subset V_1$ of $0$ (cf.\cite{GT}). Therefore, considering local solutions recursively of $\Delta f_n=f_{n-1}$ for $2\leq n\leq s+1$, we obtain a real $C^{2s+r+1,\beta}_{loc}$-function $\varphi=f_{s+1}$ on an open neighborhood $V_2\subset U$ of $0$, which satisfies $\Delta^{s+1}f_{s+1}=4^{s+1}l$ on $V_2$. Constructing a function $b'$ using the function $\varphi$ obtained above and $F=0$ in the formula (\ref{eqt}), by the computations above, we have that $\dd{b'}{\z}(z)=l(z)\z^s$. Therefore, $\dd{(b-b')}{\z}=0$ on $V_2$ and thus $b-b'$ is holomorphic on $V_2$.      
\end{proof}

\begin{theorem}\label{representbranchcoord}
Let $f:M^2\to N^n$ be a $C^{r+1,\beta}_{loc}$-map, $r\in \Na\cup \{\infty,\omega\}$, $\beta\in (0,1)$ and $p\in M^2$ a branch point of order $s$. If $(U,x)$, $(V,y)$ are $C^{r,\beta}_{loc}$-regular branch coordinates at $p$, then we have
\begin{align}
	\hat{f}_1(z)=&Re\{z^{s+1}\}+\sum_{j=0}^{s}\sum_{k=0}^{s}(-1)^{2s-j-k}\frac{(s!)^2}{j!k!}\frac{\partial^{j+k}\varphi_1}{\partial \z^{j}\partial z^{k}}(z)\z^{j}z^k\label{formulap1}\\
	&+2Re\{\int_{0}^{z}w^sF_1(w)dw\},\nonumber\\
    \hat{f}_2(z)=&Im\{z^{s+1}\}+\sum_{j=0}^{s}\sum_{k=0}^{s}(-1)^{2s-j-k}\frac{(s!)^2}{j!k!}\frac{\partial^{j+k}\varphi_2}{\partial \z^{j}\partial z^{k}}(z)\z^{j}z^k\label{formulap2}\\
	&+2Re\{\int_{0}^{z}w^sF_2(w)dw\},\nonumber\\
    \hat{f}_h(z)=&\sum_{j=0}^{s}\sum_{k=0}^{s}(-1)^{2s-j-k}\frac{(s!)^2}{j!k!}\frac{\partial^{j+k}\varphi_h}{\partial \z^{j}\partial z^{k}}(z)\z^{j}z^k+2Re\{\int_{0}^{z}w^sF_h(w)dw\},\label{formulap3}\\
	&h\in\{3,\dots,n\},\nonumber
\end{align}
for some real $C^{2s+r+1,\beta}_{loc}$-functions $\varphi_h$ and holomorphic functions $F_h$ on an open connected neighborhood $W\subset x(U)$ of $0$, satisfying 
\begin{align}\label{conditionrepresent}
	(-1)^{s}s!\frac{\partial^{s+1}\varphi_h}{\partial z^{s+1}}(0)+F_h(0)=\varphi_h(0)=0
\end{align} for each $h\in\{1,\dots,n\}$. Furthermore, for any $(m,\gamma)\in (\Na\cup \{0,\infty,\omega\})\times [0,1]$ and $s\in \Na$, if we substitute in the formulas (\ref{formulap1}), (\ref{formulap2}) and (\ref{formulap3}), real $C^{2s+m+1,\gamma}_{loc}$-functions $\varphi_h$ on an open connected neighborhood $\hat{U}\subset\Re^2$ of $0$ and holomorphic functions $F_h$ on $\hat{U}$ satisfying (\ref{conditionrepresent}), we obtain a $C^{m+1,\gamma}_{loc}$-map $\hat{f}:\hat{U}\to \Re^n$ with $0$ being a branch point of order $s$. Also, shrinking $\hat{U}$ if necessary, the standard coordinates of $\hat{U}$ and $\Re^n$ are $C^{m,\gamma}_{loc}$-regular branch coordinates at $0$ for $\hat{f}$.   
\end{theorem}
\begin{proof}
	By definition of $C^{r,\beta}_{loc}$-regular branch coordinates, $\hat{f}$ satisfy (\ref{cond1}), (\ref{cond2}) and (\ref{cond3}) with $d$ being a $C^{r,\beta}_{loc}$-map and $l$ a $C^{r-1,\beta}_{loc}$-map. From (\ref{cond1}) and (\ref{cond2}), we obtain that $\dd{((s+1)d)}{\z}(z)=l(z)\z^s$. By lemma \ref{lemform}, for each $h\in\{3,\dots,n\}$, there exists a real $C^{2s+r+1,\beta}_{loc}$-function $\varphi_h$ on an open connected neighborhood $W\subset x(U)$ of $0$ with $\varphi_h(0)=0$ and a holomorphic function $F_h$ on $W$, such that $(s+1)d_h(z)=\sum_{j=0}^{s}(-1)^{s-j}\frac{s!}{j!}\frac{\partial^{j+s+1}\varphi_h}{\partial \z^{j}\partial z^{s+1}}(z)\z^{j}+F_h(z)$.
	 
	For each $h\in \{3,\dots,n\}$, we can construct a function $f'_h(z)$, using $F_h$ with the real function $\varphi_h$ in the formula (\ref{formulap3}). Observe that the conjugate of the RHS of (\ref{formulap3}) is the same expression, then $f'_h$ is a real function, for each $h\in\{3,\dots,n\}$. Also,  we have that 
	\begin{align*}
	\dd{f'_h}{z}(z)&=\sum_{j=0}^{s}(\sum_{k=1}^{s}((-1)^{2s-j-k}\frac{(s!)^2}{j!k!}(\frac{\partial^{j+k+1}\varphi_h}{\partial \z^{j}\partial z^{k+1}}(z)\z^{j}z^k+k\frac{\partial^{j+k}\varphi_h}{\partial \z^{j}\partial z^{k}}(z)\z^{j}z^{k-1})\\
	&+(-1)^{2s-j}\frac{(s!)^2}{j!}\frac{\partial^{j+1}\varphi_h}{\partial \z^{j}\partial z}(z)\z^{j})+z^s F_h(z)\\
	&=\sum_{j=0}^{s}(\sum_{k=1}^{s}((-1)^{2s-j-k}\frac{(s!)^2}{j!k!}\frac{\partial^{j+k+1}\varphi_h}{\partial \z^{j}\partial z^{k+1}}(z)\z^{j}z^k\\
	&-(-1)^{2s-j-(k-1)}\frac{(s!)^2}{j!(k-1)!}\frac{\partial^{j+k}\varphi_h}{\partial \z^{j}\partial z^{k}}(z)\z^{j}z^{k-1}+(-1)^{2s-j}\frac{(s!)^2}{j!}\frac{\partial^{j+1}\varphi_h}{\partial \z^{j}\partial z}(z)\z^{j})\\
	&+z^s F_h(z)\\
	&=\sum_{j=0}^{s}(-1)^{s-j}\frac{s!}{j!}\frac{\partial^{j+s+1}\varphi_h}{\partial \z^{j}\partial z^{s+1}}(z)\z^{j}z^s+z^s F_h(z)=(s+1)z^sd_h(z).
    \end{align*}
    Therefore, $\hat{f}_h(z)=f'_h(z)+A(z)$, where $A(z)$ is an antiholomorphic function on $W$. Since $\hat{f}_h(z)$ and $f'_h(z)$ are real function vanishing at $z=0$, $A(z)$ is identically zero on $W$. Also, as $d_h(0)=0$ for all $h\in \{3,\dots,n\}$, we have that $(-1)^{s}s!\frac{\partial^{s+1}\varphi_h}{\partial z^{s+1}}(0)+F_h(0)=0$. 
    
    On the other hand, as $d_1=\frac{1}{2}(1+d'_1)$ and $d_2=-\frac{i}{2}(1+d'_2)$ with $d'_1(0)=d'_2(0)=0$, reasoning similarly as before,  we can prove the formulas (\ref{formulap1}) and (\ref{formulap2}) for some real $C^{2s+r+1,\beta}_{loc}$-functions $\varphi_1$, $\varphi_2$ and holomorphic functions $F_1$, $F_2$, where $(-1)^{s}s!\frac{\partial^{s+1}\varphi_h}{\partial z^{s+1}}(0)+F_h(0)=\varphi_h(0)=0$ for $h\in\{1,2\}$. To prove the last part, by direct computations we have
    \begin{align}
    	&\dd{\hat{f}_1}{z}(z)=z^s\frac{(s+1)}{2}(1+\frac{2}{s+1}(\sum_{j=0}^{s}(-1)^{s-j}\frac{s!}{j!}\frac{\partial^{j+s+1}\varphi_1}{\partial \z^{j}\partial z^{s+1}}(z)\z^{j}+F_1(z))),\label{directcomp1}\\ &\dd{\hat{f}_2}{z}(z)=z^s\frac{-i(s+1)}{2}(1+\frac{2i}{s+1}(\sum_{j=0}^{s}(-1)^{s-j}\frac{s!}{j!}\frac{\partial^{j+s+1}\varphi_2}{\partial \z^{j}\partial z^{s+1}}(z)\z^{j}+F_2(z))),\label{directcomp2}\\ &\dd{\hat{f}_h}{z}(z)=z^s(\sum_{j=0}^{s}(-1)^{s-j}\frac{s!}{j!}\frac{\partial^{j+s+1}\varphi_h}{\partial \z^{j}\partial z^{s+1}}(z)\z^{j}+F_h(z)),\text{ }h\in \{3,\dots,n\}\label{directcomp3}
    \end{align} and 
    \begin{align}
    	\dd{^2\hat{f}_k}{\z\partial z}(z)=\frac{|z|^{2s}}{4^{s+1}}\Delta^{s+1}\varphi_k(z)\label{directcomp4}
    \end{align} for $k\in \{1,\dots,n\}$. 
    These equalities imply that (\ref{cond1}), (\ref{cond2}) and (\ref{cond3}) are satisfied on $\hat{U}$ with the standard coordinates of $\hat{U}$ and $\Re^n$. Shrinking $\hat{U}$ if necessary, we have that (\ref{frontalcond}) and (\ref{quasiregularcond}) are satisfied. Then $(\hat{U},id_{\hat{U}})$, $(\Re^n,id_{\Re^n})$ are $C^{m,\gamma}_{loc}$-regular branch coordinates at $0$. 
\end{proof}

\begin{lemma}\label{beltrami}
	Let $f:M^2\to N^n$ be a $C^{k,\beta}_{loc}$-map, $(k,r)\in (\Na\cup \{\infty,\omega\})\times(\Na\cup \{0,\infty,\omega\})$, $p\in M^2$ a branch point and $\beta\in [0,1]$. If $(U,x)$, $(V,y)$ are $C^{r,\beta}_{loc}$-branch coordinates at $p$, then there exist a unique complex $C^{0,\beta}_{loc}$-function $\varpi:x(U)\to \Co$, such that $\varpi$ is a $C^{max\{k-1,r\},\beta}_{loc}$-map on $x(U)-\{0\}$, $\sup_{z\in x(U)}|\varpi(z)|<1$, $\varpi(0)=0$ and $\dd{a}{\z}=\varpi \dd{a}{z}$ on $x(U)$, where $a$ is the complex principal part of $\hat{f}$. In particular, $a:x(U)\to \Co$ is a quasiregular map.
\end{lemma}
\begin{proof}
	By definition of branch coordinates, we have that $\dd{\hat{f}}{z}(z)=(s+1)z^sd(z)$ with $d$ being a $C^{r,\beta}$-map, where $d_1(z)=\frac{1}{2}(1+d'_1(z))$, $d_2(z)=-\frac{i}{2}(1+d'_2(z))$, $d'_1(0)=0$ and $d'_2(0)=0$. Thus, the  principal part $a=\hat{f}_1+i\hat{f}_2$ of $\hat{f}$ satisfies $\dd{a}{z}(z)=(s+1)z^s(1+\frac{1}{2}(d'_1(z)+d'_2(z)))$ and $\dd{a}{\z}(z)=(s+1)\z^s\frac{1}{2}(\overline{d'_1(z)-d'_2(z)})$. By condition (\ref{quasiregularcond}) $1+\frac{1}{2}(d'_1(z)+d'_2(z))\neq 0$ on $x(U)$, then $\varpi(z)=\dd{a}{\z}(z)(\dd{a}{z}(z))^{-1}=\z^s\frac{1}{2}(\overline{d'_1(z)-d'_2(z)})(z^s(1+\frac{1}{2}(d'_1(z)+d'_2(z))))^{-1}$ is a $C^{max\{k-1,r\},\beta}_{loc}$-map well-defined on $x(U)-\{0\}$ and by lemma \ref{holderl1}, the map $\varpi$ has a unique $C^{0,\beta}_{loc}$-extension on $x(U)$ with $\varpi(0)=0$. By condition (\ref{quasiregularcond}), we have $\sup_{z\in x(U)}|\varpi(z)|<1$ and therefore, as $a$ satisfies the Beltrami equation $\dd{a}{\z}=\varpi \dd{a}{z}$, $a$ is a quasiregular map.  
\end{proof}
\begin{definition}\normalfont
	Let $f:U\to \Re^n$ be a $C^{1}$-map with $U\subset \Re^2$ an open neighborhood of $0$, $r\in \Na\cup \{0,\infty,w\}$, $0\in U$ a branch point of order $s$ and $\beta\in [0,1]$. If $(U,id_U)$, $(\Re^n,id_{\Re^n})$ are $C^{r,\beta}_{loc}$-branch coordinates at $0$, we will call the map $\varpi$ in lemma \ref{beltrami} the {\it distinguished coefficient of $f$}.  
\end{definition}
\begin{theorem}\label{Main3}
	Let $f:M^2\to N^n$ be a $C^{r+1,\beta}_{loc}$-map, $(r,m)\in (\Na\cup \{\infty,\omega\})\times (\Na\cup \{\infty,\omega\})$ with $r\geq m$, $p\in M^2$ a branch point of order $s$ and $\beta\in [0,1]$. If $(U,x)$ and $(V,y)$ are $C^{r,\beta}_{loc}$-regular branch coordinates at $p$ with $x(U)$ an open convex set, the following assertions hold: 
	\begin{enumerate}[label=(\roman*)]
		\item The distinguished coefficient $\varpi$ of $\hat{f}$ is a $C^{m,\beta}_{loc}$-map on $U$ if and only if there exist open neighborhoods $U'\subset \Re^2$, $U''\subset x(U)$ of $0$ and a $C^{m+1,\beta}_{loc}$-diffeomorphism $c:U'\to U''$ with $c(0)=0$, such that $(\hat{f}_1\circ c)(w)+i(\hat{f}_2\circ c)(w)=w^{s+1}$. 
		\item Every $C^{m+1,\beta}_{loc}$-diffeomorphism $c:U'\to U''$ as in the item $(i)$ is a quasiconformal map satisfiying $\dd{(c^{-1})}{\z}(z)=\varpi(z)\dd{(c^{-1})}{z}(z)$ and has the form $c(w)=wc_0(w)$, where $c_0$ is a $C^{m,\beta}_{loc}$-map with $c_0(0)\neq 0$.
		\item If $c:U'\to U''$ is a $C^{m+1,\beta}_{loc}$-diffeomorphism as in the item $(i)$ then $(U',id_{U'})$, $(\Re^n,id_{\Re^n})$ are $C^{min\{r,m+1\},\beta}_{loc}$-branch coordinates at $0$ for $\hat{f}\circ c$. More precisely, the expression of $\hat{f}\circ c$ is given by $(\hat{f}_1\circ c)(w)+i(\hat{f}_2\circ c)(w)=w^{s+1}$ and $(\hat{f}_h\circ c)(w)=(s+1)Re\{\int_{0}^{w}\overline{b}_h(z)z^{s}dz\}, h\in\{3,\dots,n\}$, for some complex $C^{min\{r,m+1\},\beta}_{loc}$-maps $b_h$ on $U'$ satisfying $b_h(0)=0$ for each $h\in\{3,\dots,n\}$. 
		\item If $m\in \Na$ (resp. $m\in\{\infty, \omega\}$), the distinguished coefficient $\varpi$ of $\hat{f}$ is a $C^{m,\beta}_{loc}$-map on $x(U)$ if and only if the maps $d'_1$ and $d'_2$ in (\ref{cond3}) satisfy $\dd{^j(\overline{d'_1-d'_2})}{\z^j}(0)=0$ for all $j\in\{0,\dots,m\}$ (resp. for all $j\in\Na$). Also, if $\varpi$ is a $C^{m,\beta}_{loc}$-map on $U$, then $\dd{^j\varpi}{\z^j}(0)=0$ for all $j\in\{0,\dots,m\}$ (resp. for all $j\in\Na\cup \{0\}$).
		\item If $f$ is a $C^{r+1,\beta}_{loc}$-branched conformal immersion and $(V,y)$ is a normal coordinate system at $f(p)$, the distinguished coefficient $\varpi$ of $\hat{f}$ is a $C^{1,\beta}_{loc}$-map. 
	\end{enumerate}
\end{theorem}
\begin{proof}
	First, let us prove the item $(iv)$. Observe in the proof of lemma \ref{beltrami} that $\varpi$ is a $C^{m,\beta}_{loc}$-map if and only if $\frac{\z^s}{z^s}\overline{(d_1'-d_2')}$ is a $C^{m,\beta}_{loc}$-map. Also, by (\ref{cond1}), (\ref{cond2}) and (\ref{cond3}) we have $\dd{(\overline{d'_1-d'_2})}{z}=\frac{2z^s}{s+1}(l_1(z)-il_2(z))$. Then,  the assertions follow from lemma \ref{lemext3} and item $(ii)$ of lemma \ref{lemext4}.
	
	Now, let us prove the item $(i)$. If $\varpi$ is a $C^{m,\beta}_{loc}$-map  and $a$ is the complex principal part of $\hat{f}$, we have $a(z)=\int_{0}^{1}z\dd{a}{z}(tz)+\z\dd{a}{\z}(tz)dt=(s+1)z^{s+1}\int_{0}^{1}t^s(1+\frac{1}{2}(d'_1(tz)+d'_2(tz)))+\frac{\z t}{zt}\varpi(tz)t^s(1+\frac{1}{2}(d'_1(tz)+d'_2(tz)))dt$
	on $x(U)-\{0\}$. By item $(iv)$ and lemma \ref{lemext1}, the map $\varpi'(z)=\frac{\z}{z}\varpi(z)$ has a $C^{m,\beta}_{loc}$-extension on $x(U)$. Therefore, $c'(z)=(s+1)\int_{0}^{1}t^s(1+\frac{1}{2}(d'_1(tz)+d'_2(tz)))+\varpi'(tz)t^s(1+\frac{1}{2}(d'_1(tz)+d'_2(tz)))dt$ is a $C^{m,\beta}_{loc}$-map on $x(U)$ with $c'(0)=1$, which implies that $(\frac{a(z)}{z^{s+1}})^{\frac{1}{s+1}}=(c'(z))^{\frac{1}{s+1}}$ is a non-vanishing $C^{m,\beta}_{loc}$-map on an open neighborhood $U''\subset x(U)$ of $0$ and a $C^{m+1,\beta}_{loc}$-map on $U''-\{0\}$. Setting $e(z)=z(\frac{a(z)}{z^{s+1}})^{\frac{1}{s+1}}$, we have $\dd{e}{\z}(z)=\frac{z}{s+1}(\frac{a(z)}{z^{s+1}})^{\frac{-s}{s+1}}\dd{a}{\z}(z)\frac{1}{z^{s+1}}=\frac{\varpi(z)}{s+1}(c'(z))^{\frac{-s}{s+1}}(1+\frac{1}{2}(d'_1(z)+d'_2(z)))$ and $\dd{e}{z}(z)=(\frac{a(z)}{z^{s+1}})^{\frac{1}{s+1}}+\frac{z}{s+1}(\frac{a(z)}{z^{s+1}})^{\frac{-s}{s+1}}(\dd{a}{z}(z)z^{s+1}-a(z)(s+1)z^s)z^{-2(s+1)}=(c'(z))^{\frac{1}{s+1}}+(c'(z))^{\frac{-s}{s+1}}((1+\frac{1}{2}(d'_1(z)+d'_2(z)))-c'(z))$ on $U''-\{0\}$. Since the RHS of these equations are $C^{m,\beta}_{loc}$-maps on $U''$ and $e\in C^{m+1,\beta}_{loc}(U''-\{0\},\Co)\cap C^{m,\beta}_{loc}(U'',\Co)$, by lemma \ref{lemisolated} the map $e$ is of class $C^{m+1,\beta}_{loc}$ on $U''$. Also, as $|\dd{e}{z}(0)|^2-|\dd{e}{\z}(0)|^2=1$, shrinking $U''$ if necessary and setting $U'=e(U'')$, we have that $e:U''\to U'$ is a $C^{m+1,\beta}_{loc}$-diffeomorphism. Thus, as $a(z)=(e(z))^{s+1}$ then the inverse map $c:U'\to U''$ of $e$ satisfies $(a\circ c)(w)=w^{s+1}$. Conversely, if $(a\circ c)(w)=w^{s+1}$ then $a(z)=(c^{-1}(z))^{s+1}$ and therefore $(s+1)(c^{-1}(z))^{s}\dd{(c^{-1})}{z}(z)\varpi(z)=\dd{a}{z}(z)\varpi(z)=\dd{a}{\z}(z)=(s+1)(c^{-1}(z))^{s}\dd{(c^{-1})}{\z}(z)$. Thus, \begin{align}\label{beltrami2}
		\dd{(c^{-1})}{\z}(z)=\varpi(z)\dd{(c^{-1})}{z}(z)
	\end{align} which implies that $\dd{(c^{-1})}{z}$ is a non-vanishing map on $U'$, otherwise this would contradict that $c$ is a diffeomorphism. It follows that $\varpi$ is a $C^{m,\beta}_{loc}$-map. 
	
	For the item $(ii)$, observe that $c^{-1}$ satisfies (\ref{beltrami2}), where $\sup_{z\in U}|\varpi(z)|<1$, therefore $c^{-1}$ is a quasiconformal map and by theorem 3.7.7 in \cite{As}, $c$ is also quasiconformal. On the other hand, by item $(i)$ $\varpi$ is a $C^{m,\beta}_{loc}$-map and therefore, we can construct a map $e$ as in the proof of the item $(i)$, which has the form $e(z)=ze'(z)$ and satisfies (\ref{beltrami2}), where $e'$ is a non-vanishing map of class $C^{m,\beta}_{loc}$. By Stoilow factorization (theorem 5.5.1 in \cite{As}) there exists a holomorphic function $\Phi$ on an open neighborhood of $0$ such that $c^{-1}=\Phi\circ e$ near $0$. As $e$ and $c$ are diffeomorphisms, $\Phi$ is a conformal map. Then, we can write $c^{-1}(z)=ze_0(z)$ on an open neighborhood of $0$, where $e_0$ is a non-vanishing $C^{m,\beta}_{loc}$-map. This implies that $w=c(w)e_0(c(w))$ and thus $\frac{c(w)}{w}$ has a $C^{m,\beta}_{loc}$-extension $c_0$ on $U'$ with $c_0(0)\neq 0$. On the other hand, by proposition \ref{branchfrontal} $\hat{f}$ is a normalized frontal, applying lemma \ref{normalizedchange}, the item $(i)$ of this theorem and corollary \ref{prepre2}, the item $(iii)$ follows. 
	
	Lastly, to prove the item $(v)$ observe that if $f$ is a branched conformal immersion, $\hat{f}$ satisfies the equations (\ref{conformaleq1}) and (\ref{conformaleq2}). These equations can be written in the form $\dd{\hat{f}}{x_1}^t\tilde{\G} \dd{\hat{f}}{x_1}=\dd{\hat{f}}{x_2}^t\tilde{\G} \dd{\hat{f}}{x_2}$ and $\dd{\hat{f}}{x_1}^t\tilde{\G} \dd{\hat{f}}{x_2}=0$, where $\tilde{\G}_{ij}=\tilde{g}_{ij}\circ f \circ x^{-1}$ and $\tilde{g}_{ij}$ are the components of the Riemannian metric on $N^n$ in the coordinates $(V,y)$. This is equivalent to having $\dd{\hat{f}}{\z}^t\tilde{\G} \dd{\hat{f}}{\z}=0$ and using (\ref{cond1}), we obtain $\overline{d}^t\tilde{\G}\overline{d}=0$ on $x(U)$. Then, $2\dd{\overline{d}}{\z}^t\tilde{\G}\overline{d}+\overline{d}^t\dd{\tilde{\G}}{\z}\overline{d}=0$ and since $(V,y)$ is a normal coordinate system at $f(p)$, evaluating at $0$ we have $2(\dd{\overline{d}}{\z}(0))^t\overline{d}(0)=0$. From condition (\ref{cond3}), it follows that $\dd{(\overline{d'_1-d'_2})}{\z}(0)=0$ and by item $(iv)$, the distinguished coefficient $\varpi$ of $\hat{f}$ is a $C^{1,\beta}_{loc}$-map. 
\end{proof}
\begin{remark}\label{remdistinguished}\normalfont
	In the theorem \ref{Main3}, if $m=0$, $r\geq 0$ and $(U,x)$, $(V,y)$ are $C^{r,\beta}_{loc}$-branch coordinates at $p$, the items $(i)$, $(ii)$, $(iii)$ and $(iv)$ are true. To see this, observe that the item $(iv)$ follows from the lemma \ref{holderl1}. Therefore, the proofs above of items $(i)$, $(ii)$ and $(iii)$ are valid in this case. 
\end{remark}
\begin{corollary}
	Let us choose $k,r\in \Na\cup \{0,\infty,\omega\}$ with $k\leq r$ and $\beta\in [0,1]$.
	Let $\hat{f}$ be the $C^{r+1,\beta}_{loc}$-map obtained by substituting in the formulas (\ref{formulap1}), (\ref{formulap2}) and (\ref{formulap3}), a natural number $s$, real $C^{2s+r+1,\beta}_{loc}$-functions $\varphi_h$ on an open convex neighborhood $\hat{U}\subset\Re^2$ of $0$ and holomorphic functions $F_h$ on $\hat{U}$, satisfying $(-1)^{s}s!\frac{\partial^{s+1}\varphi_h}{\partial z^{s+1}}(0)+F_h(0)=\varphi_h(0)=0$.
	Then, if $k\in\Na$ (resp. $k\in\{\infty, \omega\}$), the distinguished coefficient $\varpi$ of $\hat{f}$ is a $C^{k,\beta}_{loc}$-map if and only if $(-1)^{s}s!\frac{\partial^{s+1+j}(\varphi_1-i\varphi_2)}{\partial z^{s+1+j}}(0)+\dd{^j(F_1-iF_2)}{z^j}(0)=0$ for all $j\in\{1,\dots,k\}$ (resp. for all $j\in\Na$).
\end{corollary}
\begin{proof}
	It follows from formulas (\ref{directcomp1}), (\ref{directcomp2}) and item $(iv)$ of theorem \ref{Main3}.
\end{proof}
We will now treat the index and the distinguished degree of branch points using $C^\infty$-branch coordinates. The item $(iv)$ of theorem \ref{Main3} gives us a way to compute the distinguished degree, so here, we will give an equivalent definition for this number that makes its computation easier. First, we need the following definitions and properties:
\begin{definition}\normalfont
	Let $f:U\to\Co^n$ be a $C^{\infty}$-map, $U\subset\Co$ an open neighborhood of $0$ and $k\in \Na$. If $\dd{^jf}{z^j}(0)=0$ for all $j\in\{0,\dots,k-1\}$ and $\dd{^kf}{z^k}(0)\neq0$, we say that the {\it $z$-order of $f$}  is $k$. The natural number $k$ will be denoted by $ord_z(f)$. In the case $\dd{^jf}{z^j}(0)=0$ for all $j\in\Na$, we say that the $z$-order of $f$ is $\infty$ and we write $ord_z(f)=\infty$. 
\end{definition}
\begin{lemma}\label{zorder}
	Let $f:U\to\Co^n$, $g:U\to\Co^n$ be $C^{\infty}$-maps vanishing at $0\in U\subset \Co$, $U$ an open set and $e:U\to\Co$ a non-vanishing $C^\infty$-map. The following statements are true:
	\begin{enumerate}[label=(\roman*)]
		\item $ord_z(ef)=ord_z(f)$,
		\item $ord_z(f+g)\geq min\{ord_z(f),ord_z(g)\}$. Also, if $ord_z(f)< ord_z(g)$ then $ord_z(f+g)=ord_z(f)$, 
		\item if $n=1$, then $ord_z(fg)=ord_z(f)+ord_z(g)$,
		\item $ord_z(f)=min\{ord_z(f_1),\dots,ord_z(f_n)\}$ and
		\item $ord_z(\z f)=\infty$.
	\end{enumerate} 
\end{lemma}
\begin{proof}
	From the Leibniz rule (i.e., $\dd{^k(fg)}{z^k}=\sum_{j=0}^{k}\binom{k}{j}\dd{^{k-j}f}{z^{k-j}}\dd{^{j}g}{z^{j}}$ for any $f,g\in C^{\infty}$) it follows items $(i)$ and $(iii)$. The other items follow easily from the definition of $z$-order. 
\end{proof}
\begin{remark}\normalfont
	In the item $(iii)$ of lemma \ref{zorder}, we are adopting the conventions $\infty+\infty=\infty$ and $k+\infty=\infty+k=\infty$ for any $k\in \Na$.
\end{remark}
\begin{definition}\normalfont\label{defindex}
	Let $f:M^2\to N^n$ be a smooth map, $n\geq 3$, $p\in M^2$ a branch point of order $s$ and $(U,x)$, $(V,y)$ $C^{\infty}$-branch coordinates at $p$. Consider the map $F=(\hat{f}_3,\dots,\hat{f}_n)$ and $a=\hat{f}_1+i\hat{f}_2$ the complex principal part of $\hat{f}$. The number $\iota=ord_z(F)-1$ will be called the {\it index of $p$ respect to $(U,x)$, $(V,y)$} and the number $\varrho=ord_z(\overline{a})-(s+1)$ will be called the {\it distinguished degree of $p$ respect to $(U,x)$, $(V,y)$}.
\end{definition}
\begin{remark}\label{remindex}\normalfont
	If $d$ is the map in (\ref{cond1}) and $d_1', d_2'$ are the maps in (\ref{cond3}), as $\dd{a}{\z}(z)=(s+1)\z^s\frac{1}{2}(\overline{d'_1(z)-d'_2(z)})$ observe that $\varrho=ord_z(d_1'-d_2')$. Also, we have that $\iota=s+min\{ord_z(d_3),\dots,ord_z(d_n)\}$ and therefore $\iota>s$. 
\end{remark}
The index and the distinguished degree of a branch point satisfy certain invariance and inequalities depending on the properties of $f$ and the coordinate system $y$. We are interested in the case that $y$ is a conformal diffeomorphism, which appears when $(N^n,\gt)$ is a conformally flat manifold.
\begin{lemma}\label{lemindex0}
	Let $f:U\to\Co^n$ be a $C^{\infty}$-map, $U,U'\subset\Co$ open neighborhoods of $0$ and $h:U'\to U$ a $C^\infty$-diffeomorphism with $h(0)=0$. If $f(0)=0$, $\dd{h}{z}(0)\neq 0$ and $\dd{\overline{h}}{z}(z)=\z e(z)$ for some smooth map $e:U'\to\Co$ then $ord_z(f)=ord_z(f\circ h)$. 
\end{lemma}
\begin{proof}
	We get the result by using the properties of lemma \ref{zorder} with a simple argument by induction on the $z$-order of $f$.   	
\end{proof}
\begin{lemma}\label{lemindex1}
	Let $U,U'\subset \Co$ and $V,V'\subset \Re^n$ be open neighborhoods of the origin, $h:U'\to U$ and $y:V\to V'$ two $C^\infty$-diffeomorphisms, and $f:U \to V$ a $C^\infty$-map, such that $f$, $h$, $y$ vanish at the origin. If $(U,id_U)$, $(V,id_V)$ are $C^\infty$-branch coordinates at $0$ for $f$ and $(U',id_{U'})$, $(V',id_{V'})$ are $C^\infty$-branch coordinates at $0$ for $\tilde{f}=y\circ f \circ h$, then $h(z)=ze(z)$ for some non-vanishing $C^{\infty}$-map $e:U'\to \Co$. If additionally $y$ is an isometry, then $y(q)=\M q$ for some $\M\in O(n)$ in the form (\ref{ortogonalmatrix}).
\end{lemma}
\begin{proof}
	We have $\dd{f}{w}(w)=(s+1)w^sd(w)$ and $\dd{\tilde{f}}{z}(z)=(s+1)z^sd'(z)$ for some $C^\infty$-maps $d:U\to \Co^n$, $\tilde{d}:U'\to \Co^n$ with real and imaginary parts linearly independent. By lemma \ref{complexfrontal}, there exist framed coordinate systems $\VS'$, $\VS$ of $T_{\tilde{f}}$, $T_f$ respectively such that $\J{\tilde{f}}=\W' \J{\VS'}$ and $\J{f}=\W \J{\VS}$, where $\W, \W', \J{\VS'},\J{\VS}$ are given by (\ref{eqcomplexfrontal}). We have $\J{y}(f\circ h)\J{f}(h)\J{h}=\J{\tilde{f}}$, so setting $\B=\J{y}(f\circ h)\W(h)$, it follows that $\B\J{\VS}(h)\J{h}=\W'\J{\VS'}$. By lemma \ref{normalizedfrontal} the first and second rows of $\B$ and $\W'$ form $2\times 2$-matrix-valued maps with non-vanishing determinant, therefore $\J{\VS}(h)\J{h}=\F\J{\VS'}$ for some map $\F\in C^{\infty}(U',GL(2))$. Setting the map $E(z)=(Re\{z^{s+1}\},Im\{z^{s+1}\})$, due to the form of $\J{\VS'},\J{\VS}$ given by (\ref{eqcomplexfrontal}), we have that $\J{\VS'}=\J{E}$ and $\J{\VS}(h)=\J{E}(h)$. Thus, it holds that $\J{E\circ h}=\F\J{E}$, which is equivalent to having $\dd{(E_1\circ h)}{z}=\frac{1}{2}(\F_{11}-i\F_{12})(s+1)z^s$ and $\dd{(E_2\circ h)}{z}=\frac{1}{2}(\F_{21}-i\F_{22})(s+1)z^s$. This implies that $\dd{(h^{s+1})}{z}=F_0(z)z^s$, for some non-vanishing smooth map $F_0:U'\to\Co$. Applying $s$ times $\dd{}{z}$ to both sides of the last equality, we can conclude $\dd{h}{z}(0)\neq 0$ and hence $h(z)^s=F_1(z)z^s$ near to $0$, where $F_1$ is a complex non-vanishing $C^{\infty}$-map near $0$. Applying $\dd{}{z}$ in the last equality, we obtain that $sh(z)^{s-1}\dd{h}{z}(z)=\dd{F_1}{z}(z)z^s+sF_1(z)z^{s-1}$. Observe that the map $F_2(z)=s^{-1}(\dd{h}{z}(z))^{-1}(\dd{F_1}{z}(z)z+sF_1(z))$ is a non-vanishing $C^{\infty}$-map near the origin. Then, $h(z)^{s-1}=F_2(z)z^{s-1}$ and repeating successively this procedure, we get that $h(z)=F_s(z)z$ for a non-vanishing $C^{\infty}$-map $F_s$ near the origin. This implies that $e(z)=\frac{h(z)}{z}$ has a non-vanishing $C^{\infty}$-extension on $U'$ and it follows the result.
	
	On the other hand, if $y$ is an isometry, we have $y(q)=\M q$ for some $\M\in O(n)$. By proposition \ref{branchfrontal} $y\circ f\circ h$ and $f$ are normalized frontals, then by lemma \ref{normalizedchange} $f\circ h$ is normalized. As $\M$ is orthogonal, it defines an isometry of $\Re^n$, and by proposition \ref{isometries}, $\M$ has the form of (\ref{ortogonalmatrix}).
\end{proof}
\begin{proposition}\label{propindex}
	Let $f:M^2\to N^n$ be a smooth map, $p\in M^2$ a branch point and $(U,x)$, $(V,y)$ $C^{\infty}$-branch coordinates at $p$. If $(U',x')$, $(V',y')$ are also $C^{\infty}$-branch coordinates at $p$ and the map $y'\circ y^{-1}$ is an isometry, then the index and the distinguished degree of $p$ respect these two pair of branch coordinates are preserved.
\end{proposition}
\begin{proof}
	Setting $h=x\circ x'^{-1}$, $c=y'\circ y^{-1}$, $f'=y'\circ f\circ x'^{-1}$ and $\hat{f}=y\circ f\circ x^{-1}$, we have $f'=c\circ \hat{f}\circ h$. Then, by lemma \ref{lemindex1}, $h(z)=e(z)z$ for some non-vanishing $C^\infty$-map $e$ defined near to $0$, and $c(q)=\M q$ for some $\M\in O(n)$ in the form (\ref{ortogonalmatrix}). If we set $F'=(f'_3,\dots,f'_n)$, $F=(f_3,\dots,f_n)$, $a$ the complex principal part of $\hat{f}$ and $a'$ the complex principal part of $f'$, then $F'=\G F\circ h$ and $a'=(\D_{11}-i\D_{12})a\circ h$ (see remark \ref{complexprincipalpart}). Thus, by lemma \ref{lemindex0} $ord_z(F')=ord_z(F)$ and $ord_z(\overline{a})=ord_z(\overline{a'})$. From this, the result follows. 
\end{proof}
\begin{remark}\normalfont
	In the context of the proposition \ref{propindex}, the map $y'\circ y^{-1}$ is an isometry, for instance, when $y$, $y'$ are conformal diffeomorphism with the same conformal factor or when $y=y'$. Hence, the index and the distinguished degree of $p$ are independent of the coordinate system $(U,x)$. 
\end{remark}
\begin{proposition}\label{estimation}
	Let $f:M^2\to N^n$ be a $C^{\infty}$-branched conformal immersion, $\gt$ a smooth Riemannian metric on $N^n$, $p\in M^2$ a branch point of order $s$ and $(U,x)$, $(V,y)$ $C^{\infty}$-branch coordinates at $p$. If $y$ is a conformal diffeomorphism, $\iota$ is the index of $p$ and $\varrho$ is the distinguished degree of $p$, then $\varrho\geq 2(\iota-s)$. In particular the distinguished coefficient of $\hat{f}$ is a $C^{2(\iota-s)-1,1}_{loc}$-map.
\end{proposition}
\begin{proof}
	As $f$ is a branched conformal immersion, $\hat{f}$ satisfies the equations (\ref{conformaleq1}) and (\ref{conformaleq2}), which can be written in the form $\dd{\hat{f}}{x_1}^t\tilde{\G} \dd{\hat{f}}{x_1}=\dd{\hat{f}}{x_2}^t\tilde{\G} \dd{\hat{f}}{x_2}$ and $\dd{\hat{f}}{x_1}^t\tilde{\G} \dd{\hat{f}}{x_2}=0$, where $\tilde{\G}_{ij}=\tilde{g}_{ij}\circ f \circ x^{-1}$ and $\tilde{g}_{ij}$ are the components of the Riemannian metric on $N^n$ in the coordinates $(V,y)$. This is equivalent to having $\dd{\hat{f}}{\z}^t\tilde{\G} \dd{\hat{f}}{\z}=0$ and using (\ref{cond1}), we obtain $\overline{d}^t\tilde{\G}\overline{d}=0$ on $x(U)$. On the other hand, since $y$ is a conformal diffeomorphism, then $\tilde{\G}=\rho\id_n$ for some positive smooth function $\rho: x(U)\to \Re$. Therefore, $d^td=0$ on $x(U)$ and using (\ref{cond3}) this can be written in the form $\frac{1}{4}(1+d_1')^2-\frac{1}{4}(1+d_2')^2+d_3^2+\dots+d_n^2=0$. Thus, $d_1'-d_2'=-2(d_3^2+\dots+d_n^2)(1+\frac{1}{2}(d_1'+d_2'))^{-1}$ and by lemma \ref{zorder} $\varrho=ord_z(d_1'-d_2')\geq 2min\{ord_z(d_3),\dots,ord_z(d_n)\}=2(\iota-s)$. By item $(iv)$ of theorem \ref{Main3}, the distinguished coefficient of $\hat{f}$ is a $C^{2(\iota-s)-1,1}_{loc}$-map.
\end{proof}
\begin{remark}\normalfont
	In proposition \ref{estimation}, if $n=3$ we have that $\varrho= 2(\iota-s)$. This fact can be verified at the end of the proof above. 
\end{remark}
We finish this section with the next lemma, which will be essential to prove the theorem \ref{Theofinal} in the next section. In order to abbreviate the proof of the lemma, we define the following sets.

For an open set $U\subset \Co$ and $s,k\in \Na\cup\{0\}$ with $s\leq k$, we define $\mathfrak{C}_{s,k}(U)=\{f\in C^{\infty}(U,\Co):\text{for each $h'\in\{0,\dots,s\}$, $\dd{^{h+h'}f}{\z^h\partial z^{h'}}(0)=0$ for all $h\in\{0,\dots,k-h'\}$}\}.$ As $\mathfrak{C}_{0,k}(U)=\{f\in C^{\infty}(U,\Co):ord_z(\overline{f})\geq k+1\}$, by item $(i)$, $(ii)$ and $(iii)$ of lemma \ref{zorder} the set $\mathfrak{C}_{0,k}(U)$ is an ideal of the ring $C^{\infty}(U,\Co)$. Observe that $\mathfrak{C}_{s,k}(U)=\{f\in C^{\infty}(U,\Co):\text{for each $h'\in\{0,\dots,s\}$, $\dd{^{h'}f}{z^{h'}}$}\in \mathfrak{C}_{0,k-h'}(U)\}$, then using the Leibniz rule, it follows that $\mathfrak{C}_{s,k}(U)$ is an ideal of $C^{\infty}(U,\Co)$. Also, using the Leibniz rule, it is easy to check that the product of ideals satisfies $\mathfrak{C}_{0,k_1}(U)\mathfrak{C}_{0,k_2}(U)\subset \mathfrak{C}_{0,k_1+k_2+1}(U)$ for any $k_1, k_2\in \Na\cup\{0\}$. Given this last property, it follows similarly that $\mathfrak{C}_{s,k_1}(U)\mathfrak{C}_{s,k_2}(U)\subset \mathfrak{C}_{s,k_1+k_2+1}(U)$ for any $s\in \Na\cup\{0\}$ with $s\leq k_1$ and $s\leq k_2$. 

\begin{lemma}\label{lemcoprincipalpart}
	Let $f:M^2\to N^n$ be a $C^{\infty}$-map, $\gt$ a smooth Riemannian metric on $N^n$, $p\in M^2$ a branch point of order $s$ and $(U,x)$, $(V,y)$ $C^{\infty}$-regular branch coordinates at $p$. If $\iota$ is the index of $p$, the following statements are true:
	\begin{enumerate}[label=(\roman*)]
		\item If $\iota<2s+1$ and $b$ is the complex principal part of $\hat{f}$, then $\dd{b}{z}=z^{\iota-s-1}e(z)$ and $\dd{b}{\z}=\z^{\iota-s-1}c(z)$, where $e,c\in C^{0,1}_{loc}(x(U),\Co^{n-2})$ with $c(0)=0$ and $e(0)\neq 0$.
		\item If $\iota\geq 2s+1$ and $b$ is the complex principal part of $\hat{f}$, then $\dd{b}{z}=z^{s}e(z)$ and $\dd{b}{\z}=\z^{s}c(z)$, where $e,c\in C^{0,1}_{loc}(x(U),\Co^{n-2})$. 
	\end{enumerate}
\end{lemma}
\begin{proof}
	By proposition \ref{branchfrontal} we have $b_j=(d_{j+2}(1+\frac{1}{2}(d_1'+d_2'))-\bar{d}_{j+2}\overline{(d_1'-d_2')})(Re\{(1+d_1')(1+\bar{d}_2')\})^{-1}$ for all $j\in\{1,\dots,n-2\}$. On the other hand, $\dd{d}{\z}=(s+1)\z^sl$ and $d_1'(0)-d_2'(0)=0$ (see condition (\ref{cond1}), (\ref{cond2}) and (\ref{cond3})), where $d$ and $l$ are smooth maps. If $\iota-s-1<s$, using lemma \ref{zorder}, we obtain that $ord_z(\bar{d}_{j+2}\overline{(d_1'-d_2')})>s+1>\iota-s$ and since $\iota-s=min\{ord_z(d_3),\dots,ord_z(d_n)\}$ (see remark \ref{remindex}), we have $ord_z(b)=\iota-s$. Also, $ord_z(d_{j+2}(d_1'-d_2'))\geq \iota-s+1$, $ord_z(\bar{d}_{j+2})\geq s+1\geq \iota-s+1$ and thus $ord_z(\overline{b})\geq \iota-s+1$. We can assume $\iota-s-1\geq 1$, otherwise $0\neq\dd{b}{z}=z^{0}\dd{b}{z}$, $\dd{b}{\z}=0=\z^0\dd{b}{\z}$ and the assertion in the item $(i)$ holds. Note that $d_{j+2}, \bar{d}_{j+2}\in \mathfrak{C}_{\iota-s-1,s}(x(U))\subset \mathfrak{C}_{\iota-s-1,\iota-s-1}(x(U))\subset \mathfrak{C}_{\iota-s-2,\iota-s-1}(x(U))$ and $\dd{d_{j+2}}{z},\dd{\overline{d}_{j+2}}{z}\in \mathfrak{C}_{\iota-s-2,\iota-s-1}(x(U))$. Since $\dd{b_j}{z},\dd{\overline{b}_j}{z}$ are in the ideal generated by $d_{j+2}, \bar{d}_{j+2}, \dd{d_{j+2}}{z},\dd{\overline{d}_{j+2}}{z}$ then $\dd{b_j}{z},\dd{\overline{b}_j}{z}\in \mathfrak{C}_{\iota-s-2,\iota-s-1}(x(U))$ for all $j\in\{1,\dots,n-2\}$. As $ord_z(\dd{b}{z})=\iota-s-1$ and $ord_z(\dd{\overline{b}}{z})\geq \iota-s$, by lemma \ref{lemext4} the item $(i)$ follows. Now, if $\iota-s-1\geq s$ then $\bar{d}_{j+2}\in \mathfrak{C}_{s,\iota-s-1}(x(U))\subset \mathfrak{C}_{s,s}(x(U))\subset \mathfrak{C}_{s-1,s}(x(U))$, $\dd{\overline{d}_{j+2}}{z}\in \mathfrak{C}_{s-1,\iota-s-1}(x(U))\subset \mathfrak{C}_{s-1,s}(x(U))$,  $d_{j+2}\in \mathfrak{C}_{s,s}(x(U))\subset \mathfrak{C}_{s-1,s}(x(U))$ and $\dd{d_{j+2}}{z}\in \mathfrak{C}_{s-1,s}(x(U))$. Thus, $\dd{b_j}{z},\dd{\overline{b}_j}{z}\in \mathfrak{C}_{s-1,s}(x(U))$ for all $j\in \{1,\dots,n-2\}$ and by lemma \ref{lemext4} the item $(ii)$ follows.
\end{proof}
\section{Behavior of curvatures near branch points.}\label{sectioncurvatures}
This section will introduce some tools to deal with branch points in the Riemannian geometry. These tools will allow us to understand the behavior of curvatures near branch points using the principal and co-principal parts of the map in branch coordinates. The mathematical machinery constructed here is inspired by the previous works of the author \cite{Me1,Me2,Me3}, which is used to study frontals from the point of view of differential geometry. The reader will not require prior knowledge of these works.

Let $f:M^m\to N^n$ be a $C^{r+2,\beta}_{loc}$-frontal, $r\in\Na\cup\{0,\infty,\omega\}$ and $\beta \in[0,1]$. From now on, we will assume that $N^n$ is endowed with a $C^{r+1,\beta}_{loc}$-Riemannian metric $\tilde{g}$ unless specified otherwise. If $U\subset M^m$ is an open set and  $\UC=(\UC_1,\dots,\UC_k)\in \Gamma(f^*TN^n\vert_U)^k$, where each $\UC_j$ is of class $C^1$, we will write $\dt_{\XC}\UC=(\dt_{\XC}\UC_1,\dots,\dt_{\XC}\UC_k)$ for any $\XC\in \XF(U)$. Since $T_f$ is a $C^{r-1,\beta}_{loc}$-vector bundle, the vector bundle $\bot_f$ whose fiber $\bot_p f$ is the orthogonal complement of $T_{fp}$ is also a $C^{r-1,\beta}_{loc}$-vector bundle over $M^m$. We can define the second fundamental form $\alpha:\XF(M^m)\times \XF(M^m) \to  \Gamma(\bot_f)$ similarly as in section \ref{section-definition} by $\alpha(\XC,\YC)=\dt_{\XC}f_*\YC^\perp$, where $\XC,\YC\in \XF(M^m)$. Note that the difference now is that $\alpha$ is defined on $\XF(M^m)\times \XF(M^m)$ instead of $\XF(Reg(f))\times \XF(Reg(f))$.  
The shape operator $A_{\xi_p}$ of $f$ at $p\in Reg(f)$ with respect to $\xi_p\in \bot_p f$ is defined by  $g(A_{\xi_p}\XC_p,\YC_p)=\gt(\alpha(\XC_p,\YC_p),\xi_p)$ for all $\XC_p, \YC_p \in T_pM^m$. In this section, a metric tensor $g$ on a vector bundle $E$ of rank $n$ over $M^m$ will be a covariant $2$-tensor field being symmetric and positive semidefinite at each fiber of $E$. So, for a local frame field $(\XC_1,\dots,\XC_n)$, the determinant of the matrix with components $g_{ij}=g(\XC_{i},\XC_{j})$ could vanishes at points of $M^m$.   
\begin{definition}
	Let $E$ be a vector bundle of rank $n$ over $M^m$, $g$ a metric tensor on the vector bundle $E$ and $V\subset M^m$ an open set. If $k,s \in\amsmathbb{N}$, $\mathcal{U}\in \Gamma(E\rvert_V)^k$ and $\mathcal{V}\in \Gamma(E\rvert_V)^s$, we define the operation $\odot:\Gamma(E\rvert_V)^k \times \Gamma(E\rvert_V)^s \to C(V,\mathcal{M}_{k\times s}(\Re))$ by $(\mathcal{U}\odot \mathcal{V})_{ij}:=g(\UC_i,\VC_j)$.
\end{definition}
If $f:M^m\to N^n$ is a $C^1$-map, $\gt$ a Riemannian metric on $N^n$ and $g$ the induced metric tensor on $M^m$, we can consider the operation above respect to the metrics $g$ and $\gt$ on the vector bundles $TM^m$ and $f^*TN^n$ respectively. We will use the same symbol $\odot$ for the operation concerning these metrics. The following are some of the properties of this operation.
\begin{lemma}\label{basic2}
	Let $E$ be a vector bundle of rank $n$ over $M^m$, $g$ a metric tensor on the vector bundle $E$ and $V\subset M^m$ an open set. If $k,s \in\amsmathbb{N}$, $\mathcal{U}\in \Gamma(E\rvert_V)^k$ and $\mathcal{V}\in \Gamma(E\rvert_V)^s$, the following statements are true:
	\begin{enumerate}[label=(\roman*)]
		\item the operation $\odot$ is bilinear, where the spaces $\Gamma(E\rvert_V)^k$ and $\Gamma(E\rvert_V)^s$ are seen as $C(V,\Re)$-modules,
		\item $\UC\odot \VC=(\VC\odot \UC)^t$,
		\item if $\EC$ is an orthonormal frame field for $E\rvert_V$, then $\UC\odot \VC=(\EC\UC)^t(\EC\VC)$,
		\item $(\UC\A)\odot \VC=\A^t(\UC\odot \VC)$ and $\UC\odot (\VC\B)=(\UC\odot \VC)\B$ for all $\A\in C(V,\mathcal{M}_{k\times k'}(\Re))$, $\B\in C(V,\mathcal{M}_{s\times s'}(\Re))$,
		\item if $\EC$ is a frame field for $E\rvert_V$, then $\UC\odot \VC=(\EC\UC)^t(\EC\odot \EC)(\EC\VC)$,
		\item if $\EC$ is an orthonormal frame field for $E\rvert_V$, the components of $\UC$ are orthonormal if and only if the columns of $\EC\UC$ are orthonormal.  
	\end{enumerate}
\end{lemma}	
\begin{proof}
	Items $(i)$ and $(ii)$ follow immediately from the definition. To prove $(iii)$, if $\theta$ is the coframe field dual to $\EC$, we have that \begin{align*}
		(\mathcal{U}\odot \mathcal{V})_{ij}&=g(\UC_i,\VC_j)=g(\sum_{k'=1}^{n}\theta_{k'}(\UC_i)\EC_{k'},\sum_{s'=1}^{n}\theta_{s'}(\VC_j)\EC_{s'})=\sum_{k'=1}^{n}\theta_{k'}(\UC_i)\theta_{k'}(\VC_j)\\
		&=\sum_{k'=1}^{n}((\EC\UC)^t)_{ik'}(\EC\VC)_{k'j}=((\EC\UC)^t(\EC\VC))_{ij}.
	\end{align*}
	Item $(iv)$ easily follows from item $(iii)$ of this lemma and item $(viii)$ of lemma \ref{basic}. To prove $(v)$, using item $(iv)$ of lemma \ref{basic}, we have $\UC=\EC(\EC\UC)$ and $\VC=\EC(\EC\VC)$, then by item $(iv)$, $\UC\odot \VC=\EC(\EC\UC)\odot \EC(\EC\VC)=(\EC\UC)^t(\EC\odot \EC)(\EC\VC)$. Lastly, to prove $(vi)$, observe that the components of $\UC$ are orthonormal if and only if $\UC\odot \UC=\id_k$, then by item $(iii)$, the equivalence follows.   
\end{proof}

\begin{definition}
	\normalfont Let $F$ be a vector bundle over $M^m$, $V\subset M^m$ an open set, $k\in \Na$,  $\UC\in\Gamma(F\vert_V)^k$ and $\tilde{g}$ a metric tensor on $F$. We define the matrix-valued map $\I_\UC:=\UC\odot\UC$. In particular, if $f:M^m\to N^n$ is a $C^1$-map, $F=f^*TN^n$ and $(U,x)$ is a coordinate system of $M^m$, we just write $\I$ instead of $\I_{f_*\pdv{}{x}}=\I_{\pdv{}{x}}$. 
\end{definition}
\begin{remark}\label{desc}\normalfont
	If the components of $\UC$ are linearly independent and $\tilde{g}$ is a Riemannian metric on $F\vert_V$, considering local orthonormal frame fields $\EC$ for $F$ and by lemma \ref{basic2}, we have $\I_\UC=\U^t\U$ locally, where $\U=\EC\UC$. By lemma \ref{basic} $rank(\U)=k$ and since $\U^t\U$ is the Gram matrix of $\U$, then we have $det(\I_\UC)>0$.  
\end{remark}
\begin{definition}
	\normalfont Let $F$ be a smooth vector bundle over $M^m$, $\dt$ a Koszul connection on $F$, $\tilde{g}$ a metric tensor on $F$ and $E$ a vector subbundle of $F$. If $(V,x)$ is a coordinate system of $M^m$ and $\xi \in \Gamma(F\vert_V)$ is a $C^{1}$-section, we define $\tilde{\nabla}_{\pdv{}{x}}^E\xi:=(\pi_E(\tilde{\nabla}_{\pdv{}{x_1}}\xi),\dots, \pi_E(\tilde{\nabla}_{\pdv{}{x_m}}\xi))$, where $\pi_E:F\to E$ is the projection on $E$ respect the decomposition $F=E\bigoplus E^\bot$. We will write $\tilde{\nabla}_{\pdv{}{x}}\xi$ to denote $\tilde{\nabla}_{\pdv{}{x}}^F\xi=(\tilde{\nabla}_{\pdv{}{x_1}}\xi,\dots, \tilde{\nabla}_{\pdv{}{x_m}}\xi)$. Also, if $\mathscr{V}=(V,x,\WC)$ is a framed coordinate system of $E$, we define the matrix-valued maps $\mathfrak{J}_\VS^\xi:=\WC \tilde{\nabla}_{\pdv{}{x}}^E\xi\text{ and }\II_{\VS}^{\xi}:=-\WC \odot \tilde{\nabla}_{\pdv{}{x}}\xi$.
\end{definition}
\begin{remark}\label{remJO}\normalfont
	Observe that by item $(iv)$ of lemma \ref{basic}, we have $\tilde{\nabla}_{\pdv{}{x}}^E\xi=\WC \mathfrak{J}_\VS^\xi$ and by item $(iv)$ of lemma \ref{basic2}, it holds that $\WC\odot \tilde{\nabla}_{\pdv{}{x}}^E\xi=\I_{\WC}\mathfrak{J}_\VS^\xi$. Therefore, $\mathfrak{J}_\VS^\xi=-\I_{\WC}^{-1}\II_{\VS}^{\xi}$. On the other hand, if $\xi\in \Gamma(E^\bot\vert_V)$ and $\ell\in C(V,\Re)$, it is easy to verify that $\II_{\VS}^{\ell\xi+\xi'}=\ell\II_{\VS}^{\xi}+\II_{\VS}^{\xi'}$ and therefore $\mathfrak{J}_\VS^{\ell\xi+\xi'}=\ell\mathfrak{J}_\VS^\xi+\mathfrak{J}_\VS^{\xi'}$.
\end{remark}
Let $f:M^m\to N^n$ be a $C^1$-frontal, $(U,x)$ a coordinate system of $M^m$ and $\xi$ a $C^1$-section of $f^*TN^n\vert_U$, we simply write $\tilde{\nabla}_{\pdv{}{x}}^T\xi$ instead of $\tilde{\nabla}_{\pdv{}{x}}^{T_f}\xi$ and we define the matrix map $\II^\xi:=-f_*\dd{}{x}\odot\tilde{\nabla}_{\pdv{}{x}}\xi$. In the case that $f$ is a $C^2$-frontal, $p\in U$ and $\xi\in \Gamma(\bot_f\vert_U)$, we have that $\II^\xi(p)$ is the $m\times m$ symmetric matrix of the bilinear form $\gt(\alpha,\xi_p):T_pM^m\times T_pM^m\to \Re$ on the basis $(\dd{}{x_1}\vert_p,\dots,\dd{}{x_m}\vert_p)$. In fact, as $\gt(f_*\dd{}{x_j},\xi)=0$ on $U$, then $\gt(\alpha(\dd{}{x_i},\dd{}{x_j}),\xi)=\gt(\alpha(\dd{}{x_j},\dd{}{x_i}),\xi)=\gt(\dt_{\dd{}{x_j}}f_*\dd{}{x_i},\xi)=-\gt(f_*\dd{}{x_i},\dt_{\dd{}{x_j}}\xi)=(\II^\xi)_{ij}$.
On the other hand, observe that $\dd{}{x}$ is a frame field for $TM^m\vert_{U-\Sigma(f)}$, then defining $A_\xi\dd{}{x}:=(A_\xi\dd{}{x_1},\dots, A_\xi\dd{}{x_m})\in\XF(U-\Sigma(f))^m$, the matrix-valued map $\amsmathbb{A}_\xi:=\dd{}{x}A_\xi\dd{}{x}$ is well-defined on $U-\Sigma(f)$. We have the following equalities.
\begin{lemma}\label{decomposition}
	Let $f:M^m\to N^n$ be a $C^1$-frontal, $\VS=(U,x,\WC)$ a framed coordinate system of $T_f$ and $\xi$ a $C^1$-section of $f^*TN^n\vert_U$. It holds that $\I=\J{\VS}^t \I_{\WC}\J{\VS}$ and $\II^{\xi}=\J{\VS}^t\II_{\VS}^\xi$ on $U$. Also, if $f$ is a $C^2$-frontal and $\xi\in \Gamma(\bot_f\vert_U)$, we have $\amsmathbb{A}_\xi=\I^{-1}\II^\xi$ on $U-\Sigma(f)$.
\end{lemma}
\begin{proof}
	Since $f_*\dd{}{x}=\WC\J{\VS}$, then by item $(iv)$ of lemma \ref{basic2} and the definitions of $\I$, $\II^\xi$, the first part follows. For the last part, as $A_\xi\dd{}{x}=\dd{}{x}\amsmathbb{A}_\xi$, then $\dd{}{x}\odot A_\xi\dd{}{x}=\I\amsmathbb{A}_\xi$. By definition of $A_\xi$, we have that $\dd{}{x}\odot A_\xi\dd{}{x}$ at $p\in U$ is the $m\times m$ symmetric matrix of the bilinear form $\gt(\alpha,\xi_p):T_pM^m\times T_pM^m\to \Re$ on the basis $(\dd{}{x_1}\vert_p,\dots,\dd{}{x_m}\vert_p)$, which is equal to $\II^\xi(p)$ by the discussion above. So, we have $\II^\xi=\I\amsmathbb{A}_\xi$ on $U-\Sigma(f)$ and the result follows. 
\end{proof}
\begin{definition}
	Let $f:M^m\to N^n$ be a $C^1$-frontal and $\VS=(U,x,\WC)$ a framed coordinate system of $T_f$. We define the function $\lambda_\VS:=det(\J{\VS})$.
\end{definition}
Observe that by item $(v)$ of lemma \ref{basic}, $p\in \Sigma(f)\cap U$ if and only if $rank(\mathds{J}_{\VS}(p))<m$ and thus $\Sigma(f)\cap U=\lambda_{\VS}^{-1}(0)$.
\begin{corollary}\label{shapeO}
	Let $f:M^m\to N^n$ be a $C^2$-frontal, $\VS=(U,x,\WC)$ a framed coordinate system of $T_f$ and $\xi$ a $C^1$-section of $\bot_f\vert_U$. It holds that $-adj(\J{\VS})\mathfrak{J}_\VS^\xi=\lambda_{\VS}\amsmathbb{A}_\xi$ on $U-\Sigma(f)$.
\end{corollary}
\begin{proof}
	By remark \ref{remJO} and lemma \ref{decomposition}, on $U-\Sigma(f)$ we have $$-adj(\J{\VS})\mathfrak{J}_\VS^\xi=adj(\J{\VS})\I_{\WC}^{-1}\II_{\VS}^{\xi}=\lambda_{\VS}\J{\VS}^{-1}\I_\WC^{-1}(\J{\VS}^{-1})^t\J{\VS}^{t}\II_{\VS}^\xi=\lambda_{\VS}\I^{-1}\II^{\xi}=\lambda_{\VS}\amsmathbb{A}_\xi.$$
\end{proof}
Observe that $-adj(\J{\VS})\mathfrak{J}_\VS^\xi$ is well-defined on $U$, even on the singularities $\Sigma(f)\cap U$. By the corollary above, this map extends $\lambda_{\VS}\amsmathbb{A}_\xi$. So, we can use it to define functions on $U$ that imitate the principal curvatures for $\xi$. We define the {\it relative principal curvatures for $\xi$ in $\VS$} to be the eigenvalues $k_{1\VS}^\xi, \dots,k_{m\VS}^\xi$ of $-adj(\J{\VS})\mathfrak{J}_\VS^\xi$. We can define several functions in terms of $k_{1\VS}^\xi, \dots,k_{m\VS}^\xi$ similar to the classic symmetric curvatures. Because the ordering of the functions $k_{1\VS}^\xi, \dots,k_{m\VS}^\xi$ is arbitrary, we will use the elementary symmetric polynomials $\sigma_k(t_1,\dots,t_m)=\sum_{1\leq j_1<j_2<\dots <j_k\leq m}t_{j_1}\cdots t_{j_k}$ to have invariance under permutations as in the classic case (see for example \cite{Spvol4}). We define the {\it relative symmetric curvatures $K_{1\VS}^\xi,\dots,K_{m\VS}^\xi$ for $\xi$ in $\VS$} by $\binom{m}{j}K_{j\VS}^\xi=\sigma_j(k_{1\VS}^\xi,\dots,k_{m\VS}^\xi),$ in particular the function $H_\VS^\xi:=K_{1\VS}^\xi=\frac{1}{m}\sum_{j=1}^{m}k_{j\VS}$ will be called the {\it relative mean curvature of $f$ for $\xi$ in $\VS$}. Also, we will call the function $K_\VS^{\xi}:=det(\mathfrak{J}_\VS^{\xi})$ the {\it relative Gauss-Kronecker curvature of $f$ for $\xi$ in $\VS$}. On $U-\Sigma(f)$, we will denote the classical symmetric curvatures for $\xi$ by $K_1^\xi,\dots,K_m^\xi$. We write $H^\xi=K_1^\xi$ and $K^\xi=K_m^\xi$.
\begin{remark}\label{remHO}\normalfont
	Observe that $K^\xi=det(A_\xi)$ and $H^\xi=\frac{1}{m}tr(A_\xi)=\gt(\HC,\xi)$ on $U-\Sigma(f)$, where $\HC$ is the mean curvature vector of $f$. On the other hand, by remark \ref{remJO} and since $H_\VS^\xi=-\frac{1}{m}tr(adj(\J{\VS})\mathfrak{J}_\VS^{\xi})$, if $\xi'\in \Gamma(\bot f\vert_U)$ and $\ell\in C(U,\Re)$, we have $H_\VS^{\ell\xi+\xi'}=\ell H_\VS^{\xi}+H_\VS^{\xi'}$, and therefore it induces a section $\HC_{\VS}^*$ of the dual bundle $(\bot f\vert_U)^*$ of $\bot f\vert_U$ such that $\HC_\VS^*(\xi)=H_\VS^\xi$ on $U$ for all $\xi\in\Gamma(\bot f\vert_U)$. Since the Riemannian metric $\gt$ induces a bundle isomorphism between $\bot f\vert_U$ and $(\bot f\vert_U)^*$, then there exists $\HC_\VS\in\Gamma(\bot_f\vert_U)$ such that $\gt(\HC_\VS,\xi)=\HC_\VS^*(\xi)=H_\VS^\xi$ for all $\xi\in\Gamma(\bot f\vert_U)$. We will call $\HC_{\VS}$ the {\it relative mean curvature vector of $f$ in $\VS$}.
\end{remark}
\begin{lemma}\label{lemrelativecurvatures}
	Let $f:M^m\to N^n$ be a $C^2$-frontal, $\VS=(U,x,\WC)$ a framed coordinate system of $T_f$ and $\xi$ a $C^1$-section of $\bot_f\vert_U$. On $U-\Sigma(f)$, we have that $K_{j\VS}^\xi=(\lambda_{\VS})^jK_{j}^\xi$ for each $j\in\{1,\dots,m\}$, $K_\VS^\xi=(-1)^m\lambda_{\VS}K^\xi$ and $\HC_\VS=\lambda_{\VS}\HC$.   
\end{lemma}
\begin{proof}
	By corollary \ref{shapeO}, we have that at each point of $U-\Sigma(f)$ the set of eigenvalues of $-adj(\J{\VS})\mathfrak{J}_\VS^\xi$ is $\{\lambda_{\VS}k_1,\dots,\lambda_{\VS}k_m\}$, where $\{k_1,\dots,k_m\}$ are the eigenvalues of $\amsmathbb{A}_\xi$. Then, $\binom{m}{j}K_{j\VS}^\xi=\sigma_j(\lambda_{\VS}k_{1},\dots,\lambda_{\VS}k_{m})=(\lambda_{\VS})^j\sigma_j(k_{1},\dots,k_{m})=(\lambda_{\VS})^j\binom{m}{j}K_{j}^\xi$ and thus $K_{j\VS}^\xi=(\lambda_{\VS})^jK_{j}^\xi$ for each $j\in\{1,\dots,m\}$. As $K_{m\VS}^\xi=(\lambda_{\VS})^mK^\xi$ and $K_{m\VS}^\xi=det(-adj(\J{\VS})\mathfrak{J}_\VS^\xi)=(-1)^m\lambda_{\VS}^{m-1}K_\VS^\xi$ on $U-\Sigma(f)$, we have $K_\VS^\xi=(-1)^m\lambda_{\VS}K^\xi$. For the last equality, let $\tilde{\xi}=(\tilde{\xi}_1,\dots,\tilde{\xi}_{n-m})$ be a local frame field for $\bot_f\vert_{U-\Sigma(f)}$ over a neighborhood of $p\in U-\Sigma(f)$. Since $\HC_{\VS}=\tilde{\xi}\tilde{\xi}\HC_{\VS}$ and $\HC=\tilde{\xi}\tilde{\xi}\HC$, then by lemma \ref{basic2}, we have $\tilde{\xi}\odot \HC_{\VS}=\I_{\tilde{\xi}}\tilde{\xi}\HC_{\VS}$ and $\tilde{\xi}\odot \HC=\I_{\tilde{\xi}}\tilde{\xi}\HC$. On the other hand, by remark \ref{remHO} and the first equality stated in this lemma, $\tilde{\xi}\odot \HC_{\VS}=\lambda_{\VS}\tilde{\xi}\odot \HC$ on $U-\Sigma(f)$. Hence, $\tilde{\xi}\HC_{\VS}=\I_{\tilde{\xi}}^{-1}(\tilde{\xi}\odot \HC_{\VS})=\lambda_{\VS}\I_{\tilde{\xi}}^{-1}\tilde{\xi}\odot \HC=\lambda_{\VS}\tilde{\xi}\HC$ and by item $(iv)$ of lemma \ref{basic}, we can conclude that $\HC_{\VS}(p)=\lambda_{\VS}(p)\HC(p)$. As $p$ is an arbitrary point of $U-\Sigma(f)$ the last equality holds on $U-\Sigma(f)$.         
\end{proof}
In the context of lemma \ref{lemrelativecurvatures}, observe that the maps $\J{\VS}$, $\II_{\VS}^{\xi}$, $\mathfrak{J}_\VS^{\xi}$, $H_\VS^\xi$ and $K_\VS^\xi$ are of class $C^\infty$ on $U$ if $f$ is a $C^\infty$-frontal and $\xi$ is a smooth section. In the following lemma, we will use these maps to compute the classic curvatures near singularities and describe the behavior under specific hypotheses. These hypotheses will appear in the context of theorem \ref{TheoE} later.
\begin{lemma}\label{lemsectionalGK}
	Let $f:M^2\to N^n$ be a $C^\infty$-frontal, $r\in\Na\cup\{0,\infty,\omega\}$, $\beta \in [0,1]$, $\gt$ a smooth Riemannian metric on $N^n$, $\VS=(U,x,\WC)$ a framed coordinate system of $T_f$ and $\xi=(\xi_1,\dots,\xi_{n-2})$ a frame field of class $C^\infty$ for $\bot_f\vert_U$. Let suppose that for each $j\in \{1,\dots,n-2\}$, $\II_{\VS}^{\xi_j}=\F_j\M$ for some $\M,\F_j\in C^{r,\beta}_{loc}(U,\mathcal{M}_{2\times 2}(\Re))$. Then, the following statements are true:
	\begin{enumerate}[label=(\roman*)]
		\item If $det(\M)/\lambda_{\VS}:Reg(f)\cap U\to \Re$ has a $C^{r,\beta}_{loc}$-extension on $U$, then the sectional curvature and the Gauss-Kronecker curvature $K^{\xi'}$ on $Reg(f)\cap U$ for any $C^{\infty}$-section $\xi'$ of $\bot_f\vert_U$ have $C^{r,\beta}_{loc}$-extensions on $U$.
		\item If $p\in \Sigma(f)$, $\sum_{j=1}^{n-2}det(\F_j)(p)>0$, $\lim_{q\to p}det(\M(q))/\lambda_{\VS}(q)=\infty$ and $\xi\D$ is orthonormal for some $\D\in C^{\infty}(U,GL(n-2))$ with $\D(p)=\id_{n-2\times n-2}$, then $\lim_{q\to p}Sec(\dd{}{x_1}(q),\dd{}{x_2}(q))=\lim_{q\to p}\sum_{j=1}^{n-2}K^{\xi'_j}(q)=\infty$ for any orthonormal frame field $\xi'$ of class $C^{\infty}$ for $\bot_f\vert_U$.
	\end{enumerate} 
\end{lemma}
\begin{proof}
	Let us prove the item $(i)$. If $\tilde{\xi}$ is a $C^{\infty}$-section for $\bot_f\vert_U$, there exist $c_1,\dots, c_{n-2}\in C^\infty(U,\Re)$ such that $\tilde{\xi}=\sum_{j=1}^{n-2}c_j\xi_j$. By remark \ref{remJO}, we have $\II_{\VS}^{\tilde{\xi}}=\sum_{j=1}^{n-2}c_j\II_{\VS}^{\xi_j}$ and since $det(\A+\B)=det(\A)+det(\B)+tr(adj(\A)\B)$ for any $\A,\B \in \mathcal{M}_{2\times 2}(\Re)$, then $det(\II_{\VS}^{\tilde{\xi}})=\sum_{j=1}^{n-2}c_j^2det(\II_{\VS}^{\xi_j})+\sum_{1\leq i<j\leq n-2}c_ic_jtr(adj(\II_{\VS}^{\xi_i})\II_{\VS}^{\xi_j})$. As $tr(adj(\II_{\VS}^{\xi_i})\II_{\VS}^{\xi_j})=tr(adj(\F_i)\F_j\M adj(\M))$ then $det(\II_{\VS}^{\tilde{\xi}})=det(\M)\rho$ for some $\rho\in C^{r,\beta}_{loc}(U,\Re)$. Recalling that $K_\VS^{\tilde{\xi}}=det(\mathfrak{J}_\VS^{\tilde{\xi}})=det(-\I_{\WC}^{-1}\II_{\VS}^{\tilde{\xi}})=det(\I_{\WC}^{-1})\rho det(\M)$ (see remark \ref{remJO}), then by lemma \ref{lemrelativecurvatures} $K^{\tilde{\xi}}=K_{\VS}^{\tilde{\xi}}/\lambda_{\VS}=det(\I_{\WC}^{-1})\rho det(\M)/\lambda_{\VS} \in C^{r,\beta}_{loc}(U,\Re)$. Thus, the Gauss-Kronecker curvature $K^{\tilde{\xi}}$ has a $C^{r,\beta}_{loc}$-extension on $U$ for any smooth section $\tilde{\xi}$ of $\bot_f\vert_U$. On the other hand, by the Gauss equation (cf.\cite{toj}) $$Sec(\XC,\YC)=\tilde{S}ec(f_*\XC,f_*\YC)+\frac{\gt(\alpha(\XC,\XC),\alpha(\YC,\YC))-||\alpha(\XC,\YC)||^2}{det((f_*\XC,f_*\YC)\odot (f_*\XC,f_*\YC))}$$ for any $\XC,\YC\in \XF(Reg(f)\cap U)$ linearly independent, where $||\cdot||$ is the norm induced by $\gt$ and $\tilde{S}ec$ is the sectional curvature of $N^n$. Taking $(\XC,\YC)=(\dd{}{x_1},\dd{}{x_2})$ on $Reg(f)\cap U$ and $\xi'$ an orthonormal frame field of class $C^{\infty}$ for $\bot_f\vert_U$, we have
	$$Sec(\dd{}{x})=\tilde{S}ec(f_*\dd{}{x})+\frac{\gt(\alpha(\dd{}{x_1},\dd{}{x_1}),\alpha(\dd{}{x_2},\dd{}{x_2}))-||\alpha(\dd{}{x_1},\dd{}{x_2})||^2}{det(\I)}$$ on $Reg(f)\cap U$. Since $\alpha(\dd{}{x_i},\dd{}{x_k})=\sum_{j=1}^{n-2}(\II^{\xi'_j})_{ik}\xi'_j$ for $i,k\in \{1,2\}$, we have $\gt(\alpha(\dd{}{x_1},\dd{}{x_1}),\alpha(\dd{}{x_2},\dd{}{x_2}))-||\alpha(\dd{}{x_1},\dd{}{x_2})||^2=\sum_{j=1}^{n-2}det(\II^{\xi'_j})$ and thus $Sec(\dd{}{x})=\tilde{S}ec(f_*\dd{}{x})+(\sum_{j=1}^{n-2}det(\II^{\xi'_j}))det(\I)^{-1}$. As $f_*\dd{}{x}=\WC\J{\VS}$, then the vector fields $f_*\dd{}{x_1}, f_*\dd{}{x_2}$ generate the same plane as $\WC_1, \WC_2$ on $Reg(f)\cap U$, therefore $\tilde{S}ec(f_*\dd{}{x})=\tilde{S}ec(\WC)$. Recalling that $K^{\xi'_j}=det(\II^{\xi'_j})/det(\I)$ we have \begin{align}\label{eqgauss2}
		Sec(\dd{}{x_1},\dd{}{x_2})=\tilde{S}ec(\WC)+\sum_{j=1}^{n-2}K^{\xi'_j}\text{ on $Reg(f)\cap U$.}
	\end{align}
	As $\sum_{j=1}^{n-2}K^{\xi'_j}$ has a $C^{r,\beta}_{loc}$-extension on $U$, the equality (\ref{eqgauss2}) implies that $Sec(\dd{}{x_1},\dd{}{x_2})$ has a $C^{r,\beta}_{loc}$-extension on $U$.
	
	Now, let us prove item $(ii)$. Writing $c_{kj}$ the $(k,j)$-component of $\D$ and setting $\xi'=\xi\D$, we have seen that $$det(\II_{\VS}^{\xi'_j})=\sum_{k=1}^{n-2}c_{kj}^2det(\II_{\VS}^{\xi_k})+\sum_{1\leq k_1<k_2\leq n-2}c_{k_1j}c_{k_2j}tr(adj(\II_{\VS}^{\xi_i})\II_{\VS}^{\xi_j})=\rho_jdet(\M),$$
	for some $\rho_j\in C^{r,\beta}_{loc}(U,\Re)$ for each $j\in\{1,\dots,n-2\}$. As $\D(p)=\id_{n-2\times n-2}$, we have that $\rho_j(p)=det(\F_j(p))$. Therefore, $\sum_{j=1}^{n-2}K_\VS^{\xi'_j}=\sum_{j=1}^{n-2}det(-\I_{\WC}^{-1}\II_{\VS}^{\xi'_j})=det(\I_{\WC}^{-1})det(\M)\sum_{j=1}^{n-2}\rho _j$. Since $$\sum_{j=1}^{n-2}\rho _j(p)=\sum_{j=1}^{n-2}det(\F_j(p))>0\text{ and }\lim_{q\to p}det(\M(q))/\lambda_{\VS}(q)=\infty,$$ by lemma \ref{lemrelativecurvatures} we have $\lim_{q\to p}\sum_{j=1}^{n-2}K^{\xi'_j}(q)=\lim_{q\to p}\sum_{j=1}^{n-2}K_\VS^{\xi'_j}(q)/\lambda_{\VS}(q)=\infty$. From the equality (\ref{eqgauss2}), it follows that $\lim_{q\to p}Sec(\dd{}{x_1}(q),\dd{}{x_2}(q))=\infty$.     
\end{proof}
\begin{lemma}\label{lemfundamentalparts}
	Let $f:M^m\to N^n$ be a $C^{2}$-frontal, $\gt$ a smooth Riemannian metric on $N^n$ and $(U,x)$, $(V,y)$ coordinate systems of $M^m$, $N^n$ respectively, such that $\hat{f}$ is a normalized frontal. If $A$, $B$ are the principal part and co-principal part of $\hat{f}$ respectively, let us set $\WC=\dd{}{y}\circ f\begin{pmatrix}
		\id_m\\
		B\circ x
	\end{pmatrix}$ and $\xi=(\dd{}{y}\circ f) \I_{\dd{}{y}\circ f}^{-1}\begin{pmatrix}
	-B^t\circ x\\
	\id_{n-m}
	\end{pmatrix}$. Then, it holds that $\VS=(U,x,\WC)$ is a frame coordinate system of $T_f$ such that $\J{\VS}=\J{A}\circ x$ and $\xi$ is a frame field for $\bot_f\vert_U$ such that $(\xi\odot \xi)(x^{-1}(0))=\id_{n-m\times n-m}$. Additionally, if $y$ is an isometry then $\II_{\VS}^{\xi_k}=\J{B_k}\circ x$, where $B_k$ is the map given by the $k$-row of $B$.  
\end{lemma}
\begin{proof}
We have $\J{\hat{f}}=\begin{pmatrix}
	\id_m\\
	B
\end{pmatrix}\J{A}$, where $A$, $B$ are the principal and co-principal parts of $\hat{f}$. Since $(U,x,\dd{}{y}\circ f)$ is a framed coordinate system of $f^*TN^n$ and $(\dd{}{y}\circ f)f_*\dd{}{x}=\J{f}=\begin{pmatrix}
	\id_m\\
	B\circ x
\end{pmatrix}(\J{A}\circ x)$, then we have that $\VS=(U,x,\WC)$ is a frame coordinate system of $T_f$ (see remark \ref{remarkframeinduced}) such that $\J{\VS}=\J{A}\circ x$. Also, by item $(iv)$ of lemma \ref{basic2} we have $\WC\odot\xi=\begin{pmatrix}
	\id_m&B^t\circ x
\end{pmatrix}\begin{pmatrix}
	-B^t\circ x\\
	\id_{n-m}
\end{pmatrix}=-B^t\circ x+B^t\circ x=0$, which means that $\gt(\WC_j,\xi_k)=0$ for all $j\in\{1,\dots,m\}$, $k\in\{1,\dots,n-m\}$, that is, $\xi$ is a frame field for $\bot_f\vert_U$. As $B(0)=0$, we have $(\xi\odot\xi)(x^{-1}(0))=\id_{n-m\times n-m}$. Now, as $\gt(\WC_j,\xi_k)=0$ then $-\gt(\WC_j,\dt_{\dd{}{x_i}}\xi_k)=\gt(\dt_{\dd{}{x_i}}\WC_j,\xi_k)$ for $i\in\{1,\dots,m\}$. Therefore, the $i$-column of $\II_{\VS}^{\xi_k}=-\WC\odot\dt_{\dd{}{x}}\xi_k$ is equal to $\dt_{\dd{}{x_i}}\WC\odot\xi_k$, which is the $k$-column of $\dt_{\dd{}{x_i}}\WC\odot\xi=(\dt_{\dd{}{x_i}}\dd{}{y}\circ f\begin{pmatrix}
	\id_m\\
	B\circ x
\end{pmatrix}+\dd{}{y}\circ f\begin{pmatrix}
	0\\
	\dd{}{x_i}(B\circ x)
\end{pmatrix})\odot\xi.$ If $y$ is an isometry, the Christoffel symbols of $\dt$ in $y$ are $0$ on $V$ and therefore $\dt_{\dd{}{x_i}}\dd{}{y}\circ f=0$. Thus, 
\begin{align*}
	\dt_{\dd{}{x_i}}\WC\odot\xi=\dd{}{y}\circ f\begin{pmatrix}
		0\\
		\dd{}{x_i}(B\circ x)
	\end{pmatrix}\odot\xi&=\begin{pmatrix}
		0&\dd{}{x_i}(B^t\circ x)
	\end{pmatrix}\begin{pmatrix}
		-B^t\circ x\\
		\id_{n-m}
	\end{pmatrix}\\
	&=\dd{}{x_i}(B^t\circ x),
\end{align*} which implies that the $k$-column of $\dt_{\dd{}{x_i}}\WC\odot\xi$ is $\dd{}{x_i}(B_k^t\circ x)$, where $B_k$ is the $k$-row of $B$. We conclude that $\II_{\VS}^{\xi_k}=\J{B_k}\circ x$. 
\end{proof}
\begin{theorem}\label{Theofinal}
	Let $f:M^2\to N^n$ be a $C^{\infty}$-map, $(N^n,\gt)$ a smooth Riemannian manifold, $p\in M^2$ a branch point of order $s$ and $(U,x)$, $(V,y)$ $C^{\infty}$-regular branch coordinates at $p$. If $y$ is a conformal diffeomorphism and $\iota$ the index of $p$, the following statements hold:
	\begin{enumerate}[label=(\roman*)]
		\item If $\iota\geq2s+1$, then the sectional curvature $Sec$ of the induced metric on $U-\{p\}$, the mean curvature vector $\HC\vert_{U-\{p\}}$ and the Gauss-Kronecker curvature $K^\xi\vert_{U-\{p\}}$ have $C^{0,1}_{loc}$-extensions on $U$, where $\xi$ is any local smooth section of the normal bundle of $f$ over $U$. 
		\item If $\iota<2s+1$, then $\lim_{q\to p}Sec(q)=\infty$ and $\lim_{q\to p}\sum_{j=1}^{n-2}K^{\xi'_j}=\infty$ for any smooth orthonormal frame field $(\xi'_1,\dots,\xi'_{n-2})$ for the normal bundle of $f$ over $U$.    
	\end{enumerate}
\end{theorem}
\begin{proof}
	First, let us suppose that $y$ is an isometry. By proposition \ref{branchfrontal} $\hat{f}$ is a normalized frontal and by lemma \ref{lemfundamentalparts} the principal part $A$ and the co-principal part $B$ induce a framed coordinate system $\VS=(U,x,\WC)$ of $T_f$ and a frame field $\xi$ for $\bot_f\vert_U$, such that $\J{\VS}=\J{A}\circ x$, $(\xi\odot \xi)(p)=\id_{n-2\times n-2}$ and $\II_{\VS}^{\xi_k}=\J{B_k}\circ x$ for each $k\in\{1,\dots,n-2\}$. If $\iota\geq2s+1$ and $b$ is the complex co-principal part of $\hat{f}$, then by item $(ii)$ lemma \ref{lemcoprincipalpart} $\dd{b}{z}=z^{s}e(z)$ and $\dd{b}{\z}=\z^{s}c(z)$, for some $e,c\in C^{0,1}_{loc}(x(U),\Co^{n-2})$. Thus, $\dd{B_{k1}}{z}(z)=\dd{b_k}{z}(z)+\dd{\bar{b}_k}{z}(z)=z^s(e_k(z)+\bar{c}_k(z))/2$ and $\dd{B_{k2}}{z}(z)=-i(\dd{b_k}{z}(z)-\dd{\bar{b}_k}{z}(z))=z^si(-e_k(z)+\bar{c}_k(z))/2$, which implies that $\J{B_k}=\F_k\M$, where 
	\begin{align}\label{eqmatrices}		
		\F_k=\begin{pmatrix}
			Re\{e_k+\bar{c}_k\}&-Im\{e_k+\bar{c}_k\}\\
			Re\{-ie_k+i\bar{c}_k\}&-Im\{-ie_k+i\bar{c}_k\}
		\end{pmatrix},\M=\begin{pmatrix}
			Re\{z^s\}&-Im\{z^s\}\\
			Im\{z^s\}&Re\{z^s\}
		\end{pmatrix}.
	\end{align}
	On the other hand, by condition (\ref{cond1}) $\J{A}=(s+1)\D'\M$, where $\D'$ is the matrix in the determinant of the condition (\ref{frontalcond}), so $det(\D')\neq 0$ on $x(U)$. Then, the function $det(\M\circ x)/\lambda_{\VS}=det(\M\circ x)/(det(\M\circ x) det((s+1)\D'\circ x))=1/det((s+1)\D'\circ x)$ has a $C^{\infty}$-extension on $U$. Since $\II_{\VS}^{\xi_k}=\J{B_k}\circ x=(\F_k\circ x)(\M\circ x)$, the item $(i)$ of lemma \ref{lemsectionalGK} implies that the sectional curvature and the Gauss-Kronecker curvature respect to any smooth section of $\bot_f\vert_U$ have $C^{0,1}_{loc}$-extensions on $U$. Also, as $H_\VS^{\xi_k}=-\frac{1}{2}tr(adj(\J{\VS})\mathfrak{J}_\VS^{\xi_k})$ and $\mathfrak{J}_\VS^{\xi_k}=-\I_\WC^{-1}\II_{\VS}^{\xi_k}$, then $H_\VS^{\xi_k}=\frac{(s+1)}{2}tr(adj(\M\circ x)adj(\D'\circ x)\I_\WC^{-1}(\F_k\circ x)(\M\circ x))$ and by the cyclic property of the trace, we can write $H_\VS^{\xi_k}=det(M\circ x)\rho_k$, where $\rho_k \in C^{0,1}_{loc}(U,\Re)$. By lemma \ref{lemrelativecurvatures}, $H^{\xi_k}=H^{\xi_k}_{\VS}/\lambda_{\VS}=\rho_k det(\M\circ x)/\lambda_{\VS}$ and thus $H^{\xi_k}$ has a $C^{0,1}_{loc}$-extension on $U$ for each $k\in\{1,\dots,n-2\}$. Since $H^{\xi_k}=\gt(\HC,\xi_k)$ for each $k\in\{1,\dots,n-2\}$, then $\HC$ has a $C^{0,1}_{loc}$-extension on $U$ and we have completed the proof of item $(i)$. 
	
	To prove item $(ii)$, if $\iota<2s+1$, then by item $(i)$ of lemma \ref{lemcoprincipalpart} $\dd{b}{z}=z^{\iota-s-1}e(z)$ and $\dd{b}{\z}=\z^{\iota-s-1}c(z)$, for some $e,c\in C^{0,1}_{loc}(x(U),\Co^{n-2})$ with $c(0)=0$ and $e(0)\neq 0$. Similarly as before, we have $\II_{\VS}^{\xi_k}=(\F_k\circ x)(\M'\circ x)$, where $\F_k$, $\M'$ have the same form of the matrices in (\ref{eqmatrices}), exchanging $s$ by $\iota-s-1$. Note that $\sum_{k=1}^{n-2}det(\F_k)(x(p))=\sum_{k=1}^{n-2}|e_k(0)|^2>0$ and $\lim_{q\to p}(det(\M'\circ x)/\lambda_{\VS})(q)=\lim_{q\to p}(det(\M'\circ x)/(det(\M\circ x) det((s+1)\D'\circ x)))(q)=\infty$. Also, since $(\xi\odot\xi)(p)=\id_{n-2\times n-2}$, the frame $\xi(p)$ is orthonormal and by Gram-Schmidt process, there exists $\D\in C^{\infty}(U,GL(n-2))$ with $\D(p)=\id_{n-2\times n-2}$ such that $\xi\D$ is an orthonormal frame field for $\bot_f\vert_U$. By item $(ii)$ of lemma \ref{lemsectionalGK}, item $(ii)$ of the theorem follows.  
	
	Now, if $y$ is a conformal diffeomorphism, then there exists $\varphi\in C^{\infty}(V,\Re)$ such that $V$ endowed with the metric $\gt^1=e^{2\varphi}\gt$ turns $y:V\to y(V)$ into an isometry, so items $(i)$ and $(ii)$ are true with this metric. By the formulas of conformal changes of metric (cf.\cite{Bes} p.59), $Scal^1=e^{-2\varphi}(Scal-2\Delta\varphi)$, where $\Delta$ is the Laplacian operator respect to $\gt$. Since $Scal=2Sec$ and $Scal^1=2Sec^1$, we have $Sec^1=e^{-2\varphi}(Sec-\Delta\varphi)$ and by equation (\ref{eqgauss2}), item $(ii)$ follows. Also, by the formulas of conformal changes $\alpha^1(\XC,\YC)=\alpha(\XC,\YC)-\gt(f_*\XC,f_*\YC)\nabla\varphi^\bot$, where $\nabla\varphi$ is the gradient of $\varphi$.  Then, for any smooth section $\xi$ of $\bot_f\vert_U$, we have $\II^{1\xi}=e^{2\varphi}\II^\xi-\I \gt^1(\nabla\varphi^\bot,\xi)$ and since $\I^1=e^{2\varphi}\I$, this implies that $\A_\xi^1=\A_\xi-\gt(\nabla\varphi^\bot,\xi)\id_2$ (see lemma \ref{decomposition}). Therefore, $H^\xi=H^{1\xi}+\gt(\nabla\varphi^\bot,\xi)/2$ and $K^{1\xi}=K^\xi+\gt(\nabla\varphi^\bot,\xi)^2-\gt(\nabla\varphi^\bot,\xi)2H^\xi$. If $\iota\geq2s+1$, since $y:V\to y(V)$ is an isometry when $V$ is endowed with $\gt^1$, then as item $(i)$ is true with this metric, we have that $Sec^1$, $H^{1\xi}$ and $K^{1\xi}$ have $C^{0,1}_{loc}$-extensions and by the previous formulas $Sec$, $H^{\xi}$ and $K^{\xi}$ have also $C^{0,1}_{loc}$-extensions. It follows item $(i)$.
	
\end{proof}

\bibliographystyle{siam}

\bibliography{references}

\end{document}